\documentclass[a4paper, 11pt]{article}
\usepackage{anysize}
\marginsize{3,5cm}{2,5cm}{2,5cm}{2,5cm}
\usepackage{amssymb}
\usepackage{amsmath,amsbsy,amssymb,amscd}
\usepackage{t1enc}
\usepackage[cp1250]{inputenc}
\usepackage[all]{xy}
\usepackage{color}
\usepackage{amsfonts}
\usepackage{latexsym}
\usepackage{amsthm}
\usepackage{mathrsfs}
\usepackage{hyperref}
\usepackage{graphicx}
\usepackage{comment}

\usepackage{color}
\usepackage{caption}
\usepackage{subcaption}
\usepackage{enumerate}
\usepackage{xcolor}
\usepackage[draft]{todonotes}
\usepackage{float}
\usepackage{amsmath}
\usepackage{amsfonts}
\usepackage{amssymb}
\usepackage{amsthm}
\usepackage{mathtools}
\usepackage{graphicx}
\usepackage{fancyhdr}
\usepackage{geometry}
\usepackage{hyperref}
\usepackage{mathrsfs}
\usepackage{longtable}
\usepackage[nottoc,numbib]{tocbibind}
\usepackage{centernot}
\usepackage{color}
\usepackage{verbatim}
\usetikzlibrary{arrows, decorations.markings, positioning}
\usepackage{lipsum}
\usepackage{enumitem}
\usetikzlibrary{decorations.pathreplacing}

\allowdisplaybreaks

\makeatletter
\def\blfootnote{\xdef\@thefnmark{}\@footnotetext}
\makeatother

\makeatletter
\DeclareRobustCommand{\format@sec@number}[2]{{\normalfont\upshape#1}#2}

\makeatother

\definecolor{orange}{rgb}{1,0.5,0}

\usepackage{mathrsfs} 

\DeclareMathAlphabet{\mathpzc}{OT1}{pzc}{L}{it} 

\def\a{\alpha}

\numberwithin{equation}{section}

\theoremstyle{definition}
\newtheorem{definition}{Definition}[section]
\newtheorem{theorem}[definition]{Theorem}

\newtheorem{proposition}[definition]{Proposition}

\newtheorem{lemma}[definition]{Lemma}

\newtheorem{remark}[definition]{Remark}

\newtheorem{claim}[definition]{Claim}

\def\geq{\geqslant}
\def\leq{\leqslant}
\def\R{\mathbb{R}}

\def\Z{\mathbb{Z}}
\def\N{\mathbb{N}}

\def\epsilon{\varepsilon}

\def\a{\alpha}

\def\R{\mathbb R}
\def\N{{\mathbb N}}

\def\Z{\mathbb Z}

\def\S{\mathbb S}

\def\({\biggl(}
\def\){\biggr)}
\def\<{\mathbf{\langle}}
\def\>{\mathbf{\rangle}}

\def\uent{\overline{\text{ent}}}
\def\lent{\underline{\text{ent}}}
\def\ent{\text{ent}}

\newcommand{\Meng}[2]{\left\{#1\mathrel{}\middle|\mathrel{}#2\right\}}

\newcommand{\bea}{\begin{eqnarray}}
  \newcommand{\eea}{\end{eqnarray}}
  \newcommand{\beab}{\begin{eqnarray*}}
  \newcommand{\eeab}{\end{eqnarray*}}
\renewcommand{\a}{\alpha}
  \newcommand{\be}{\begin{equation}}
  \newcommand{\ee}{\end{equation}}

\newcommand{\abs}[1]{\left| #1 \right|}

\newcommand{\freq}{\operatorname{freq}}
\newcommand{\sh}{\operatorname{sh}}

\title{Slow entropy of some combinatorial constructions
}
\author{Shilpak Banerjee\footnote{Department of Mathematics, Indraprastha Institute of Information Technology Delhi (IIIT-Delhi), Okhla Industrial Estate, Phase-III, New Delhi, 110020, India, E-mail: banerjee.shilpak@gmail.com} \and Philipp Kunde\footnote{Department of Mathematics, The Pennsylvania State University, University Park, PA 16802, USA, E-mail: pkunde.math@gmail.com P.K. acknowledges financial support from a DFG Forschungsstipendium under Grant No. 405305501.} \and Daren Wei\footnote{Einstein Institute of Mathematics, The Hebrew University of Jerusalem, Givat Ram. Jerusalem, 9190401, Israel, E-mail: Daren.Wei@mail.huji.ac.il  D.W. was partially supported by the NSF grant DMS-16-02409 and ERC 2020 grant HomDyn (grant no. 833423).}}

\begin{document}
\maketitle

\begin{abstract}
Measure-theoretic slow entropy is a more refined invariant than the classical measure-theoretic entropy to characterize the complexity of dynamical systems with subexponential growth rates of distinguishable orbit types. In this paper we prove flexibility results for the values of upper and lower polynomial slow entropy of rigid transformations as well as maps admitting a good cyclic approximation. Moreover, we show that there cannot exist a general upper bound on the lower measure-theoretic slow entropy for systems of finite rank.

\end{abstract}

\section{Introduction}
Dating back to the foundational paper of J.\ von Neumann \cite{Ne} the classification of measure-preserving transformations (MPT's) up to isomorphism is a fundamental theme in ergodic theory. While the \emph{isomorphism problem} for general ergodic transformations remains intractable and is inaccessible to countable methods that use countable amount of information \cite{FRW}, it has been solved for some special classes of transformations. Two great successes are the Halmos-von Neumann
classification of ergodic MPT's with pure point spectrum by countable subgroups of the unit circle \cite{HN} and D. Ornstein's classification of Bernoulli shifts by their
measure-theoretic entropy \cite{Or}.

In general, measure-theoretic entropy plays a central role in structural questions for dynamical systems and is a crucial tool in detecting chaoticity of a system. However, the measure-theoretic entropy is positive if and only if the system has exponential growth of distinguishable orbit types and it does not provide any information for systems with slower orbit growth. To measure the complexity of systems with subexponential orbit growth several invariants like \emph{sequence entropy} by A. Kushnirenko \cite{Kushnirenko}, \emph{slow entropy} by A. Katok and J.-P. Thouvenot \cite{KatokThou}, \emph{measure-theoretic complexity} by S. Ferenczi \cite{Fe}, \emph{entropy dimension} by M. de Carvalho \cite{Carvalho}, \emph{entropy convergence rate} by F. Blume (\cite{Blume1}, \cite{Blume2}, \cite{Blume3}), \emph{scaled entropy} by A.Vershik \cite{Vershik}, and \emph{amorphic complexity} by G. Fuhrmann, M. Gr\"oger, and T. J\"ager have been introduced and studied intensely (e.g. \cite{Goodman}, \cite{K77}, \cite{Hulse}, \cite{HuangShaoYe}, \cite{FerencziPark}, \cite{VerGor}, \cite{ADP}, \cite{DouHuangPark}, \cite{Ver11}, \cite{HKM}, \cite{CK}, \cite{Hochman}, \cite{KanigowskiVinhageWei}, \cite{KanigowskiKundeVinhageWei}, \cite{Vep20}). We refer to the recent survey article \cite{KanigowskiKatokWei} for definitions and further background. In this paper we focus on measure-theoretic slow entropy (see Section \ref{sec:slowentropy}) and address three problems stated in \cite{KanigowskiKatokWei}.

Motivated by results of Ferenczi \cite[Proposition 5]{Fe} and Kanigowski \cite[Proposition 1.3]{Kanigowski} showing that the lower measure-theoretic slow entropy of rank-one systems with respect to the polynomial scaling function $a_n(t)=n^t$ is bounded from above by $1$, Question 6.2.1 in \cite{KanigowskiKatokWei} asks about the slow entropy for systems of finite rank. On the one hand, we prove in Proposition \ref{prop:Nospacers} that every rank-two transformation \textit{without} spacers has lower measure-theoretic polynomial slow entropy at most $1$. In fact, the proof of Proposition \ref{prop:Nospacers} also gives that any finite rank system with bounded height ratio has linear complexity, see Remark \ref{rem:multitower} for more details. On the other hand, we show in Section \ref{sec:CrazyRankTwo} that there cannot be a general bound on the lower measure-theoretic slow entropy for systems of finite rank.
\begin{theorem}[Slow entropy of finite rank system]\label{thm:slowentropyFiniterank}
	For any scaling function $\{a_n(t)\}_{n\in\mathbb{N},t>0}$ satisfying $\lim_{n\to+\infty}\frac{\log a_n(t)}{n}=0$, there exists an ergodic measure-preserving rank two system $(T,X,\mathcal{B},\mu_X)$ such that its lower slow entropy 
	$$\lent_{a_n(t)}^{\mu_X}(T)=+\infty,$$
	where $X$ is a finite interval of $[0,+\infty)$,  $\mu_X(A)=\frac{\mu(A\cap X)}{\mu(X)}$ for any $A\in\mathcal{B}$ and $\mu$ is the Lebesgue measure of $\mathbb{R}$.
\end{theorem}

The construction bases upon the \emph{cutting-and-stacking technique} (cf. \cite{Fried}) and uses additional spacer levels. The idea of the construction is to construct a rank-two system such that the height ratio between the two towers goes to infinity by adding spacers in a staircase-like manner to the shorter tower at each step and exchanging mass between the short tower and tall tower using a probabilistic approach. Roughly speaking, the ``staircase-like'' spacers forbids the long matching phenomenon in the short tower part and ``randomness'' of exchanging mass forbids the long matching phenomenon in the tall tower part.

The aforementioned probabilistic approach is a key method for our constructions throughout the paper. We present it in detail in Section \ref{sec:probabilisticLemma} and expect it to have further applications in combinatorial constructions.

In particular, we use it to answer Question 6.1.2 from \cite{KanigowskiKatokWei}. Here, we recall that a measure preserving dynamical system $(X, \mathscr{B},\mu, T)$, $T$ is \emph{rigid} if there exists an increasing sequence $\{t_n\}_{n \in \N}$ such that for any $f\in L^2(X,\mathscr{B},\mu)$, $T^{t_n}f\to f$ in $L^2$. Since rigid transformations always have measure-theoretic entropy zero \cite{FuWe}, Question 6.1.2 from \cite{KanigowskiKatokWei} asks whether it is possible for the upper measure-theoretic slow entropy of a rigid transformation to be positive with respect to the polynomial scale $a_n(t)=n^t$. In Section \ref{sec:rigidConstruction} we answer this question in the affirmative and obtain the following flexibility result.
\begin{theorem} \label{theo:rigidUpper}
	For every $u\in [0,\infty]$ there exists an ergodic Lebesgue measure preserving rigid transformation $T$ with $\uent^{\mu}_{n^t}(T)=u$.
\end{theorem}
We are even able to construct rigid transformations with positive lower measure-theoretic polynomial slow entropy.
\begin{theorem} \label{theo:rigidLower}
	For every $u\in [0,\infty]$ there exists an ergodic Lebesgue measure preserving rigid transformation $T$ with $\lent^{\mu}_{n^t}(T)= u$.
\end{theorem}
Currently we do not know if there exists a rigid transformation $T$ with $\lent^{\mu}_{n^t}(T)=\uent^{\mu}_{n^t}(T)=u$ for some $0<u<\infty$. In Theorem \ref{theo:rigidLowerInf} we actually show that for any given subexponential rate $a_n(t)$ there exists an ergodic Lebesgue measure preserving rigid transformation $T$ with $\lent^{\mu}_{a_n(t)}(T)= \infty$. Independently and by different methods, T. Adams also gave a positive answer to \cite[Question 6.1.2]{KanigowskiKatokWei} in his recent paper  \cite{Adams}. He shows that for any given subexponential rate $a_n(t)$ there exists a rigid and weakly mixing transformation such that the lower slow entropy is infinite with respect to $a_n(t)$. While his construction uses the technique of independent cutting-and-stacking, we use an abstract \emph{Approximation by Conjugation method} (also known as AbC or Anosov-Katok method).

The AbC method was developed by D. Anosov and A. Katok in \cite{AK} and is one of the most powerful tools of constructing smooth volume-preserving diffeomorphisms of entropy zero with prescribed ergodic or topological properties. In particular, it provided the first example of an ergodic $C^{\infty}$ diffeomorphism on the disc $\mathbb{D}^2$. In a more general framework as described in \cite[chapter 8]{K}, one can consider any measure space admitting a non-trivial circle action $\mathcal{R} = \left\{R_t\right\}_{t \in \mathbb{S}^1}$. The transformations are constructed as limits of conjugates $T_n = H_n \circ R_{\alpha_{n+1}} \circ H^{-1}_n$ with $\alpha_{n+1} = \frac{p_{n+1}}{q_{n+1}}=\alpha_n + \frac{1}{k_n \cdot l_n \cdot q^2_n} \in \mathbb{Q}$ and $H_n = H_{n-1} \circ h_n$, where the $h_n$'s are measure-preserving maps in the regularity under consideration (e.g. measurable, smooth or real-analytic) satisfying $R_{\frac{1}{q_n}} \circ h_n = h_n \circ R_{\frac{1}{q_n}}$. In each step the conjugation map $h_n$ and the parameter $k_n$ are chosen such that the transformation $T_n$ imitates the desired property with a certain precision. In a final step of the construction, the parameter $l_n$ is chosen large enough to guarantee closeness of $T_{n}$ to $T_{n-1}$, and so the convergence of the sequence $\left(T_n\right)_{n \in \mathbb{N}}$ to a limit transformation in the measurable, smooth, or even real-analytic category is provided. We refer to the survey articles \cite{FK} and \cite{Ksurvey} for more information on the AbC method and its wide range of applications.

In many ways, the AbC method can be seen as a constructive and smooth version of the concept of \emph{periodic approximation in ergodic theory} which was developed by A. Katok and A. Stepin \cite{KS}. The  concept has numerous applications in measurable dynamics and helped to solve several long-standing problems dating back to von Neumann and Kolmogorov. We give a primer on periodic approximation in Section \ref{sec:periodicApp} and recommend \cite{K} as a comprehensive exposition of the concept.

In particular, our transformations from Theorems \ref{theo:rigidUpper} and \ref{theo:rigidLower} admit periodic approximations. Hereby, these results already give a first answer to Question 6.3.1 in \cite{KanigowskiKatokWei} which asks about the slow entropy of systems constructed by fast periodic approximation. A particularly important class are cyclic approximations, where the periodic process consists of a single tower (see Definition \ref{def:cyclicApp} for the precise notion). The existence of a good cyclic approximation implies various properties for the limit transformation $T$ as summarized in \cite[Proposition 3.2]{K}, e.g. $T$ is ergodic, rigid, has simple spectrum and its maximal spectral type is singular. In Section \ref{sec:goodCyclic} we continue our investigation of slow entropy for transformations with periodic approximation by realizing systems with good cyclic approximation and any given value of upper polynomial slow entropy.
\begin{theorem} \label{theo:cyclicUpper}
	For every $u\in [0,\infty]$ there exists a Lebesgue measure preserving transformation $T$ with good cyclic approximation and $\uent^{\mu}_{n^t}(T)=u$.
\end{theorem}
By \cite{KS} (see also \cite[Proposition 3.2]{K}) a transformation admitting a good cyclic approximation is ergodic and rigid. Hence, our construction for Theorem \ref{theo:cyclicUpper} provides another family of examples for Theorem \ref{theo:rigidUpper}. In fact, we show in Theorem \ref{theo:CyclicInfinity} that for any subexponential scale $a_n(t)$ there is a transformation $T$ with good cyclic approximation and $\uent^{\mu}_{a_n(t)}=\infty$. Like \cite[Theorem 1]{Adams} these constructions confirm a conjecture by Ferenczi in \cite[page 205]{Fe} that there are rank one systems $T$ with $\uent^{\mu}_{n^t}(T)=\infty$ (therein the conjecture is phrased in terms of the measure-theoretic complexity). To the best of our knowledge, they also provide the first standard systems (i.e. loosely Bernoulli with zero entropy or loosely Kronecker) with growth rate faster than any polynomial scale.

Since transformations with good cyclic approximations are rank one, their lower measure-theoretic polynomial slow entropy can be at most $1$ by the results of Ferenczi and Kanigowski (see Proposition \ref{prop:upperBoundLowerCyclic} for more details). We realize all possible values of the lower polynomial slow entropy.
\begin{theorem}\label{theo:cyclicLower}
	For every $u\in [0,1]$ there exists a Lebesgue measure preserving transformation $T$ with good cyclic approximation and $\lent^{\mu}_{n^t}(T)=u$.
\end{theorem}

These transformations are constructed by a twisted version of the abstract AbC method. We stress that the AbC transformations constructed in this paper are not necessarily smooth. (To guarantee smoothness further growth conditions have to be posed on the parameter sequence $\left(l_n \right)_{n\in \N}$.) In an accompanying paper \cite{BKW} we present a method how to compute the measure-theoretic and topological slow entropy of smooth AbC diffeomorphisms.

\paragraph{Acknowledgments:} The authors would like to thank Elon Lindenstrauss and Svetlana Katok for their warm encouragement. We are grateful to Adam Kanigowski for helpful discussions and warm support. We are also grateful to Terry Adams for sharing his preprints and helpful discussions. S.~B. would like to thank Federico Rodriguez Hertz for an invitation for a project related visit to Penn State, however this trip could not be materialized since travel restrictions were imposed due to the COVID19 outbreak.

\paragraph{Plan of the rest of the paper:} In Section \ref{sec:DefandPre}, we provide definitions of measure-theoretic slow entropy, finite rank systems, and periodic approximations. In Section \ref{sec:probabilisticLemma}, we prove a probabilistic lemma that will serve as a key tool for our combinatorial constructions throughout the paper. For instance, we will use it to determine the combinatorial mapping behavior of the conjugation maps in our AbC constructions in Sections \ref{sec:rigidConstruction} as well as \ref{sec:goodCyclic}. In Section \ref{sec:rigidConstruction}, we prove the flexibility of both upper and lower measure-theoretic slow entropy of the rigid transformations with respect to polynomial scale. In Section \ref{sec:goodCyclic}, we first show the flexibility of upper measure-theoretic slow entropy for good cyclic transformations with values from $0$ to $+\infty$ with respect to polynomial scale. Then we prove flexibility of lower measure-theoretic polynomial slow entropy of good cyclic transformations with values from $0$ to $1$, which implies that Proposition 1.3 in \cite{Kanigowski} is sharp. In Section \ref{sec:finiterank}, we turn to rank-two systems. We first provide an upper bound of lower measure-theoretic slow entropy for rank-two system without spacers. Then for any subexponential scale, we construct a rank-two system such that its lower measure-theoretic slow entropy is infinite with respect to the given scale.

\section{Definitions and Preliminaries}\label{sec:DefandPre}
\subsection{Measure-theoretic slow entropy}\label{sec:slowentropy}

Suppose $T$ is a measure-preserving transformation on a standard Borel probability space $(X,\mathcal{B},\mu)$, $\mathcal{P}=\{P_1,\ldots,P_m\}$ is a finite measurable partition and $$\Omega_{m,n}=\{w=(w_k)_{k=0}^{n-1}:w_k\in\{1,\ldots,m\}\}.$$ The coding map $\phi_{\mathcal{P},n}(x):X\to\Omega_{m,n}$ of $T$ with respect to partition $\mathcal{P}$ is defined as $\phi_{\mathcal{P},n}(x)=w(x)$, where $T^k(x)\in P_{w_k(x)}$. Then for $w,w'\in\Omega_{m,n}$, the Hamming metric $d_n^H(w,w')$ is defined by
\begin{equation}\label{eq:HammingMetric}
d_{n}^H(w,w')=\frac{1}{n}\sum_{k=0}^{n-1}(1-\delta_{w_k,w'_k}),
\end{equation}
where $\delta_{a,b}=\left\{
                     \begin{array}{ll}
                       1, & \hbox{if $a=b$;} \\
                       0, & \hbox{if $a\neq b$.}
                     \end{array}
                   \right.$. Next define an $\epsilon$-Hamming ball as
$$B_{\mathcal{P},n}(x,T,\epsilon)=\{y\in X:d_n^H(w(x),w(y))<\epsilon\}$$
and a family $\kappa_n(\epsilon,T,\mathcal{P})$ of  $\epsilon$-Hamming balls such that $\mu(\cup\kappa_n(\epsilon,T,\mathcal{P}))>1-\epsilon$ is a $(\epsilon,\mathcal{P},n)$-covering of $X$. Finally we denote the minimal cardinality of  $(\epsilon,\mathcal{P},n)$-covering as \begin{equation}
S_{\mathcal{P}}(T,n,\epsilon)=\min\{\operatorname{Card}(\kappa_n(\epsilon,T,\mathcal{P}))\}.
\end{equation}


In this setting, slow entropy of invertible measure-preserving transformations can be introduced as follows:
\begin{definition}[Slow entropy, Katok-Thouvenot \cite{KatokThou}]\label{def:slowEntropy}
Let $\{a_n(t)\}_{n\in\mathbb{N},t>0}$ be a family of positive sequences increasing to infinity and monotone in $t$. Then define the lower measure-theoretic slow entropy of $T$ with respect to a finite measurable partition $\mathcal{P}$ by
\begin{equation}\label{eq:slowEntroDef}
\lent_{a_n(t)}^{\mu}(T,\mathcal{P})=\lim_{\epsilon\to0} A(\epsilon,\mathcal{P}),
\end{equation}
where
$
A(\epsilon,\mathcal{P})=\left\{
\begin{array}{ll}
\sup B(\epsilon,\mathcal{P}), & \hbox{if $B(\epsilon,\mathcal{P})\neq\emptyset$;} \\
0, & \hbox{if $B(\epsilon,\mathcal{P})=\emptyset$,}
\end{array}
\right.
$
for
\begin{equation}\label{eq:lowerSlowKey}
B(\epsilon,\mathcal{P})=\{t>0:\liminf_{n\to\infty}\frac{ S_{\mathcal{P}}(T,n,\epsilon)}{a_n(t)}>0\}.
\end{equation}
The lower measure-theoretic slow entropy of $T$ is defined by
\begin{equation}
\lent_{a_n(t)}^{\mu}(T)=\sup_{\mathcal{P}}\lent_{a_n(t)}^{\mu}(T,\mathcal{P}).
\end{equation}
The upper measure-theoretic slow entropy of $T$ can be defined in a similar way by replacing $\liminf$ by $\limsup$ in \eqref{eq:lowerSlowKey} and denoted as $\uent_{a_n(t)}^{\mu}(T)$. If $\lent_{a_n(t)}^{\mu}(T)=\uent_{a_n(t)}^{\mu}(T)$, then we define this value as the measure-theoretic slow entropy of $T$ with respect to $a_n(t)$ and denote it as $\ent_{a_n(t)}^{\mu}(T)$.
\end{definition}

One of the most important features of measure-theoretic slow entropy is the following generating sequence property:
\begin{proposition}[Proposition 1 in \cite{KatokThou}]\label{prop:generatingSequence}
Let $(X,\mathcal{B},\mu,T)$ be a measure-preserving transformation and  $\mathcal{P}_1\leq\mathcal{P}_2\leq\ldots$ be an increasing sequence of finite measurable partitions of $X$ such that $\vee_{m=1}^{+\infty}\mathcal{P}_m$ generates the $\sigma-$algebra $\mathcal{B}$. Then for any scales $a_n(t)$, we have
$$\lent_{a_n(t)}^{\mu}(T)=\lim_{m\to+\infty}\lent_{a_n(t)}^{\mu}(T,\mathcal{P}_m),$$
$$\uent_{a_n(t)}^{\mu}(T)=\lim_{m\to+\infty}\uent_{a_n(t)}^{\mu}(T,\mathcal{P}_m).$$
\end{proposition}

We end this subsection with a modified folklore estimate on the cardinality of words that are $\epsilon$-Hamming close to a given word. Similar versions have appeared in several references, e.g. \cite[Theorem 1]{Katok80}, \cite[Proposition 2]{Fe} and \cite[Lemma 2.3]{Adams}. We provide a proof here for completeness:
\begin{lemma}\label{lem:HammingCloseEstimate}
Let $\Sigma$ be a finite alphabet, $\epsilon\in(0,\frac{1}{10})$ and $w\in\Sigma^n$. Then the cardinality of words in $\Sigma^n$, that are $\epsilon$-Hamming close to $w$, is less than
$$\exp(n(2\epsilon\log|\Sigma|-(1-\epsilon)\log(1-2\epsilon)-3\epsilon\log\epsilon)),$$
where $|\Sigma|$ is the cardinality of $\Sigma$. In particular, in case of $\epsilon \geq \frac{1}{|\Sigma|}$ and $\log|\Sigma|\geq1$ we can estimate this upper bound on the number of $\epsilon$-close $n$-words by $|\Sigma|^{9\epsilon n}$.
\end{lemma}
\begin{proof}
Suppose $\epsilon n\geq 1$ and $w'\in\Sigma^n$ is $\epsilon$-Hamming close to $w$, then at most $[\epsilon n]+1$
digits of $w'$ may be different from corresponding digits in $w$. We denote them as $i_1,i_2,\ldots,i_{[\epsilon n]+1}$, where $[\epsilon n]$ is the integer part of $\epsilon n$. It is worth to point that the $\epsilon$-Hamming closeness of $w$ and $w'$ does not forbid $w'_{i_k}=w_{i_k}$ for some $k=1,\ldots,[\epsilon n]+1$, which corresponds to the cases that the Hamming distance between $w'$ and $w$ is even smaller than $\epsilon$. Since $w'_{i_k}$ may be equal to $w_{i_k}$ for $k=1,\ldots,[\epsilon n]+1$, the total cardinality of possible choices of codes on a fixed set of indices $\{i_1,\ldots,i_{[\epsilon n]+1}\}$ is  $|\Sigma|^{[\epsilon n]+1}$. Moreover, the total cardinality of possible choices of index sets $\{i_1,\ldots,i_{[\epsilon n]+1}\}$ is less than $\binom{n}{[\epsilon n]+1}$. Combining these two estimates, the total cardinality of length $n$ words that are $\epsilon$-Hamming close to $w$ is less than
$\binom{n}{[\epsilon n]+1}|\Sigma|^{[\epsilon n]+1}$.
By definition of binomial coefficients and Stirling's approximation formula, we get the following estimates:
\begin{equation*}
\begin{aligned}
\binom{n}{[\epsilon n]+1}|\Sigma|^{[\epsilon n]+1}&\leq\frac{n!}{([\epsilon n]+1)!(n-[\epsilon n]-1)!}|\Sigma|^{[\epsilon n]+1}\\
&\leq \frac{e\sqrt{ n}(\frac{n}{e})^n|\Sigma|^{[\epsilon n]+1}}{\sqrt{2\pi([\epsilon n]+1)}(\frac{[\epsilon n]+1}{e})^{[\epsilon n]+1}\sqrt{2\pi(n-[\epsilon n]-1)}(\frac{n-[\epsilon n]-1}{e})^{n-[\epsilon n]-1}}\\
&\leq\frac{e\sqrt{ n}(\frac{n}{e})^n|\Sigma|^{2\epsilon n}}{\sqrt{2\pi\epsilon n}(\frac{\epsilon n}{e})^{[\epsilon n]+1}\sqrt{2\pi(n-2\epsilon n)}(\frac{n-2\epsilon n}{e})^{n-[\epsilon n]-1}}\\
&=\frac{e|\Sigma|^{2\epsilon n}}{2\pi\sqrt{(n-2\epsilon n)}\epsilon^{[\epsilon n]+1+\frac{1}{2}}(1-2\epsilon)^{n-[\epsilon n]-1}}\\
&\leq\frac{e|\Sigma|^{2\epsilon n}}{2\pi\sqrt{(n-2\epsilon n)}\epsilon^{3\epsilon n}(1-2\epsilon)^{n-\epsilon n}}\\
&=\frac{e}{2\pi\sqrt{(n-2\epsilon n)}}\cdot\frac{|\Sigma|^{2\epsilon n}}{\epsilon^{3\epsilon n}(1-2\epsilon)^{n-\epsilon n}}\\
&\leq \frac{e}{2\pi\sqrt{n(1-2\epsilon)}}\cdot\frac{|\Sigma|^{2\epsilon n}}{\epsilon^{3\epsilon n}(1-2\epsilon)^{n-\epsilon n}}\\
&\leq\exp(n(2\epsilon\log|\Sigma|-(1-\epsilon)\log(1-2\epsilon)-3\epsilon\log\epsilon))
\end{aligned}
\end{equation*}
where last inequality follows from $\epsilon n\geq 1$ and $\epsilon\in(0,\frac{1}{10})$.

If $\epsilon n<1$, then by definition of Hamming metric, we know that the cardinality of words in $\Sigma^n$ that are $\epsilon$-Hamming close to $w$ is $1$. Since $\exp(n(2\epsilon\log|\Sigma|-(1-\epsilon)\log(1-2\epsilon)-3\epsilon\log\epsilon))\geq 1$ for any $n\in\mathbb{Z}^+$ and $\epsilon\in(0,\frac{1}{10})$, we complete the estimate.

If $\epsilon\geq\frac{1}{|\Sigma|}$ and $\log|\Sigma|\geq1$, then we obtain that
\begin{equation*}
\begin{aligned}
&-(1-\epsilon)\log(1-2\epsilon)\leq\log(1+\frac{2\epsilon}{1-2\epsilon})\leq \frac{2\epsilon}{1-2\epsilon}\leq 4\epsilon\leq4\epsilon\log|\Sigma|,\\
&-3\epsilon\log\epsilon\leq3\epsilon\log|\Sigma|,
\end{aligned}
\end{equation*}
which completes the proof.
\end{proof}

\subsection{Finite rank systems}
Although there are many equivalent definitions of finite rank systems, we only provide the following geometrical definition of finite rank systems in our paper. For other equivalent definitions we refer to \cite{Ferenczirank} for more details.
\begin{definition}[Finite rank system \cite{Ferenczirank}]\label{def:finiteRank}
A system is of rank at most $\ell$ if for every partition $\mathcal{P}$ of $X$, for every $\epsilon>0$, there exist $\ell$ subsets $B_i$ of $X$, $\ell$ positive integers $h_i$ and a partition $\mathcal{P}'$ of $X$ such that
\begin{enumerate}[label=\alph*)]
  \item $T^jB_i$ are disjoint for $1\leq i\leq\ell$ and $0\leq j\leq h_i-1$;
  \item $|\mathcal{P}-\mathcal{P}'|<\epsilon$;
  \item $\mathcal{P}'$ is refined by the partition $$\{T^jB_i:1\leq i\leq\ell, 0\leq j\leq h_i-1\}\cup\{X\setminus\cup_{1\leq i\leq\ell, 0\leq j\leq h_i-1}T^jB_i\}.$$
\end{enumerate}
A system is of rank $\ell$ if it is of rank at most $\ell$ and not of rank at most $\ell-1$. A system is of infinite rank if no such finite $\ell$ exists. Moreover, if $X\setminus\cup_{1\leq i\leq\ell, 0\leq j\leq h_i-1}T^jB_i=\emptyset$, then we call such systems a finite rank systems \emph{without} spacers.
\end{definition}

\begin{remark}\label{rem:towerLength}
It is worth to point out that for a rank $\ell$-system $(X,T,\mathcal{B},\mu)$, each $h_i$ will go to $+\infty$ as $\mathcal{P}$ converges to an atom partition and $\epsilon\to0$. If there is at least one bounded $h_i$ in this process, then $\mathcal{P}\to\text{atom partition}$ and $\epsilon\to0$ imply that $\mu(B_i)\to0$ and thus $\mu(\bigcup_{j=0}^{h_i}T^jB_i)\to 0$, which gives that the rank of $T$ is less than or equal to $\ell-1$ and thus a contradiction.
\end{remark}
\begin{remark}\label{rem:fullmeasure}
If $\mathcal{P}_n$ is a family of partitions that converges to atom partition as $n\to+\infty$ and $\epsilon>0$, then by taking a subsequence if necessary, we can always assume that  $\sum_{i=1}^{\ell}h_n^{(i)}\mu(B_n^{(i)})\geq1-\epsilon$ as the measure of each atom of $\mathcal{P}_n$ goes to zero as $n\to\infty$, where $h_n^{(i)}$ and $B_n^i$ are the heights and bases in Definition \ref{def:finiteRank} with respect to $\mathcal{P}_n$ and $\epsilon$.
\end{remark}

\subsection{Periodic approximation}\label{sec:periodicApp}
In this section, we provide an elementary introduction about periodic approximation, which will be used in Section \ref{sec:goodCyclic}. We refer to \cite{K} for more details.

Suppose $(X,\mu)$ is a Lebesgue space. A tower $t$ is an ordered sequence of disjoint subsets: $t=\{c_1,\ldots,c_{h(t)}\}$ of $X$ with $\mu(c_i)=m(t)$ for $1\leq i\leq h(t)$, where $c_1$ is denoted as the base of the tower $t$ and  $h(t)$ is denoted as the height of the tower $t$. Moreover, we associated a cyclic measure-preserving permutation $\sigma$ with tower $t$ such that $\sigma(c_i)=c_{i+1}$ for $1\leq i\leq h(t)$, where $c_{h(t)+1}=c_1$. Then the periodic process can be defined as follows:
\begin{definition}[Periodic process]
A periodic process is a collection of disjoint towers covering $X$, together with an equivalence relation among these towers which identifies their bases.
\end{definition}

For any given periodic process, we introduce the following two standard partitions associated with the periodic process. The first partition $\xi$ consists of all elements of all towers of the periodic process and the second partition $\eta$ consists of the unions of bases of towers in each equivalence class and their images under the iterates of $\sigma$, where we will drop certain towers if we go beyond the heights of these towers and we continue until the highest tower in the equivalence class has been exhausted. It is clear that $\eta$ is a subpartition of $\xi$, i.e. $\eta\leq\xi$. In fact, $(\xi,\eta,\sigma)$ completely determines the periodic process except for the situation that all towers within an equivalence class have the same height. Since we will not be interested in making this distinction, we identify the triple $(\xi,\eta,\sigma)$ with the periodic process from now on. With the help of these notations, we can introduce the concept of periodic approximations:
\begin{definition}[Exhaustive periodic process]
The sequence $(\xi_n,\eta_n,\sigma_n)$ of periodic processes is called exhaustive if $\eta_n$ converges to the decomposition into points as $n\to\infty$.
\end{definition}
\begin{definition}[Periodic approximation]
For a given measure-preserving transformation $T$ and an exhaustive sequence of periodic processes $(\xi_n,\eta_n,\sigma_n)$, we say $(\xi_n,\eta_n,\sigma_n)$ forms a periodic approximation of $T$ if
$$\sum_{c\in\xi_n}\mu(Tc\triangle\sigma_nc)\to 0\text{ as }n\to\infty.$$
\end{definition}

One of the important features of periodic approximations is the speed of approximation:
\begin{definition}
A measure-preserving transformation admits a periodic approximation $(\xi_n,\eta_n,\sigma_n)$ with speed $g(n)$ if for a certain subsequence $\{n_k\}$ we have
$$\sum_{c\in\xi_{n_k}}\mu(Tc\triangle\sigma_{n_k}c)<g(n_k).$$
\end{definition}

Now we have enough ingredients to define a good cyclic approximation.

\begin{definition}[Cyclic approximation]\label{def:cyclicApp}
A periodic process which consists of a single tower is called a cyclic process. An approximation by an exhaustive sequence of cyclic processes is a cyclic approximation. In particular, a cyclic approximation is a good cyclic approximation if its speed is $o(\frac{1}{q_n})$, where $q_n$ is the height of cyclic processes' single tower at step $n$.
\end{definition}

\section{A probabilistic method}\label{sec:probabilisticLemma}
One of the key tools for all our combinatorial constructions throughout this paper is a probabilistic method similar to the so-called ``Substitution Lemma'' in \cite{FRW}. Like \cite[Proposition 45]{FRW} our method allows to select good choices of words so that each constructed sequence contains each letter in the underlying alphabet the same number of times and that all pairs of letters occur with about the same frequency when comparing two substantial substrings with each other, even after some sliding along the sequence. In addition to that, our word construction also takes care of Hamming-separation for substrings of shorter lengths. We need this extension in order to prove our statements on the lower slow entropy.

To state our probabilistic method precisely, we introduce some notation.

\begin{definition}
	Let $\Sigma$ be an alphabet. For a word $w\in\Sigma^{k}$ and $x\in\Sigma$
	we write $r(x,w)$ for the number of times that $x$ occurs in $w$
	and $\freq(x,w)=\frac{r(x,w)}{k}$ for the frequency of occurrences
	of $x$ in $w$. Similarly, for $(w,w')\in\Sigma^{k}\times\Sigma^{k}$
	and $(x,y)\in\Sigma\times\Sigma$ we write $r(x,y,w,w')$ for the
	number of $i<k$ such that $x$ is the $i$-th member of $w$ and
	$y$ is the $i$-th member of $w'$. We also introduce $\freq(x,y,w,w')=\frac{r(x,y,w,w')}{k}$.
\end{definition}
We also state the Law of Large Numbers with its large deviations using Chernoff bounds:
	\begin{lemma}[Law of Large Numbers] \label{lem:law}
	Let $\left(X_{i}\right)_{i\in\mathbb{N}}$ be a sequence of independent
	identically distributed random variables taking value $1$ with probability
	$p$ and taking value $0$ with probability $1-p$. Then for any $\delta>0$ we have
	\[
	\mathbb{P}\left(\left|\frac{1}{n}\sum_{i=1}^{n}X_{i}-p\right|\geq\delta\right)<\exp\left(-\frac{n\delta^{2}}{4}\right)
	\]
	and for the upper tail we have the following estimate
	\[
	\mathbb{P}\left(\sum_{i=1}^{n}X_{i}\geq(1+\delta)np\right)\leq\exp\left(-\frac{\delta^{2}}{2+\delta}np\right).
	\]
\end{lemma}

Inspired by the proof of the Substitution Lemma in \cite{FRW} we apply the Law of Large Numbers to guarantee the existence of selections of words that are Hamming-apart from each other. To avoid Hamming-closeness of too many substrings of shorter length, we have to provide several extensions to \cite[Proposition 44]{FRW}. Here, for a word $w\in \Sigma^k$, we let $\sh^t(w)$ be the word whose entries are moved to the left by $t$ units. Furthermore, $w\upharpoonright I$ denotes the restriction of $w$ on the index set $I$, i.e. if $I=\{i_1,i_2,\ldots\}$, then $(w\upharpoonright I)_p=w_{i_p}$. With the help of these notations, our probabilistic method can be stated as follows.

\begin{proposition}\label{prop:prob}
	Let $\epsilon\in(0,\frac{1}{10}), \gamma\in(0,1)$ and $\Sigma$ be a finite alphabet with at least $4$ symbols. For any increasing sequence ${\mathbf b}=\{b_n\}_{n\in \N}$ with $\lim_{n\to \infty} \frac{\log b_n}{n} = 0$ there exists $K_0 \in\mathbb{N}$ such that for all $k\geq K_0$, that are multiples of $|\Sigma|$, and all $N\leq b_k$ there are $\tau_{{\mathbf b}},\beta(k) \in \mathbb{N}$ with $\tau_{{\mathbf b}}\leq\beta(k)$ and a collection of sequences $\Theta\subset\Sigma^k$ with cardinality $|\Theta|=N$ satisfying the following properties:
	\begin{enumerate}
		\item(Exact uniformity) For every $x\in\Sigma$ and every $w\in\Theta$, we have
		$$\operatorname{freq}(x,w)=\frac{1}{|\Sigma|};$$
		\item(Hamming separation on substantial substrings) Let $0\leq t<(1-\gamma)k$, $w,w'\in\Theta$ and $I$ be an interval of indices in the overlap of $w$ and $\sh^t(w')$. Suppose its length $L\geq \gamma k$. If $w,w'$ are different from each other, then we have
		\begin{equation}
			d^H_{L}(w\upharpoonright I,\sh^t(w')\upharpoonright I)\geq 1-\frac{1}{|\Sigma|}-\gamma|\Sigma|;
		\end{equation}
		if $1\leq t\leq(1-\gamma)k$, then we also have
		\begin{equation}\label{eq:selfSliding}
			d^H_{L}(w\upharpoonright I,\sh^t(w)\upharpoonright I)\geq 1-\frac{1}{|\Sigma|}-\gamma|\Sigma|;
		\end{equation}
		\item(Hamming separation for substrings of intermediate length) Let $0\leq t<(1-2\gamma)k$, $w,w'\in\Theta$ and $I$ be a subinterval of indices in the overlap of $w$ and $\sh^t(w')$. Suppose its length $\beta(k) \leq L < \gamma k$. If $w,w'$ are different from each other, then we have
		\begin{equation}
			d^H_{L}(w\upharpoonright I,\sh^t(w')\upharpoonright I)\geq \epsilon;
		\end{equation}
		if $1\leq t\leq(1-2\gamma)k$, then we also have
		\begin{equation}\label{eq:selfSliding2}
			d^H_{L}(w\upharpoonright I,\sh^t(w)\upharpoonright I)\geq \epsilon,
		\end{equation}
		\item(Hamming separation on short substrings) Let $\tau_{{\mathbf b}} \leq L < \beta(k)$, $0 \leq t < (1-2\gamma)k$, $w\in \Theta$, and $W$ be the substring of $w$ on the indices $[t,t+L) \cap \mathbb{Z}$. Then we have that $W$ is $\epsilon$-Hamming close to a proportion of at most $\frac{2}{b_L}$ of $L$-substrings in $w'\upharpoonright[0,(1-\gamma)k)$ for any $w'\in \Theta$, $w'\neq w$.
	\end{enumerate}
\end{proposition}

\begin{remark}
	While property (2) guarantees that the words in $\Theta$ are Hamming-apart from each other, properties (3) and (4) imply that there is also enough Hamming-separation on the level of shorter substrings. Altogether, the proposition yields for every length $\tau_{{\mathbf b}}\leq L \leq k$ that one needs at least $\min\left(N,\frac{b_L}{2}\right)$ many $(\varepsilon,L)$-Hamming balls to cover a proportion of at least $1-2\gamma$ of the $L$-strings occurring in the words in $\Theta$.
\end{remark}

\begin{proof}	
	Let $\sigma<\frac{\gamma^2}{3}$.
	We will use the Law of Large Numbers to show that for sufficiently
	large $k$, most choices in $\Sigma^{(1-\sigma)\cdot k}$
	satisfy approximate versions of properties (1) and (2). Then we
	choose the endstrings such that exact uniformity in property (1) holds
	true. Since the endstring has small proportion $\sigma$ of the total
	length, its effect on the second property will be small. Properties (3) and (4) are restricted to strings not affected by modifications to the endstrings anyway.
	
	We consider $\varOmega_{k}\coloneqq\left(\Sigma^{k} \right)^{N}=\Sigma^{k}\times\dots\times\Sigma^{k}$ equipped with the counting measure as our probability space. For \emph{each} $x\in\Sigma$ and every $i\in\left\{ 0,1,\dots,k-1\right\} $, let $X_{i}$ be the random variable that takes the value $1$ if $x$ occurs in the $i$-th place of an element $w\in\Sigma^{k}$ and $0$ otherwise. Then the $X_{i}$ are independent and identically distributed. Hence, the Law of Large Numbers gives some $\delta>0$ and $k_{x}=k_{x}(\delta)$ such that for all $k\geq k_{x}$ a proportion $\left(1-(\exp(-\delta^2/4))^k\right)$ of sequences in $\Sigma^{k}$ satisfy
	\begin{equation}\label{eq:prop13}
		\sum_{i=0}^{(1-\sigma)k-1}X_{i}\leq \frac{k}{|\Sigma|}.
	\end{equation}
	Moreover, we define for \emph{each} $0\leq t <(1-\gamma)k$ and \emph{every} pair $(x,y)\in\Sigma\times\Sigma$ the random variable $Y^t_{i}$ that takes the value $1$ if $x$ occurs in the $i$-th place of $w$ for an element $w\in\Sigma^{k}$ and $y$ is the $i$-th entry of $\sh^t(w')$ for some $w'\in\Sigma^{k}$, $w'\neq w$. Otherwise, $Y^t_{i}$ takes the value $0$. Since the $Y^t_{i}$, $i=0,\dots, k-t-1$, are independent and identically distributed, the Law of Large Numbers gives $k_{x,y}=k_{x,y}(\sigma)$ such that for all $k\geq k_{x,y}$, all $0\leq t <(1-\gamma)k$, all $0\leq s<(1-\gamma)k-t$, all $\gamma k \leq L <k-t-s$, and all but an $(\exp(-\sigma^2/4))^L$ proportion of sequences $(w,w')\in\Sigma^{k}\times\Sigma^{k}$ satisfy
	\begin{equation}\label{eq:prop23}
		\abs{\frac{1}{L}\sum_{i=s}^{s+L-1}Y^t_{i}-\frac{1}{|\Sigma|^{2}}}<\sigma.
	\end{equation}
	Finally, we introduce for \emph{each} $1\leq t <(1-\gamma)k$ the random variable $Z^t_i$ that takes the value $1$ if the $i$-th symbol of $w$ agrees with the $i$-th symbol of $\sh^t(w)$ for some $w\in \Sigma^k$. Otherwise, $Z^t_i$ takes the value $0$. Since for each $1\leq t <(1-\gamma)k$ the $Z^t_{i}$ are independent and identically distributed, the Law of Large Numbers gives $k(\sigma)$ such that for all $k\geq k(\sigma)$, all $1\leq t <(1-\gamma)k$, all $0\leq s<(1-\gamma)k-t$, and all $\gamma k \leq L <k-t-s$ an $\left(1-(\exp(-\sigma^2/4))^L\right)$ proportion of sequences $w\in\Sigma^{k}$ satisfy
	\begin{equation}\label{eq:prop33}
		\abs{\frac{1}{L}\sum_{i=s}^{s+L-1}Z^t_{i}-\frac{1}{|\Sigma|}}<\sigma.
	\end{equation}
	
	We point out that the number of requirements is less than $N |\Sigma|+2k^3N^2|\Sigma|^{2} \leq 3k^3N^2|\Sigma|^{2}$. Since
	\[
	3k^3N^2|\Sigma|^{2} \cdot (\exp(-\sigma^2/4))^{\gamma k} \leq 3k^3b_k^2|\Sigma|^{2} \cdot (\exp(-\sigma^2/4))^{\gamma k} \to 0 \text{ as } k\to \infty,
	\]
	we conclude by Bernoulli inequality that for sufficiently large $k$ the vast majority of elements in $\varOmega_{k}$ satisfies all the conditions \eqref{eq:prop13}, \eqref{eq:prop23}, and \eqref{eq:prop33}.
	
	In order to get property (3) we introduce for each $0\leq t <(1-2\gamma)k$ the random variable $V^t_{i}$ that takes the value $1$ if the $i$-th entry of an element $w\in\Sigma^{k}$ agrees with the $i$-th entry of $\sh^t(w')$ for some $w'\in\Sigma^{k}$ (in case of $t=0$ we have to assume $w'\neq w$). Otherwise, $V^t_{i}$ takes the value $0$. Since the $V^t_{i}$ are independent and identically distributed, the Law of Large Numbers with $p=\frac{1}{|\Sigma|}$ and $\delta=\frac{1-\varepsilon-p}{p}$ gives $L_{t}=L_{t}(\delta)$ such that for all $L\geq L_{t}$, all $0\leq s <k-t-L$, and all but a $\exp\left(-\frac{\delta^2}{2+\delta}Lp\right)$ proportion of sequences $(w,w')\in\Sigma^{k}\times\Sigma^{k}$ satisfy
	\begin{equation} \label{eq:prob3}
		\frac{1}{L} \sum^{s+L-1}_{i=s}V^t_i \leq (1+\delta)\cdot p = 1-\varepsilon.
	\end{equation}
	We note that this corresponds to a $d^H_L$-distance of at least $\varepsilon$ between the two strings under consideration. Since we want to control all those strings for every $L\in \N$ with $L_{3,\min}\leq L\leq \gamma k$, this gives less than $N^2k^3$ many conditions and we have to bound
	\begin{equation} \label{eq:L1}
		N^2k^3 \cdot \exp\left(-\frac{\delta^2}{2+\delta}L_{3,\min}p\right) \leq N^2k^3 \cdot \exp\left(-\frac{\delta}{2}L_{3,\min}p\right)=N^2k^3 \cdot \exp\left(-\frac{1-\varepsilon-p}{2}L_{3,\min}\right)
	\end{equation}
	where we used $\delta\geq 2$ by $|\Sigma|\geq 4$. The expression in (\ref{eq:L1}) goes to $0$ as $k\to \infty$ if
	\begin{equation}\label{eq:L11}
		L_{3,\min}> \frac{8\cdot (\ln(b_k)+\ln(k))}{1-\varepsilon-p}.
	\end{equation}
	
	To impose property (4) of this proposition we denote for $L\in \mathbb{Z}^+$, $0 \leq t < (1-2\gamma)k$, $w\in \Sigma^k$ the substring of $w$ on the indices $[t,t+L) \cap \mathbb{Z}$ by $W_{w,L,t}$. Then we introduce for each $0\leq i <(1-\gamma)k$ and $w'\in \Sigma^k$  the random variable $U_i$ that takes the value $1$, if the string $w'\upharpoonright[i,i+L-1]$ is $2\epsilon$-Hamming close to $W_{w,L,t}$, and $0$ otherwise. Since for all $w'\neq w$ and each $0\leq s < L$ the $\{U_{iL+s}\}_{0 \leq i < \lfloor \frac{(1-\gamma)k}{L} \rfloor}$ are independent and identically distributed, we can apply the Law of Large Numbers again. Note that in the notation of Lemma \ref{lem:law} we have $n=(N-1)\lfloor \frac{(1-\gamma)k}{L} \rfloor$ and $p=\frac{N_L}{|\Sigma|^L}$, where $N_L$ is the number of $L$-strings in the alphabet $\Sigma$ that are $2\varepsilon$-Hamming close to $W_{w,L,t}$. Since $\{b_L\}_{L\in \N}$ is subexponential with respect to $L$, Lemma \ref{lem:HammingCloseEstimate} implies that there exists $\tau_{\mathbf b}\in \N$ only depending on ${\mathbf b}=\{b_L\}_{L\in\mathbb{N}}$ such that
	\begin{equation} \label{eq:condTau}
		b_Lp<1 \ \text{ for every } L\geq\tau_{\mathbf b}.
	\end{equation}
    Then we apply the upper tail estimate in Lemma \ref{lem:law} with
\begin{equation}\label{eq:delta}
\delta = \frac{1.5}{b_Lp}-1.
\end{equation}
By this choice we have for all but a $\exp\left(-\frac{\delta^2}{2+\delta}np\right)$ proportion of sequences in $\varOmega_k$ that
	\begin{equation} \label{eq:prob4}
		\sum^{n-1}_{j=0} U_j \leq \frac{2}{b_L}\cdot N \cdot \frac{k}{L}
	\end{equation}
where $U_j$ is indexed by the $n$ many observations as described above. Since we want to control \eqref{eq:prob4} for each possible length $\tau_{\mathbf b}\leq L \leq L_{4,\max}$, all the $N$ representatives in $\varOmega_k$ and the $(1-2\gamma)k$ different initial positions $t$, there are less than $Nk^2$ many conditions. Thus, we have to bound
	\begin{equation}\label{eq:L2}
		Nk^2 \cdot \exp\left(-\frac{\delta^2}{2+\delta}np\right) \leq Nk^2 \cdot \exp\left(-\frac{\delta}{2}N\frac{k}{L}p\right) \leq Nk^2 \cdot \exp\left(-\frac{Nk}{4b_LL}\right).
	\end{equation}
	Hence, we need
	\begin{equation}\label{eq:L22}
		L_{4,\max}< \frac{b_kk}{b_{L_{4,\max}}\cdot 4 \cdot (\ln(b_k)+2\ln(k))}.
	\end{equation}
	
	To guarantee that for sufficiently large $k$ the vast majority of elements in $\varOmega_k$ satisfies conditions (3) and (4), it suffices to have $L_{4,\max}>L_{3,\min}$ which follows from
	\begin{equation} \label{eq:Lcombine}
		b_{L_{4,\max}}< \frac{(1-\varepsilon-p)\cdot b_kk}{40 \cdot \left(\ln(b_k)+\ln(k)\right)^2}.
	\end{equation}
	By direct computation we note that this condition is fulfilled for $L_{4,\max}=\gamma' k$ with some $\gamma'>0$ if $\lim_{k\to\infty}k^{-c}b_k=0$ for some $c\in\mathbb{R}^+$. If $\lim_{k\to\infty}k^{-c}b_k=+\infty$ for every $c\in\mathbb{R}^+$, then we write inequality (\ref{eq:Lcombine}) as
	\[
	\frac{b_{L_{4,\max}}}{\left(b_k\right)^{3/4}}< \frac{(1-\varepsilon-p)\cdot \left(b_k\right)^{1/4}\cdot k}{32 \cdot \left(\ln(b_k)+\ln(k)\right)^2},
	\]
	which is fulfilled for $L_{4,\max}<\sqrt{k}$ since 
	\[
	\lim_{k\to \infty} \frac{b_{\sqrt{k}}}{\left(b_k\right)^{3/4}} =0.
	\]
	
	We pick $k$ large enough such that there is such a collection of sequences $\Theta'\subset\Sigma^{k}$ with cardinality $|\Theta'|=N$ satisfying all the conditions (\ref{eq:prop13}), (\ref{eq:prop23}), (\ref{eq:prop33}), (\ref{eq:prob3}), and (\ref{eq:prob4}). Then by inequality \eqref{eq:prop13} in any $w_{in}\in\Theta^{\prime}$ we can replace the symbols at at most $\sigma k$ places in the endstring to obtain a word $w$ in which each element of $\Sigma$ occurs the same number of times. Clearly, the constructed word $w$ satisfies exact uniformity. The sequences built this way constitute our collection $\Theta\subset\Sigma^{k}$.
	
	To check the second property we denote for $w,w^{\prime}\in\Theta$ their original strings in $\Theta^{\prime}$ by $w_{in}$ and $w_{in}^{\prime}$, respectively. From inequality \eqref{eq:prop23} we obtain for every $\gamma k \leq L \leq k-t$, every interval $I$ of length $L$ in the overlap, and every $x,y\in\Sigma$ that
	\[
	\abs{\frac{r(x,y,w_{in}\upharpoonright I,\sh^t(w'_{in})\upharpoonright I)}{L}-\frac{1}{|\Sigma|^{2}}}<\sigma.
	\]
	Since $w_{in}$ and $\sh^t(w'_{in})$ were changed at most $\sigma k$ places, at most $2\sigma k$ positions in the alignment of $w$ and $\sh^t(w')$ can be affected. Hereby, we conclude that
	\begin{equation}
		\begin{aligned}
			& \abs{\freq\left(x,y,w\upharpoonright I,\sh^t(w')\upharpoonright I\right)-\frac{1}{|\Sigma|^{2}}}\\
			\leq & \abs{ \frac{r(x,y,w\upharpoonright I,\sh^t(w')\upharpoonright I)-r(x,y,w_{in}\upharpoonright I,\sh^t(w'_{in})\upharpoonright I)}{L}} \\&+ \abs{\frac{r(x,y,w_{in}\upharpoonright I,\sh^t(w'_{in})\upharpoonright I)}{ L} -\frac{1}{|\Sigma|^{2}}}\\
			\leq & \frac{2\sigma  k}{ \gamma k}+\sigma	< \gamma.
		\end{aligned}
	\end{equation}
	
	In particular, this implies
	\[
	d^H_{L}\left(w\upharpoonright I,\sh^t(w')\upharpoonright I\right)\geq 1-\frac{1}{|\Sigma|}-\gamma |\Sigma|,
	\]
	which yields the first part of property (2). Similarly, we check its second part with the aid of (\ref{eq:prop33}).	
\end{proof}

\begin{remark}\label{rem:tau}
	From condition (\ref{eq:condTau}) in the proof we see that $\tau_{\mathbf b} \in \N$ has to satisfy that
	\[
	\frac{b_L N_L}{|\Sigma|^L}<1 \ \text{ for every } L\geq\tau_{\mathbf b},
	\]
	where $N_L$ is the number of $L$-strings in the alphabet $\Sigma$ that are $2\varepsilon$-Hamming close to a given word in $\Sigma^L$. If $2\varepsilon \geq \frac{1}{|\Sigma|}$, we get the estimate $N_L \leq |\Sigma|^{18L\varepsilon}$ by Lemma \ref{lem:HammingCloseEstimate}. Then we can state the condition on $\tau_{\mathbf b}$ as
	\[
	\frac{b_L}{|\Sigma|^{(1-18\varepsilon)L}}<1 \ \text{ for every } L\geq\tau_{\mathbf b}.
	\]
\end{remark}

\begin{remark} \label{rem:smallerGrowth}
	Let ${\mathbf c}=\{c_n\}_{n\in \N}$ be an increasing sequence with $c_n \leq b_n$ for all $n\in \N$. Then the proof of Proposition \ref{prop:prob} can also be used to construct a collection $\Theta \subset \Sigma^k$ with cardinality $|\Theta|=N \leq b_k$ such that for every length $\tau_{{\mathbf c}}\leq L \leq k$, one needs at least $\min\left(N,\frac{c_L}{2}\right)$ many $(\varepsilon,L)$-Hamming balls to cover a proportion of at least $1-2\gamma$ of the $L$-strings occurring in the words in $\Theta$. As above the number $\tau_{\mathbf c} \in \N$ has to satisfy that
	\[
	\frac{c_L N_L}{|\Sigma|^L}<1 \ \text{ for every } L\geq\tau_{\mathbf c},
	\]
	where $N_L$ is the number of $L$-strings in the alphabet $\Sigma$ that are $2\varepsilon$-Hamming close to a given word in $\Sigma^L$.
\end{remark}

We end this section by making the following immediate observation from Proposition \ref{prop:prob} to fix some notation.

\begin{remark} \label{rem:kForCLT}
	Let $s\in \Z^+$, $s\geq 4$, $\varepsilon,\gamma>0$, and $\{b_n\}_{n\in \N}$ be a subexponential sequence with $b_n \nearrow \infty$. Then there is $K_0=K_0(\{b_n\}_{n\in \N},s,\varepsilon,\gamma)$ such that for all $k\geq K_0$ and any finite alphabet $\Sigma$ of cardinality $|\Sigma|=s$, there is a collection $\Theta \subset \Sigma^{k}$ of cardinality $|\Theta|=\lfloor b_k\rfloor $ such that the words in $\Theta$ satisfy properties (2)-(4) from Proposition \ref{prop:prob}.
\end{remark}

\section{Slow entropy of rigid transformations}\label{sec:rigidConstruction}

\subsection{Combinatorics of the conjugation map} \label{subsec:combin}
We start this section with a brief review of the symbolic representation of AbC transformations from \cite{FW1} and \cite{K}. Here, our constructions can be viewed as taking place on $\mathbb{T}^2=\R^2 / \Z^2$, $\mathbb{D}$ or $\mathbb{A}=\mathbb{S}^1\times [0,1]$. For this purpose, we introduce the following notation with $q,s\in \Z^+$:
\begin{align*}
& \Delta_{q,s} := \{(x,y)\in \S^1\times [0,1) :0\leq x  < \frac{1}{q}, 0\leq y < \frac{1}{s}\},\\
& \Delta_{q,s}^{i,j} := \{(x,y)\in \S^1\times [0,1) :(x,y)=\big(x'+\frac{i}{q},y'+\frac{j}{s}\big)\text{ for some }(x',y')\in \Delta_{q,s}\}.
\end{align*}
We collect the above sets to form the following partition
\begin{align*}
& \xi_{q, s} =\{\Delta_{q,s}^{i,j}: 0\leq i< q,\; 0\leq j< s\}.
\end{align*}
Each transformation $T$ in Theorems \ref{theo:rigidUpper} and \ref{theo:rigidLower} is obtained as a limit of an \emph{untwisted} AbC construction $T_n=H_n \circ R_{\alpha_{n+1}}\circ H^{-1}_n$, i.e. $h_n\left( \Delta_{q_n,1}\right)=\Delta_{q_n,1}$. We use three sequences of parameters $(k_n)_{n\in\N}$, $(l_n)_{n\in\N}$, and $(s_n)_{n\in\N}$. In particular, we suppose that $k_n$ as well as $s_n$ are multiples of $s_{n-1}$. In our application $k_n$ will be chosen large enough to satisfy some requirements determined by the lemmas in Section \ref{sec:probabilisticLemma}. Moreover, we will choose $h_n$ as a measure preserving translation commuting with $R_{\alpha_n}$ and permuting $\xi_{k_nq_n,s_n}$, i.e. it takes atoms of the form $\Delta_{k_nq_n,s_n}^{i,j}$ to atoms $\Delta_{k_nq_n,s_n}^{i',j'}$. Clearly, $h_n$ is determined by its action on $\Delta_{q_n,1}$. We also state some requirements on the action of $h_n$ that guarantee ergodicity of our limit transformation $T$ (see \cite[Theorem 58]{FW1}):
\begin{enumerate}[label=(R\arabic*)]
	\item\label{item:R1} The sequence $s_n$ tends to $\infty$;
	\item\label{item:R2} Strong uniformity: For each $\Delta^{0,j}_{q_{n-1},s_{n-1}}\in \xi_{q_{n-1},s_{n-1}}$ and each $s<s_n$ we have that the cardinality of
	\[
	\Meng{t<k_n}{h_n \left( \Bigg[\frac{t}{k_nq_n},\frac{t+1}{k_nq_n}\Bigg) \times \Bigg[\frac{s}{s_n},\frac{s+1}{s_n}\Bigg) \right)\subseteq \Delta^{0,j}_{q_{n-1},s_{n-1}}}
	\]
	is $\frac{k_n}{s_{n-1}}$;
	\item\label{item:R3} Given $s<s_n$ we can associate a $k_n$-tuple $\left(j_0,j_1,\dots,j_{k_n-1}\right)_s$ such that
	\[
	h_n\left( \Bigg[\frac{t}{k_nq_n},\frac{t+1}{k_nq_n}\Bigg) \times \Bigg[\frac{s}{s_n},\frac{s+1}{s_n}\Bigg) \right) \subseteq \Delta^{0,j_t}_{q_{n-1},s_{n-1}}
	\]
	for all $0\leq t<k_n$. Then we assume that the map $s \mapsto \left(j_0,j_1,\dots,j_{k_n-1}\right)_s$ is one-to-one.
\end{enumerate}

To give the symbolic representation of untwisted AbC transformations we inductively assume that the $T_{n-1}$-names with length $q_n$ of the tower bases $H_{n-1} (\Delta^{0,s}_{q_n,s_{n-1}})$, $0\leq s<s_{n-1}$, are $u_0,\dots , u_{s_{n-1}-1}$. Then for each $0\leq s^{\ast}<s_{n}$ we define a sequence of words $w_0,\dots , w_{k_n -1}$ by setting $w_j=u_s$ iff
\[
h_n\left( \Big[\frac{j}{k_nq_n}, \frac{j+1}{k_nq_n} \Big) \times \Big[\frac{s^{\ast}}{s_n}, \frac{s^{\ast}+1}{s_n}\Big) \right) \subseteq \Delta^{0,s}_{q_{n}, s_{n-1}}.
\]
Then according to \cite[section 7]{FW1} the $T_n$-name with length $k_nl_nq^2_n=q_{n+1}$ of the tower $H_n(\Delta^{0,s^{\ast}}_{q_{n+1}, s_n})$ is given by
\begin{equation}\label{eq:circularOperator}
\mathcal{C}_n(w_0,\dots,w_{k_n-1}) \coloneqq \prod^{q_n -1}_{i=0} \prod^{k_n-1}_{j=0} \left(b^{q_n-j_i} w^{l_n-1}_j e^{j_i}\right),
\end{equation}
where $b,e$ are so-called spacer symbols and $j_i \in \{0,\dots , q_n-1\}$ such that $j_i \equiv (p_n)^{-1}i \mod q_n$.

We also add the following well-known fact for AbC transformations.
\begin{lemma} \label{lem:rig}
	Let $T$ be constructed by the abstract AbC method with any sequence $(l_n)_{n\in\N}$ such that $\sum_{n\in \N}\frac{1}{l_n}<\infty$. Then $T$ is rigid along the sequence $(q_n)_{n\in\N}$.
\end{lemma}

\begin{proof}
	Let $A\subseteq X$ be any measurable set and $\varepsilon>0$. Since the sequence $\{\eta_m\}_{m\in \N}$ of partitions $\eta_m\coloneqq H_m\left(\xi_{k_mq_m,s_m}\right)$ is generating, there is $N\in \N$ such that for all $n\geq N$ we can $\varepsilon$-approximate $A$ by a union of sets in $\eta_n$. We also note that for $t\leq q_{n}$ and $n$ sufficiently large we have by our assumption on the sequence $(l_n)_{n\in\N}$ that
	\begin{align*}
	\sum_{c\in \eta_n} \mu\left(T^t(c)\triangle T^t_{n-1}(c)\right) & \leq  \sum_{i=n}^{\infty}\sum_{c\in \eta_i} \mu\left(T^t_{i}(c)\triangle T^t_{i-1}(c)\right) \\
	& = \sum_{i=n}^{\infty}\sum_{c\in \eta_i} \mu\left(H_{i}R^t_{\alpha_{i+1}}H^{-1}_i(c)\triangle H_{i-1} \circ h_i \circ R^t_{\alpha_i} \circ h^{-1}_i \circ H^{-1}_{i-1}(c)\right) \\
	& = \sum_{i=n}^{\infty}\sum_{\tilde{c}\in \xi_{k_iq_i,s_i}} \mu\left(R^t_{\alpha_{i+1}-\alpha_i}(\tilde{c})\triangle \tilde{c}\right) \\
	& \leq \sum_{i=n}^{\infty}\frac{t}{l_{i}q_{i}} < \sum_{i=n}^{\infty}\frac{1}{l_{i}} < \varepsilon
	\end{align*}
where the first equality is due to $h_i\circ R_{\alpha_i}=R_{\alpha_i}\circ h_i$. Since $T^{q_{n}}_{n-1}=\text{id}$, we conclude $\mu\left(T^{q_n}(A)\triangle A\right)<2\varepsilon$. This yields the rigidity of $T$.
\end{proof}

\begin{remark} \label{rem:safety}
	When moving from the partition element $\Delta^{i,j}_{k_nq_n,s_n}$ to $\Delta^{i+1,j}_{k_nq_n,s_n}$ the mapping behavior of $h_n$ changes instantly. To keep control of the orbit of $T^t_n$ compared to $T^t_{n-1}$ for small numbers of iterates $t\leq q_n$ we introduce the sets
	\[
	\overline{\Delta}^{i,j}_{k_nq_n,s_n} =\Bigg[ \frac{i}{k_nq_n}, \frac{i+1}{k_nq_n}-\frac{2}{k_nl_nq_n} \Bigg) \times \Bigg[\frac{j}{s_n}, \frac{j+1}{s_n} \Bigg).
	\]
	Since $t\cdot \abs{\a_{n+1}-\a_n}\leq \frac{1}{k_nl_nq_n}$, for any point $P\in \overline{\Delta}^{i,j}_{k_nq_n,s_n}$ the images $T^t_n(P)$ and $T^t_{n-1}(P)$ lie in the same element $\Delta^{i',j'}_{k_nq_n,s_n}$.
	
	Then we define
	\[
	\Theta_n \coloneqq \bigcup_{0\leq i<k_nq_n}\bigcup_{0\leq j <s_n} \overline{\Delta}^{i,j}_{k_nq_n,s_n}
	\]
	and call it the \emph{safe domain}.
	
	Similarly, to compare the orbits of $T^t$ and $T^t_{n-1}$ for small numbers of iterates $t\leq q_n$ we define for $m\geq n$ the sets
	\[
	\tilde{\Delta}^{i,j}_{k_mq_m,s_m} =\Bigg[ \frac{i}{k_mq_m}, \frac{i+1}{k_mq_m}-\frac{2q_n}{k_ml_mq^2_m} \Bigg) \times \Bigg[\frac{j}{s_m}, \frac{j+1}{s_m} \Bigg)
	\]
	and its union
	\[
	\Xi_{n,m} \coloneqq \bigcup_{0\leq i<k_mq_m}\bigcup_{0\leq j <s_m} \tilde{\Delta}^{i,j}_{k_mq_m,s_m}.
	\]
	Note that $\Xi_{n,n}=\Theta_n$. Hereby, we define
	\[
	\Xi_n \coloneqq \bigcap_{m\geq n} \Xi_{n,m}.
	\]
	Then we note that for any $0\leq t\leq q_n$ and any point $P\in \Delta^{i,j}_{k_nq_n,s_n}\cap \Xi_n$ the images $T^t(P)$ and $T^t_{n-1}(P)$ lie in the same element $\Delta^{i',j'}_{k_nq_n,s_n}$.
	
	We note that $\mu\left(\Xi_{n,m}\right) \geq  1-\frac{2}{l_m}$. Under the assumption $\sum_{m\in \N}\frac{1}{l_m}<\infty$ this yields for any given $\delta>0$ that $\mu\left(\Xi_n  \right)>1-\delta$ for $n$ sufficiently large.
\end{remark}

\subsection{Proof of Theorem \ref{theo:rigidUpper}}
Let $0 < u <\infty$. We want to construct an ergodic rigid transformation $T$ with $\uent^{\mu}_{n^t}(T)=u$. For a start, we choose a sequence $(\epsilon_n)_{n\in \N}$ of quickly decreasing positive numbers satisfying
\begin{equation} \label{eq:eps1}
\prod^{\infty}_{n=1}(1-\epsilon_n)>\frac{1}{2}.
\end{equation}
The parameter sequences $(k_m)_{m\in\N}$, $(l_m)_{m\in\N}$, and $(s_m)_{m\in\N}$ are defined inductively. Assume that they have been defined up to $m=n-1$. Notice that this also determines the number $q_n=k_{n-1}l_{n-1}q^2_{n-1}$. Then we apply Proposition \ref{prop:prob} with $\varepsilon=\frac{1}{1000}$, $\gamma=\frac{1}{s^2_{n-1}}$, alphabet $\Sigma=\{0,1,\dots,s_{n-1}-1\}$, and subexponential sequence $\{b_k\}_{k\in \N}=\{k^{q_n+1}\}_{k\in \N}$. According to this we choose $k_n$ as a multiple of $s_{n-1}$ and sufficiently large such that there are at least $\lfloor \frac{8}{\epsilon_n}k^{q_n}_n q^{2u}_n  \rfloor s_{n-1}$ many \emph{good words} satisfying all the properties of Proposition \ref{prop:prob}. Then we set
\begin{equation}\label{eq:s1}
s_n=\lfloor \frac{8}{\epsilon_n}k^{q_n}_n q^{2u}_n \rfloor \cdot s_{n-1}.
\end{equation}
As adumbrated before, $s_n$ is a multiple of $s_{n-1}$ and \ref{item:R1} is satisfied. The collection of good words determines the combinatorics of the conjugation map $h_n$. In particular, property (1) from Proposition \ref{prop:prob} gives \ref{item:R2} and property (2) implies \ref{item:R3}. Thus, $T$ is ergodic by \cite[Theorem 58]{FW1}. Finally, we set
\begin{equation}\label{eq:l1}
l_n=\lfloor \frac{8}{\epsilon_n} \left(k_{n}\right)^{\frac{q_n}{u}} \rfloor,
\end{equation}
which guarantees that $\sum_{n\in \N} \frac{1}{l_n} <\infty$. As a result, we obtain that $T$ is a rigid transformation by Lemma \ref{lem:rig}.

To compute the upper measure-theoretic slow entropy we will use Proposition \ref{prop:generatingSequence} with the generating sequence $\{\eta_m\}_{m\in \N}$ of partitions $\eta_m\coloneqq H_m\left(\xi_{k_mq_m,s_m}\right)$. We start with the following observation.

\begin{lemma}\label{lem:lowerlemma}
For any $n>m$, let $A^{(n)}_{i,j}=\Delta^{i,j}_{k_nq_n,s_n}\cap \Xi_{n+1}$ for any $0\leq i <k_nq_n$ and $0\leq j<s_{n}$. If $P_1\in H_n\left(A^{(n)}_{i_1,j_1}\right)$, $P_2\in H_n\left(A^{(n)}_{i_2,j_2}\right)$ and $j_1\neq j_2$, then $P_1$ and $P_2$ are $\left(\prod^{n-1}_{i=m}\left(1-\frac{4}{s_{i}}-\frac{4}{l_{i}}\right),q_{n+1}\right)$-Hamming apart from each other with respect to the partition $\eta_m$.
\end{lemma}
\begin{proof}

We show this inductively starting with $n=m+1$. To see this we note that for every $t\leq q_{n+1}$ and $P \in H_n(A^{(n)}_{i,j})$ the images $T^t\left(P\right)$ and $T^t_n\left(P\right)$ lie in the same element $\Delta^{i',j'}_{k_nq_n,s_n}$ by Remark \ref{rem:safety}. Then we use that for the map $T_n$ the points from $H_n\left(A^{(n)}_{i_1,j_1}\right)$ and $H_n\left(A^{(n)}_{i_2,j_2}\right)$ with $j_1\neq j_2$ are $\left(1-\frac{4}{s_{n-1}},q_{n+1}\right)$-Hamming apart from each other by the second property of Proposition \ref{prop:prob} describing the combinatorics of $h_n$. This finishes the base case.
	
In the inductive step, we suppose that for some $n>m$ points  $(x_1,y_1) \in H_n\left(A^{(n)}_{i_1,j_1}\right)$ and $(x_2,y_2) \in H_n\left(A^{(n)}_{i_2,j_2}\right)$ with $j_1\neq j_2$ are $\left(\prod^{n-1}_{i=m}\left(1-\frac{4}{s_{i}}-\frac{4}{l_{i}}\right),q_{n+1}\right)$-Hamming apart from each other with respect to the partition $\eta_m$, where the sets $A^{(n)}_{i,j}$ are defined as above for any $0\leq i <k_nq_n$ and $0\leq j<s_{n}$. Then we consider the sets $A^{(n+1)}_{v,w}$ for $0\leq v <k_{n+1}q_{n+1}$ and $0\leq w<s_{n+1}$. Once again, we use the second property of Proposition \ref{prop:prob} to see that points $P^{\prime}_1$ from $A^{(n+1)}_{v_1,w_1}$ and $P^{\prime}_2$ from $A^{(n+1)}_{v_2,w_2}$ with $w_1\neq w_2$ have $\{h_{n+1}\circ R^t_{\alpha_{n+2}}\}_{0\leq t \leq q_{n+2}}$-trajectories that are $1-\frac{4}{s_n}$ apart from each other with respect to the alphabet $\{0,1,\dots,s_n -1\}$ labeling the $j$-coordinate in $A^{(n)}_{i,j}$. Moreover, noticing that for at most $\frac{4q_{n+2}}{l_n}$ many $t\leq q_{n+2}$, we have that $h_{n+1}\circ R^t_{\alpha_{n+2}}(P^{\prime}_1)$ or $h_{n+1}\circ R^t_{\alpha_{n+2}}(P^{\prime}_2)$ do not lie in the safe domain of $h_n$. Along the other iterates we can apply the induction assumption. Altogether these yields that for the map $T_{n+1}$, points in $H_{n+1}\left(A^{(n+1)}_{v_1,w_1}\right)$ and $H_{n+1}\left(A^{(n+1)}_{v_2,w_2}\right)$ with $w_1\neq w_2$ are $\left(\prod^{n}_{i=m}\left(1-\frac{4}{s_{i}}-\frac{4}{l_{i}}\right),q_{n+2}\right)$-Hamming apart from each other with respect to the partition $\eta_m$. Since for every $t\leq q_{n+2}$ and $P\in H_{n+1}\left(A^{(n+1)}_{v,w}\right)$, the images $T^t\left(P\right)$ and $T^t_{n+1}\left(P\right)$ lie in the same element $\Delta^{i',j'}_{k_{n+1}q_{n+1},s_{n+1}}$ by Remark \ref{rem:safety}, this finishes the proof of the induction step.
	
\end{proof}

The previous lemma allows us to prove a lower bound on $S_{\eta_m}(T,q_{n+1},\epsilon)$, where $S_{\eta_m}(T,t,\epsilon)$ is defined in Section \ref{sec:slowentropy}.
\begin{lemma}\label{lem:lower1}
	Let $0<\varepsilon<\frac{1}{2}$ and $m\in \N$, then for $n>m$, we have
	$$S_{\eta_m}\left(T, q_{n+1},\epsilon\right)\geq \frac{s_n}{2}.$$
\end{lemma}
\begin{proof}
Since $\frac{4}{s_i}+\frac{4}{l_{i}}<\epsilon_i$ by (\ref{eq:s1}) and (\ref{eq:l1}), assumption (\ref{eq:eps1}) implies $\prod^{\infty}_{i=m}\left(1-\frac{4}{s_{i}}-\frac{4}{l_{i}}\right)>\varepsilon$. Then Lemma \ref{lem:lowerlemma} yields for $P\in H_{n}(\Xi_{n+1})$ that
$$\mu \left(B_{\eta_m, q_{n+1}}(P,T,\epsilon)\cap H_{n}(\Xi_{n+1})\right)\leq\frac{1}{s_{n}},$$
which together with $\mu(\Xi_{n+1})>1-\delta$ finishes the proof.
\end{proof}

\begin{lemma}\label{lem:upper1}
	Let $\varepsilon>0$ and $m\in \N$, then for $n$ sufficiently large and $\frac{\varepsilon}{2} l_nq_n\leq L<\frac{\varepsilon}{2} l_{n+1}q_{n+1}$,  we have
$$S_{\eta_m}\left(T, L,\epsilon\right)\leq \frac{2}{\varepsilon^2}k_n q_n s_n.$$
\end{lemma}

\begin{proof}
	For any $n\in \N$, $0\leq i_1 <k_n q_n$, and $0\leq i_2 <s_n$ we define the sets
	\[
	\tilde{\Delta}_{i_1,i_2}=\Bigg[ \frac{i_1}{k_nq_n}, \frac{i_1 +1-\frac{\varepsilon}{2}}{k_n q_n} \Bigg) \times \Bigg[\frac{i_2}{s_n}, \frac{i_2 +1}{s_n} \Bigg)
	\]
	and its union
	\[
	\Theta_{n,\varepsilon} \coloneqq \bigcup_{0\leq i_1 <k_n q_n} \bigcup_{0\leq i_2 <s_n}\tilde{\Delta}_{i_1,i_2}.
	\]
	We note that the sets $\tilde{\Delta}_{i_1,i_2}$ are defined in this way to guarantee that for any point $P\in \tilde{\Delta}_{i_1,i_2}$ and all $0\leq t <\frac{\varepsilon}{2} l_nq_n$ the images $T^t_n(P)$ and $T^t_{n-1}(P)$ lie in the same element $\Delta^{i',j'}_{k_nq_n,s_n}$.
	
	Hereby, we define for $0\leq j_1<k_nq_n$, $0\leq j_2 < \lfloor \frac{2}{\varepsilon^2} \rfloor$ and $0\leq j_3<s_n$ the sets
	\[
	B_{j_1,j_2,j_3} = \Bigg[\frac{j_1}{k_nq_n}+\frac{j_2 \cdot \varepsilon^2 }{2k_nq_n},\frac{j_1}{k_nq_n}+\frac{(j_2 +1) \cdot \varepsilon^2 }{2k_nq_n} \Bigg) \times \Bigg[ \frac{j_3}{s_n}, \frac{j_3 +1}{s_n} \Bigg) \cap \Theta_{n+1,\varepsilon} \cap  \Xi_{n+2},
	\]
where the intersection with $\Theta_{n+1,\varepsilon}$ guarantees that $T_{n+1}^t(P)$ and $T_{n}^t(P)$ lie in the same atom for $0\leq t \leq \frac{\epsilon}{2}l_{n+1}q_{n+1}$ and the intersection with $\Xi_{n+2}$ guarantees that $T^t(P)$ and $T^t_{n+1}(P)$ lie in the same atom for $0\leq t\leq q_{n+2}$ by Remark \ref{rem:safety}.

	For $n$ sufficiently large the union of all these sets $B_{j_1,j_2,j_3}$ covers a measure of at least $1-\varepsilon$ of the space. Since points within one set $B_{j_1,j_2,j_3}$ stay $(\varepsilon,L)$-Hamming close for $\frac{\varepsilon}{2} l_{n}q_{n}\leq L<\frac{\varepsilon}{2} l_{n+1}q_{n+1}$, we conclude the statement.
\end{proof}

\begin{remark}\label{rem:IndepCombi}
	We conclude from its proof that the statement of Lemma \ref{lem:upper1} is independent of the combinatorics of the conjugation map $h_n$.
\end{remark}

Altogether we can compute the upper measure-theoretic slow entropy of $T$.
\begin{lemma}
	We have $\uent^{\mu}_{n^t}(T)=u$.
\end{lemma}

\begin{proof}
	On the one hand, we see with the aid of Lemma \ref{lem:lower1} that
	\begin{equation}
\begin{aligned}
	\limsup_{N\to \infty} \frac{S_{\eta_m}\left(T,N,\epsilon\right)}{N^t} &\geq \limsup_{n\to \infty} \frac{S_{\eta_m}\left(T,q_{n+1},\epsilon\right)}{q^t_{n+1}}\geq \lim_{n\to \infty} \frac{\frac{1}{2}s_n}{\left(k_nl_nq^2_n\right)^t}
	\\&\geq \lim_{n\to \infty} \frac{\frac{1}{2}\epsilon^t_n \cdot k^{q_n}_n \cdot q^{2u}_n}{8^t \cdot \left(k_n\right)^{t+q_n\frac{t}{u}}\cdot q^{2t}_n},
\end{aligned}	
\end{equation}
which is positive for all $t<u$. This yields $\uent^{\mu}_{n^t}(T)\geq u$.

On the other hand, Lemma \ref{lem:upper1} implies that
\begin{equation}
\begin{aligned}
\limsup_{L\to \infty} \frac{S_{\eta_m}\left(T,L,\epsilon\right)}{L^t}& \leq \lim_{n\to \infty} \frac{\frac{2}{\varepsilon^2}k_n q_n s_n}{\left(\frac{\varepsilon}{2} l_{n}q_{n}\right)^t}  \leq \lim_{n\to \infty} \frac{16\epsilon^{t-1}_n k^{q_n+1}_n q^{2u+1}_n s_{n-1}}{\varepsilon^{t+2} \left(k_n\right)^{q_n\frac{t}{u}}q^{t}_n},
\end{aligned}
\end{equation}
which is zero for all $t> u$.

Altogether, we obtain $\uent^{\mu}_{n^t}(T)=u$.
\end{proof}

For any given subexponential scale $a_n(t)$ the existence of an ergodic rigid transformation $T$ with $\uent^{\mu}_{a_n(t)}(T)=\infty$ follows from the stronger Theorem \ref{theo:rigidLowerInf} in the next subsection. This finishes the proof of Theorem \ref{theo:rigidUpper}.

\subsection{Proof of Theorem \ref{theo:rigidLower}}

In this section, we present the construction in the two cases: $\lent^{\mu}_{n^t}(T)=u<\infty$ and $\lent^{\mu}_{n^t}(T)=\infty$.

\subsubsection{Case 1: $\lent^{\mu}_{n^t}(T)=u<\infty$}
We start by constructing a rigid transformation with $\lent^{\mu}_{n^t}(T)=u$ for some $0<u<\infty$.

As before, we choose a sequence $(\epsilon_n)_{n\in \N}$ of quickly decreasing positive numbers satisfying
\begin{equation} \label{eq:eps2}
\prod^{\infty}_{n=1}(1-\epsilon_n)>\frac{1}{2}
\end{equation}
and the integer sequences $(k_m)_{m\in\N}$, $(l_m)_{m\in\N}$, and $(s_m)_{m\in\N}$ are defined inductively. Assume that they have been defined up to $m=n-1$ which also determines the number $q_n=k_{n-1}l_{n-1}q^2_{n-1}$. Furthermore, we assume $s_{n-1} \geq 1000$ and $k_{n-1} \geq 4$.

In the next step, we set
\begin{equation} \label{eq:l2}
l_n = \lfloor \frac{6}{\epsilon_n}\cdot (k_{n-1})^{n-1} \rfloor.
\end{equation}
We apply Proposition \ref{prop:prob} with $\epsilon=\frac{1}{1000}$, $\gamma=\frac{1}{s^2_{n-1}}$, alphabet $\Sigma=\{0,1,\dots, s_{n-1} -1\}$, and subexponential sequence ${\mathbf b}=\{b_k\}_{k\in \N} = \{(kk_{n-1})^{nu}\}_{k\in \N}$. Let $K_0$ be the resulting threshold. Then we choose
\begin{equation}\label{eq:k2}
k_n \geq \max\left( K_0 ,\frac{4}{\epsilon_{n}}\right)
\end{equation}
as a multiple of $s_{n-1}$ and such that
\begin{equation}\label{eq:s2}
s_n =\lfloor (k_{n}k_{n-1})^{nu}\rfloor
\end{equation}
is a multiple of $s_{n-1}$. We also denote the corresponding value of $\tau_{{\mathbf b}}$ and $\beta(k_n)$ from Proposition \ref{prop:prob} by $\tau_n$ and $\beta_n$, respectively. Moreover, the collection of words $\Theta$ with cardinality $|\Theta | =s_n$ constructed in Proposition \ref{prop:prob} determines the combinatorics of the conjugation map $h_n$. This finishes the construction in the inductive step.

To estimate the lower measure-theoretic slow entropy via Proposition \ref{prop:generatingSequence}, we make the following observation with regard to the generating sequence $\{\eta_m=H_m\left(\xi_{k_mq_m,s_m}\right)\}_{m\in \N}$.
\begin{lemma}\label{lem:lower2}
	Let $\varepsilon =\frac{1}{1000}$ and $m\in \N$, then for $n>m$, we have
	\[
	S_{\eta_{m}}\left(T,N,\frac{\varepsilon}{2}\right)\geq\begin{cases}
	0.5s_{n-1} & \text{for }q_{n}\leq N<5l_{n}q_{n},\\
	0.5\lfloor (jk_{n-1})^{nu} \rfloor & \text{for }jl_{n}q_{n} \leq N <(j+1)l_nq_n,\text{ where }5\leq j<\beta_{n},\\
	0.5s_{n} & \text{for }\beta_{n}l_{n}q_{n}\leq N<k_{n}l_{n}q_{n}^{2}=q_{n+1}.
	\end{cases}
	\]
\end{lemma}

\begin{proof}
We consider for any $0\leq i<k_nq_n$, $0\leq j<s_{n}$ the set
\[
A^{(n)}_{i,j}=\overline{\Delta}^{i,j}_{k_nq_n,s_n}\cap \Xi_{n+1}.
\]
We recall that for every $t\leq N \leq q_{n+1}$ and $P\in H_n\left(A^{(n)}_{i,j}\right)$ the images $T^t\left(P\right)$ and $T^t_n\left(P\right)$ lie in the same element $\Delta^{i',j'}_{k_nq_n,s_n}$ by Remark \ref{rem:safety}. As in the proof of Lemma \ref{lem:lower1} we show inductively that for every $n>m$ points $P_1 \in H_n\left(A^{(n)}_{i_1,j_1}\right)$ and $P_2\in H_n\left(A^{(n)}_{i_2,j_2}\right)$ with $j_1 \neq j_2$ are $\left(\prod^{n-1}_{i=m}\left(1-\frac{6}{s_{i}}-\frac{6}{l_{i}}\right),q_{n+1}\right)$-Hamming apart from each other for the map $T_n$ with respect to the partition $\eta_m$. Here, we recall that for the map $T_n$ the trajectories from different $H_n\left(A^{(n)}_{i,j}\right)$ are determined by the properties of Proposition \ref{prop:prob} describing the combinatorics of $h_n$. Hence, Proposition \ref{prop:prob} also describes the combinatorics for the different $A^{(n+1)}_{v, w}$ under $\{h_{n+1}\circ R^t_{\alpha_{n+2}}\}_{0\leq t \leq N}$ with respect to the alphabet $\{0,1,\dots,s_{n} -1\}$ labeling the $j$-coordinate in $A^{(n)}_{i,j}$. We also use that for at most $\frac{6N}{l_n}$ many $t\leq N \leq q_{n+2}$ we have that $h_{n+1}\circ R^t_{\alpha_{n+2}}(P_1)$ or $h_{n+1}\circ R^t_{\alpha_{n+2}}(P_2)$ do not lie in the safe domain of $h_n$.
\begin{itemize}
	\item By $s_n \geq 1000$ and Remark \ref{rem:tau} and $\epsilon=\frac{1}{1000}$, we get the condition
	\[
	\frac{(Lk_n)^{(n+1)u}}{(s_n)^{0.99L}}<1 \text{ for every } L\geq \tau_{n+1}
	\]
	for the integer $\tau_{n+1}$ in our application of Proposition \ref{prop:prob}. Since $k_n \geq 4$, we can choose $\tau_{n+1}=5$. Hence, part (4) of Proposition \ref{prop:prob} corresponds to $jl_{n+1}q_{n+1} \leq N <(j+1)l_{n+1}q_{n+1}$ with $5\leq j < \beta_{n+1}$ in our application, while part (3) corresponds to $\beta_{n+1}l_{n+1}q_{n+1}\leq N < \gamma_{n+1}k_{n+1}l_{n+1}q_{n+1}$.
	\item Part (2) of Proposition \ref{prop:prob} yields that for every $\gamma_{n+1}k_{n+1}l_{n+1}q_{n+1}\leq N \leq k_{n+1}l_{n+1}q_{n+1}$ the trajectories of distinct  $H_{n+1}\left(A^{(n+1)}_{v,w}\right)$ under $T_{n+1}$ have Hamming distance $\left(\left(1-\frac{4}{s_{n}}-\frac{6}{l_{n}}\right) \cdot \prod^{n-1}_{i=m}\left(1-\frac{6}{s_{i}}-\frac{6}{l_{i}}\right),N\right)$ from each other with respect to the partition $\eta_m$. Then for $k_{n+1}l_{n+1}q_{n+1}< N \leq k_{n+1}l_{n+1}q^2_{n+1}=q_{n+2}$, the different strings coming from Proposition \ref{prop:prob} are repeated as in the symbolic representation in (\ref{eq:circularOperator}). Since $\gamma_{n+1}=\frac{1}{s^2_n}$, we conclude that for every such $N$ the trajectories are $\left(\prod^{n}_{i=m}\left(1-\frac{6}{s_{i}}-\frac{6}{l_{i}}\right),N\right)$-Hamming apart from each other.
\end{itemize}
As recalled above, we have for every $i\leq N\leq q_{n+2}$ and $P\in H_{n+1}\left(A^{(n+1)}_{i_2,j_2}\right)$  that the images $T^i\left(P\right)$ and $T^i_{n+1}\left(P\right)$ lie in the same element $\Delta^{i',j'}_{k_{n+1}q_{n+1},s_{n+1}}$. Using equations (\ref{eq:l2}), (\ref{eq:k2}), and (\ref{eq:s2}) we also have
\begin{align*}
\prod^{n}_{i=m}\left(1-\frac{6}{s_{i}}-\frac{6}{l_{i}}\right) \geq\prod^{n}_{i=m}\left(1-\epsilon_i\right) >\frac{1}{2}
\end{align*}
by our assumption (\ref{eq:eps2}). Then we conclude the proof as in Lemma \ref{lem:lower1}.
\end{proof}	

Hereby, we can compute the lower measure-theoretic polynomial entropy of $T$.
\begin{lemma} \label{lem:rigidLowerExact}
	We have $\lent^{\mu}_{n^t}(T)=u$.
\end{lemma}

\begin{proof}
	On the one hand, we investigate
$\liminf_{N\to \infty}\frac{S_{\eta_m}\left(T,N,\varepsilon\right)}{N^t}$
for $\varepsilon \leq \frac{1}{2000}$ by examining the three different cases in the statement of Lemma \ref{lem:lower2}. In case of $q_n \leq N<5l_nq_n$ we see
$$ \frac{S_{\eta_m}\left(T,N,\varepsilon\right)}{N^t}\geq \frac{0.5s_{n-1}}{\left(5l_nq_n\right)^t} \geq  \frac{0.5\cdot (k_{n-1})^{(n-1)u}}{\left(\frac{30}{\epsilon_n} \cdot (k_{n-1})^{n} \cdot l_{n-1} \cdot q^2_{n-1}\right)^t},$$
whose limit is positive for all $t<u$. Similarly, for $jl_nq_n\leq N< (j+1)l_nq_n$ with $5\leq j < \beta_n$ we have
$$	\frac{S_{\eta_m}\left(T,N,\varepsilon\right)}{N^t}\geq \frac{0.25(jk_{n-1})^{nu}}{\left((j+1) l_nq_n\right)^t} \geq \frac{0.25 \cdot j^{u} \cdot (k_{n-1})^{nu}} { \left((j+1)\cdot \frac{6}{\epsilon_n} \cdot (k_{n-1})^{n}\cdot l_{n-1} \cdot q^2_{n-1}\right)^t},$$
whose limit is also positive for all $t<u$. Clearly, in the remaining case $\beta_n l_nq_n \leq N <q_{n+1}$ the limit is positive for all $t<u$ as well. Altogether, this yields $\lent^{\mu}_{n^t}(T)\geq u$.
	
	On the other hand, as observed in Remark \ref{rem:IndepCombi} we can apply Lemma \ref{lem:upper1} again, which yields
\begin{equation}
\begin{aligned}
\liminf_{L\to \infty} \frac{S_{\eta_m}\left(T,L,\varepsilon\right)}{L^t} \leq \lim_{n\to \infty} \frac{\frac{2}{\varepsilon^2}k_n q_n s_n}{\left(\frac{\varepsilon}{2} l_{n+1}q_{n+1}\right)^t}  \leq \lim_{n\to \infty} \frac{2 \cdot (k_{n})^{nu+1} \cdot (k_{n-1})^{nu}  \cdot q_n}{\varepsilon^{t+2} \cdot\left( (k_n)^{n+1} \cdot (k_{n-1})^{n-1} \cdot q^2_n\right)^t}.
\end{aligned}
\end{equation}	
This is zero for all $t> u$.
	
	Altogether, we obtain $\lent^{\mu}_{n^t}(T)=u$.
\end{proof}

This finishes the proof of Theorem \ref{theo:rigidLower} for $0<u<\infty$.

\subsubsection{Case 2: $\lent^{\mu}_{n^t}(T)=\infty$}
In order to cover the case $\lent^{\mu}_{n^t}(T)=\infty$, we prove the following stronger statement.

\begin{theorem} \label{theo:rigidLowerInf}
	For any scaling function $\{a_n(t)\}_{n\in\mathbb{N},t>0}$ satisfying $\lim_{n\to\infty}\frac{\log a_n(t)}{n}=0$, there exists an ergodic Lebesgue measure preserving rigid transformation $T$ with $\lent_{a_n(t)}^{\mu}(T)=\infty$.
\end{theorem}

\begin{proof}
	It suffices to prove the theorem for any scaling function $\{a_n(t)\}_{n\in\mathbb{N},t>0}$ with at least polynomial growth. In particular, for any given $t_1 >0$, there exists $t_2>t_1$ such that $a_n(t_2)>a_{n^2}(t_1)$ for all $n\in \N$. To build a rigid ergodic transformation $T$ with $\lent^{\mu}_{a_n(t)}(T)=\infty$, we follow the construction above but with an additional inductively defined increasing sequence $(u_n)_{n\in \N}$ of integers. This time, we define the sequence $(l_n)_{n\in \N}$ in advance as $l_n=n^2$. This choice guarantees $\sum^{\infty}_{n=1} \frac{1}{l_n} <\infty$ and, hence, rigidity of $T$ by Lemma \ref{lem:rig}. For the inductive choice of the other parameters we proceed as follows. We assume that $u_{n} \in \Z^+$ is large enough such that $a_{m}(u_{n}) \geq a_{m^2}(u_{n-1})$ for all $m\in \N$,
	\begin{equation*}
	s_{n-1} = \lfloor a_{k_{n-1}l_{n}q_{n-1}}(u_{n})\rfloor,
	\end{equation*}
	and that (recalling the notation from Remark \ref{rem:kForCLT})
	\begin{equation*}
		k_{n-1} \geq K_0\left( \{a_{(j+1)l_nq_{n-1}}(u_n)\}_{j\in \N},s_{n-2}, \frac{1}{1000}, \frac{1}{s^2_{n-2}}\right)
	\end{equation*}
    is sufficiently large such that for all $j\geq \sqrt{k_{n-1}}$, we have
    \begin{equation*}
    	\left(a_{k_{n-1}l_nq_{n-1}}(u_{n}) \right)^{0.99j} > a_{(j+1)\sqrt{k_{n-1}}l_n q_{n-1}}(u_{n}).
    \end{equation*}
	Note that this implies
	\begin{equation} \label{eq:FastGrowth}
		\left(a_{k_{n-1}l_nq_{n-1}}(u_{n}) \right)^{0.99j} > a_{(j+1)\sqrt{k_{n-1}}l_n q_{n-1}}(u_{n})
	\end{equation}
    for all $j\geq 2$.
	
	We start the inductive step by choosing an $u_{n+1}>u_{n}$ such that $a_m(u_{n+1}) \geq a_{m^2}(u_n)$ for all $m\in \N$. To determine the combinatorics of $h_n$ we apply Proposition \ref{prop:prob} and the observation in Remark \ref{rem:smallerGrowth} with alphabet $\Sigma=\{0, 1, \dots, s_{n-1} -1\}$, $\varepsilon=\frac{1}{1000}$, $\gamma=\frac{1}{s^2_{n-1}}$ and subexponential sequences $\{b_j\}_{j\in \N} = \{ a_{(j+1)l_{n+1}q_{n}}(u_{n+1})\}_{j\in \N}$ and $\{c_j\}_{j\in \N} = \{ a_{(j+1)\sqrt{k_{n-1}}l_{n}q_{n-1}}(u_{n})\}_{j\in \N}$. Let $L_0$ be the resulting threshold. Then we choose
	\begin{equation*}
	k_n \geq \max\left( L_0 ,\frac{4}{\epsilon_{n}}\right)
	\end{equation*}
	as a multiple of $s_{n-1}$, such that
	\begin{equation*}
		\left(a_{k_{n}l_{n+1}q_{n}}(u_{n+1}) \right)^{0.99j} > a_{(j+1) \sqrt{k_{n}}l_{n+1}q_{n}}(u_{n+1})
	\end{equation*}	
	for all $j\geq \sqrt{k_n}$, and such that
	\begin{equation*}
	s_n \coloneqq \lfloor a_{k_n l_{n+1}q_n}(u_{n+1})\rfloor
	\end{equation*}
	is a multiple of $s_{n-1}$. Hereby, the induction assumptions are satisfied.

	For this construction we investigate
	$\liminf_{N\to \infty}\frac{S_{\eta_m}\left(T,N,\varepsilon\right)}{a_N(t)}	$
	as in the proof of Lemma \ref{lem:rigidLowerExact}. We note that $a_{k_n l_{n+1}q_n}(u_{n+1}) \geq a_{k^2_n l^2_{n+1}q^2_n}(u_{n}) \geq a_{k_n l_{n+1}q_{n+1}}(u_{n})$ for all $n\in \N$. Furthermore, \eqref{eq:FastGrowth} implies that according to Remark \ref{rem:smallerGrowth} we can take $\tau_{{\mathbf c}}=2$. In case of $q_n \leq N<2l_nq_n$ we see
	$$ \frac{S_{\eta_m}\left(T,N,\varepsilon\right)}{a_N(t)}\geq \frac{0.5s_{n-1}}{a_{2l_nq_n}(t)} \geq  \frac{0.25a_{k_{n-1}l_{n}q_{n-1}}(u_{n})}{a_{2l_nq_n}(t)} \geq  \frac{0.25a_{k_{n-1}l_n q_{n}}(u_{n-1})}{a_{2l_nq_n}(t)}  ,$$
	whose limit is positive for all $t\leq u_{n-1}$. Similarly, for $jl_nq_n\leq N< (j+1) l_nq_n$, $2\leq j <\beta_n$, we have
	$$	\frac{S_{\eta_m}\left(T,N,\varepsilon\right)}{a_N(t)}\geq \frac{0.5a_{(j+1)\sqrt{k_{n-1}}l_nq_{n-1}}(u_n)}{a_{(j+1) l_nq_n}(t)}\geq \frac{0.5a_{(j+1)^2l_nq_{n}}(u_{n-1})}{a_{(j+1) l_nq_n}(t)},$$
	whose limit is also positive for all $t\leq u_{n-1}$. In the remaining case $\beta_n l_nq_n \leq N <q_{n+1}$ we have
	$$\frac{S_{\eta_m}\left(T,N,\varepsilon\right)}{a_N(t)}\geq \frac{0.5s_{n}}{a_{q_{n+1}}(t)} \geq  \frac{0.25a_{k_n l_{n+1}q_n}(u_{n+1})}{a_{k_{n}l_nq^2_n}(t)} >  \frac{0.25a_{k_n l_{n+1}q^2_n}(u_{n})}{a_{k_{n}l_nq^2_n}(t)} $$
	and this limit is positive for all $t\leq u_{n}$. Since $u_n \to \infty$ as $n\to \infty$, this yields $\lent^{\mu}_{a_n(t)}(T)=\infty$.
\end{proof}

\section{Slow entropy for good cyclic approximations} \label{sec:goodCyclic}
The previous section already provides examples of measure preserving transformations admitting a periodic approximation with positive or even infinite lower measure-theoretic polynomial entropy. In this section we examine one particular class of periodic approximations, the so-called good cyclic approximations. We recall their definition from Definition \ref{def:cyclicApp}.

On the one hand, we know that the lower measure-theoretic polynomial entropy for a transformation with good cyclic approximation is at most one following from estimates in Proposition $5$ in \cite{Fe} and Proposition 1.3 in \cite{Kanigowski}:
\begin{proposition}\label{prop:upperBoundLowerCyclic}
If $T:(X,\mu)\to (X,\mu)$ admits a good cyclic approximation, then the lower measure-theoretic slow entropy with respect to the polynomial scale $a_n(t)=n^t$ is at most $1$.
\end{proposition}
In fact, we show that the above estimate is optimal by Theorem \ref{theo:cyclicLower}:  every possible value in $[0,1]$ is attained by a transformation admitting a good cyclic approximation.

On the other hand, we are able to construct transformations admitting a good cyclic approximation with positive or even infinite upper measure-theoretic polynomial entropy in Theorem \ref{theo:cyclicUpper}, which exhibits different features of upper and lower measure-theoretic slow entropy.


As for the proof Theorem \ref{theo:cyclicUpper} and Theorem \ref{theo:cyclicLower}, both constructions are based on a \emph{twisted} abstract AbC method. For this purpose, we introduce further notation: For $0\leq i<q$, $0\leq j <k$ we define the ``stripes''
\begin{equation*}
\Gamma^{i,j}_{k,q} \coloneqq \Bigg[\frac{i}{q}+\frac{j}{kq}, \frac{i}{q}+\frac{j+1}{kq} \Bigg) \times \left[0,1 \right].
\end{equation*}
Hereby, we introduce the partition
\[
\zeta_{k,q} \coloneqq \Meng{\Gamma^{i,j}_{k,q}}{0\leq i<q, \ 0\leq j<k}.
\]
Under the assumption that $k$ is a multiple of $s$ we also use the following rectangles for $0\leq i<q$, $0\leq j< \frac{k}{s}$, $0\leq t <s$:
\[
R^{i,j,t}_{q,k,s} \coloneqq \Bigg[\frac{i}{q}+\frac{j\cdot s}{kq}, \frac{i}{q}+\frac{(j+1)\cdot s}{kq}\Bigg) \times \Bigg[\frac{t}{s}, \frac{t+1}{s}\Bigg) =\bigcup_{0\leq u<s} \Delta^{ik+js+u,t}_{kq,s},
\]
where $\Delta_{q,s}^{i,j}$ is defined in Section \ref{subsec:combin}. Finally notice that $\mu(\Gamma^{i,j}_{k,q})=\mu(R^{i,j,t}_{q,k,s})$.

\subsection{Proof of Theorem \ref{theo:cyclicUpper}}
For a start, we let $0<u<\infty$. We choose a sequence $(\epsilon_n)_{n\in \N}$ of quickly decreasing positive numbers satisfying
\begin{equation} \label{eq:eps3}
\prod^{\infty}_{n=1}(1-\epsilon_n)>\frac{1}{2}
\end{equation}
and this time we fix $(s_n)_{n\in\N}$ as a strictly increasing sequence of positive integers from the beginning, where $s_n$ is a multiple of $s_{n-1}$. This will suffice to guarantee that $\{\eta_m=H_m\left(\xi_{k_mq_m,s_m}\right)\}_{m\in \N}$ is a generating sequence. Besides that, we follow our inductive scheme again. Assume that the parameter sequences $(k_m)_{m\in\N}$ and $(l_m)_{m\in\N}$ have been defined up to $m=n-1$, which also determines the number $q_n=k_{n-1}l_{n-1}q^2_{n-1}$. Moreover, we assume that $k_{n-1}$ is of the form $k_{n-1}=r_{n-1} \cdot \lfloor (r_{n-1})^{q_{n-1} u} \rfloor$ with some $r_{n-1}\in \Z^+$.
Then we apply Proposition \ref{prop:prob} with subexponential sequence $\{b_j\}_{j\in \N} = \{j^{uq_n}\}_{j\in \N}$,
\[
\varepsilon = \frac{1}{1000}, \ \gamma= \frac{r^2_{n-1}}{k^2_{n-1}},  \ \text{ and alphabet } \Sigma =\left\{0,1,\dots , \frac{k_{n-1}}{r_{n-1}}-1\right\}.
\]
Accordingly, we can choose a number $r_n \in \N$ as a multiple of $k_{n-1}$ as well as $s_n$,  sufficiently large such that $r_n \geq \frac{12}{\epsilon_n}$ and there is a collection $\Theta \subset \Sigma^{r_n}$ with $|\Theta |= \lfloor (r_n)^{q_n u} \rfloor$ of \emph{good words} of length $r_n$ satisfying properties (1) and (2) in Proposition \ref{prop:prob}. In particular, property (2) yields for all $w,w' \in \Theta$ that
\begin{equation}\label{Hamming3}
d^H(w,w')\geq 1-\frac{2r_{n-1}}{k_{n-1}}.
\end{equation}
Then we concatenate all these words from $\Theta$ to obtain one word of length
\begin{equation} \label{eq:k3}
k_n \coloneqq r_n \cdot \lfloor (r_n)^{q_n u} \rfloor.
\end{equation}
This word $w=w_0 \dots w_{k_n-1}$ determines an assignment
\begin{equation}\label{eq:assignment1}
\psi_n:\{0,1,\dots ,k_n -1\} \to \left\{0,1,\dots ,\frac{k_{n-1}}{r_{n-1}} -1\right\}, \ \psi_n(j)=w_j.
\end{equation}
With it $h_n$ will be of the form
\begin{equation}
\begin{aligned}
& h_n\left(\Gamma^{i,j}_{k_n,q_n}\right) \\
= & R^{i+\psi_n(j)r_{n-1} l_{n-1}q_{n-1}, \ \lfloor \frac{j}{s_n} \rfloor ,\ j \mod s_n}_{q_n,k_n,s_n} \\
= &  \Bigg[\frac{\psi_n(j) \cdot r_{n-1}}{k_{n-1}q_{n-1}}+\frac{i}{q_n}+\frac{\lfloor \frac{j}{s_n} \rfloor \cdot s_n}{k_nq_n}, \frac{\psi_n(j) \cdot r_{n-1}}{k_{n-1}q_{n-1}}+\frac{i}{q_n}+\frac{\left(\lfloor \frac{j}{s_n} \rfloor+1\right)\cdot s_n}{k_nq_n}\Bigg)\\& \times \Bigg[\frac{j \mod s_n}{s_n}, \frac{(j \mod s_n) +1}{s_n}\Bigg),
\end{aligned}
\end{equation}
i.e. descriptively speaking long stripes are mapped to rectangles. This will guarantee that the sequence of partitions $\{H_n(\zeta_{k_n,q_n})\}_{n\in \N}$ is generating. We also note that the conjugation map $h_n$ of this form satisfies the required relation
\[
h_n \circ R_{\frac{1}{q_n}} = R_{\frac{1}{q_n}} \circ h_n.
\]

To finalize the construction we set
\begin{equation}\label{eq:l3}
l_n= (r_n)^{q_n}.
\end{equation}

The combinatorics allows us to make the following observation with regard to the generating sequence $\{\eta_m=H_m\left(\xi_{k_mq_m,s_m}\right)\}_{m\in \N}$. It will prove to be useful when estimating the upper measure-theoretic slow entropy via Proposition \ref{prop:generatingSequence}:
\begin{lemma}\label{lem:lower3}
	Let $0<\varepsilon<\frac{1}{2}$ and $m\in \N$. For $n>m$ we have
$$S_{\eta_m}\left(T,r_nl_nq_n,\varepsilon\right)\geq 0.5 \lfloor (r_n)^{q_n u}\rfloor .$$
\end{lemma}

\begin{proof}
	For any $0\leq i<q_n$ and $0\leq j<k_{n}$ we define the set
	\begin{equation*}
	\overline{\Gamma}^{i,j}_{k_n,q_n} \coloneqq \Gamma^{i,j}_{k_n,q_n} \cap \Xi_{n+1}.
	\end{equation*}
	Then we note that for every $t\leq q_{n+1}$ and $P \in H_n\left(\overline{\Gamma}^{i,j}_{k_n,q_n}\right)$  the images $T^t\left(P\right)$ and $T^t_n\left(P\right)$ lie in the same element $R^{i',j'}_{k_nq_n,s_n}$ by Remark \ref{rem:safety}. The proof of our lemma follows from the following three claims:
\begin{claim}\label{claim:1}
Let $0\leq i_1,i_2<q_{m+1}$, $0\leq j_1,j_2 < \frac{k_{m+1}}{r_{m+1}}$, and $0\leq v_1,v_2<r_{m+1}$, then the points $P_1 \in H_{m+1}\left(\overline{\Gamma}^{i_1,j_1r_{m+1}+v_1}_{k_{m+1},q_{m+1}}\right)$ and $P_2 \in H_{m+1}\left(\overline{\Gamma}^{i_2,j_2r_{m+1}+v_2}_{k_{m+1},q_{m+1}}\right)$ with $j_1 \neq j_2$ are $\left(1-\frac{3r_m}{k_{m}},q_{m+2}\right)$-Hamming apart from each other with respect to the partition $\eta_m$ and the map $T$.
\end{claim}
	
\begin{proof}
Since $q_{m+2}= k_{m+1} l_{m+1} q^2_{m+1}$, the $\{h_{m+1}\circ R^t_{\alpha_{m+2}}\}_{t\leq q_{m+2}}$ trajectory of a point $ P \in \overline{\Gamma}^{i,jr_{m+1}+v}_{k_{m+1},q_{m+1}}$ corresponds to $q_{m+1}$ times repeating a string of length $k_{m+1}=r_{m+1} \cdot \lfloor (r_{m+1})^{q_{m+1} u} \rfloor$ in the alphabet $\Sigma=\{0,1,\dots ,\frac{k_{m}}{r_{m}} -1\}$ describing the assignment $\psi_{m+1}$ (see \eqref{eq:assignment1}).

We divide this string into the words of length $r_{m+1}$ obtained by Proposition \ref{prop:prob}. Under our assumption $j_1 \neq j_2$, the aligned words are different from each other and we can exploit property (2) in Proposition \ref{prop:prob} to see that with frequency at most $\frac{3r_{m}}{k_{m}}$ our stripes $\overline{\Gamma}^{i_1,j_1r_{m+1}+v_1}_{k_{m+1},q_{m+1}}$ and $\overline{\Gamma}^{i_2,j_2r_{m+1}+v_2}_{k_{m+1},q_{m+1}}$ are mapped to the same $\frac{r_{m}}{k_{m}q_{m}}$-domain $\Big[\frac{ir_{m}}{k_{m}q_{m}},\frac{(i+1) r_{m}}{k_{m}q_{m}}\Big)\times [0,1]$, $i\in \mathbb{N}$, under $h_{m+1}\circ R^t_{\alpha_{m+2}}$, $t\leq q_{m+2}$. Hence, $P_1 \in H_{m+1}\left(\overline{\Gamma}^{i_1,j_1r_{m+1}+v_1}_{k_{m+1},q_{m+1}}\right)$ and $P_2 \in H_{m+1}\left(\overline{\Gamma}^{i_2,j_2r_{m+1}+v_2}_{k_{m+1},q_{m+1}}\right)$ with $j_1 \neq j_2$ are $\left(1-\frac{3r_{m}}{k_{m}},q_{m+2}\right)$-Hamming apart from each other with respect to the partition $\eta_m=H_m\left(\xi_{k_{m}q_{m},s_{m}}\right)$ and the map $T_{m+1}$. By our observation above this also holds true for the map $T$.
\end{proof}

\begin{claim}\label{claim:2}
Let $n>m$, $0\leq i_1,i_2<q_n$, $0\leq j_1,j_2 < \frac{k_n}{r_n}$, and $0\leq v_1,v_2<r_n$. Then the points $P_1 \in H_n\left(\overline{\Gamma}^{i_1,j_1r_n+v_1}_{k_n,q_n}\right)$ and $P_2 \in H_n\left(\overline{\Gamma}^{i_2,j_2r_n+v_2}_{k_n,q_n}\right)$ with $j_1 \neq j_2$ are $\left(\prod^{n-1}_{i=m}\left(1-\frac{3r_i}{k_{i}}-\frac{6}{l_i}\right),q_{n+1}\right)$-Hamming apart from each other with respect to the partition $\eta_m$ and the map $T$.
\end{claim}
	
\begin{proof}
We show this by induction. The base clause $n=m+1$ follows from Claim \ref{claim:1}. In the inductive step, we suppose that for some $n>m$ points from $H_n\left(\overline{\Gamma}^{i_1,j_1r_n+v_1}_{k_n,q_n}\right)$ and $H_n\left(\overline{\Gamma}^{i_2,j_2r_n+v_2}_{k_n,q_n}\right)$, respectively, with $j_1 \neq j_2$  are $\left(\prod^{n-1}_{i=m}\left(1-\frac{3r_i}{k_{i}}-\frac{6}{l_i}\right),q_{n+1}\right)$-Hamming apart from each other with respect to the partition $\eta_m$. Then we consider $P^{\prime}_1 \in A_1 \coloneqq \overline{\Gamma}^{i^{\prime}_1,j^{\prime}_1r_{n+1}+v^{\prime}_1}_{k_{n+1},q_{n+1}}$ and $P^{\prime}_2 \in A_2 \coloneqq \overline{\Gamma}^{i^{\prime}_2,j^{\prime}_2r_{n+1}+v^{\prime}_2}_{k_{n+1},q_{n+1}}$ with $j^{\prime}_1 \neq j^{\prime}_2$.

Then for at most $\frac{6q_{n+2}}{l_n}$ many $t\leq q_{n+2}$, we have that $h_{n+1}\circ R^t_{\alpha_{n+2}}(P^{\prime}_1)$ or $h_{n+1}\circ R^t_{\alpha_{n+2}}(P^{\prime}_2)$ do not lie in the safe domain of $h_n$. On the remaining iterates we use the second property in Proposition \ref{prop:prob} again and follow the argument from Claim \ref{claim:1} to see that with frequency at most $\frac{3r_{n}}{k_{n}}$ these stripes are mapped into the same $\frac{r_n}{k_{n}q_{n}}$-domain under $h_{n+1}\circ R^t_{\alpha_{n+2}}$, $t\leq q_{n+2}$.  By the induction assumption this yields that for the map $T_{n+1}$, the points from $H_{n+1}\left(A_1\right)$ and $H_{n+1}\left(A_2 \right)$, respectively, are $\left(\prod^{n}_{i=m}\left(1-\frac{3r_i}{k_{i}}-\frac{6}{l_i}\right),q_{n+2}\right)$-Hamming apart from each other with respect to the partition $\eta_m$. By the initial observation this also holds true under the map $T$.
\end{proof}

\begin{claim}\label{claim:3}
Let $n>m$, $0\leq i_1,i_2<q_n$, $0\leq j_1,j_2 < \frac{k_n}{r_n}$, and $0\leq v_1,v_2<r_n$. Then the points $P_1 \in H_n\left(\overline{\Gamma}^{i_1,j_1r_n+v_1}_{k_n,q_n}\right)$ and $P_2 \in H_n\left(\overline{\Gamma}^{i_2,j_2r_n+v_2}_{k_n,q_n}\right)$  with $j_1 \neq j_2$ are  $$\left(\prod^{n-1}_{i=m}\left(1-\frac{3r_i}{k_{i}}-\frac{6}{l_i}\right),r_nl_nq_n\right)$$-Hamming apart from each other with respect to the partition $\eta_m$.
\end{claim}	
\begin{proof}
By the second property of Proposition \ref{prop:prob} describing the assignment $\psi_n$ and, hence, the combinatorics of $h_n$ we have that for  $\overline{\Gamma}^{i_1,j_1r_n+v_1}_{k_n,q_n}$ and $\overline{\Gamma}^{i_2,j_2r_n+v_2}_{k_n,q_n}$ with $j_1 \neq j_2$ a proportion of at most $\frac{3r_{n-1}}{k_{n-1}}$ of the iterates $h_{n}\circ R^t_{\alpha_{n+1}}$, $t\leq r_nl_nq_n$, is mapped into the same $\frac{r_{n-1}}{k_{n-1}q_{n-1}}$ domain. Using Claim \ref{claim:2} for $n-1$, this yields that for the map $T_n$ the points $P_1 \in H_n\left(\overline{\Gamma}^{i_1,j_1r_n+v_1}_{k_n,q_n}\right)$ and $P_2 \in H_n\left(\overline{\Gamma}^{i_2,j_2r_n+v_2}_{k_n,q_n}\right)$  with $j_1 \neq j_2$ are $\left(\prod^{n-1}_{i=m}\left(1-\frac{3r_i}{k_{i}}-\frac{6}{l_i}\right),r_nl_nq_n\right)$-Hamming apart from each other. By the initial observation this also holds true under the map $T$ since $r_nl_nq_n <q_{n+1}$.
\end{proof}
	
Additionally, we note that
\begin{equation}\label{eq:prodEpsilon}
\prod^{\infty}_{i=m}\left(1-\frac{3r_i}{k_{i}}-\frac{6}{l_i}\right)>\prod^{\infty}_{i=m}\left(1-\epsilon_i\right)>\varepsilon
\end{equation} by assumption (\ref{eq:eps3}). Then we complete the proof by the same method as in Lemma \ref{lem:lower1} using Claim \ref{claim:3} and estimate \eqref{eq:prodEpsilon}.
\end{proof}

Then we also check that our limit transformation $T=\lim_{n\to \infty} T_n$ admits a good cyclic approximation.
\begin{lemma}\label{lem:goodcyclic}
	Suppose $\sum_{n\in \N}\frac{1}{l_n} < \infty$ and that the conjugation map $h_n$ is of the form $h_n\left(\Gamma^{i,j}_{k_n,q_n}\right)
	= R^{i+\psi_n(j)r_{n-1} l_{n-1}q_{n-1}, \ \lfloor \frac{j}{s_n} \rfloor ,\ j \mod s_n}_{q_n,k_n,s_n}$. Then the limit transformation $T=\lim_{n\to \infty} T_n$ admits a good cyclic approximation.
\end{lemma}

\begin{proof}
	We recall that under the given form of the conjugation map $h_n$ the sequence of partitions $\{H_n(\zeta_{k_nl_nq_n,q_n})\}_{n\in \N}$ is generating. We compute the weak distance between $T$ and $T_n$ with respect to the partition $H_n(\zeta_{k_nl_nq_n,q_n})$:
	\begin{align*}
	q_{n+1} \cdot d\left(H_n(\zeta_{k_nl_nq_n,q_n}),T,T_n\right) & = q_{n+1} \cdot \sum_{c\in \zeta_{k_nl_nq_n,q_n}} \mu\left(T(H_n(c))\triangle T_n(H_n(c))\right) \\
	& \leq q_{n+1} \cdot \sum_{i=n+1}^{\infty}\sum_{\tilde{c}\in \xi_{k_iq_i,s_i}} \mu\left(R_{\alpha_{i+1}-\alpha_i}(\tilde{c})\triangle \tilde{c}\right) \\
	& \leq q_{n+1} \cdot \sum_{i=n+1}^{\infty}\frac{1}{l_{i}q_{i}} < \sum_{i=n+1}^{\infty}\frac{1}{l_{i}}
	\end{align*}
	Since this goes to $0$ as $n\to \infty$, this suffices to conclude the weak convergence of the cyclic approximation $(H_n(\zeta_{k_nl_nq_n,q_n}),T_n)$ to the limit transformation $T$. We also obtain that the speed of convergence is of the order $o(\frac{1}{q})$, i.e. we have a good cyclic approximation.
\end{proof}

\begin{remark}\label{rem:cyclicCrit}
	We note that proof and conclusion of Lemma \ref{lem:goodcyclic} do not depend on the exact combinatorics of the assignment $\psi_n:\{0,1,\dots ,k_n -1\} \to \left\{0,1,\dots ,\frac{k_{n-1}}{r_{n-1}} -1\right\}$.
\end{remark}

Now we are ready to compute the upper slow entropy

\begin{lemma}
	We have $\uent^{\mu}_{n^t}(T)=u$.
\end{lemma}	

\begin{proof}
	On the one hand, we see with the aid of Lemma \ref{lem:lower3} that
	\begin{equation*}
	\limsup_{N\to \infty} \frac{S_{\eta_m}\left(T,N,\varepsilon\right)}{N^t} \geq \limsup_{n\to \infty} \frac{S_{\eta_m}\left(T,r_nl_nq_n,\varepsilon\right)}{\left(r_nl_nq_n\right)^t}  \geq \lim_{n\to \infty} \frac{0.5\lfloor (r_n)^{q_n u}\rfloor}{\left((r_n)^{q_n +1}\cdot q_n\right)^t},
	\end{equation*}
	which is positive for all $t<u$. This yields $\uent^{\mu}_{n^t}(T)\geq u$.
	
	On the other hand, Lemma \ref{lem:upper1} still holds true since its proof was independent of the combinatorics of $h_n$ as observed in Remark \ref{rem:IndepCombi}. This implies that
\begin{equation}
\begin{aligned}
\limsup_{L\to \infty} \frac{S_{\eta_m}\left(T,L,\varepsilon\right)}{L^t} \leq \lim_{n\to \infty} \frac{\frac{2}{\varepsilon^2}k_n q_n s_n}{\left(\frac{\varepsilon}{2} l_{n}q_{n}\right)^t}  \leq \lim_{n\to \infty} \frac{2^{t+1}\cdot \lfloor (r_n)^{q_nu} \rfloor \cdot r_n \cdot q_n \cdot s_n}{\varepsilon^{t+2} \cdot (r_n)^{q_nt} \cdot q^{t}_n},
\end{aligned}
\end{equation}
which is zero for all $t> u$. Altogether, we obtain $\uent^{\mu}_{n^t}(T)=u$.
\end{proof}

The case of a transformation with good cyclic approximation and $\uent^{\mu}_{n^t}(T)=\infty$ is covered by the subsequent more general theorem.

\begin{theorem} \label{theo:CyclicInfinity}
	For any scaling function $\{a_n(t)\}_{n\in\mathbb{N},t>0}$ satisfying $\lim_{n\to+\infty}\frac{\log a_n(t)}{n}=0$, there exists a Lebesgue measure preserving  transformation $T$ with good cyclic approximation and $\uent_{a_n(t)}^{\mu}(T)=\infty$.
\end{theorem}

\begin{proof}
	As before, we fix $(s_n)_{n\in\N}$ as a strictly increasing sequence of positive integers from the beginning, where $s_n$ is a multiple of $s_{n-1}$. We also set $l_n=n^2$ which yields that the limit transformation admits a good cyclic approximation by Lemma \ref{lem:goodcyclic}. Additionally, we fix a strictly increasing sequence $(u_n)_{n\in \N}$ of positive integers. Besides that, we follow our inductive scheme again.
	
	Assume that the parameter sequence $(k_m)_{m\in\N}$ has been defined up to $m=n-1$, which also determines the number $q_n=k_{n-1}l_{n-1}q^2_{n-1}$. Moreover, we assume that $k_{n-1}$ is of the form $k_{n-1}=r_{n-1} \cdot \lfloor a_{r_{n-1}l_{n-1}q_{n-1}}(u_{n-1}) \rfloor$ with some $r_{n-1}\in \Z^+$. Then we apply Proposition \ref{prop:prob} with subexponential sequence $\{b_j\}_{j\in \N}=\{a_{jl_{n}q_{n}}(u_n)\}_{j\in \N}$,
	\[
	\varepsilon = \frac{1}{1000}, \ \gamma= \frac{r^2_{n-1}}{k^2_{n-1}}, \ \text{ and alphabet } \Sigma =\left\{0,1,\dots , \frac{k_{n-1}}{r_{n-1}}-1\right\}.
	\]
	Accordingly, we can choose a number $r_n \in \N$ as a multiple of $k_{n-1}$ as well as $s_n$ and sufficiently large such that $r_n \geq \frac{6}{\epsilon_n}$ and there is a collection $\Theta \subset \Sigma^{r_n}$ with $|\Theta |= \lfloor a_{r_nl_{n}q_{n}}(u_n) \rfloor$ of \emph{good words} of length $r_n$ satisfying properties (1) and (2) in Proposition \ref{prop:prob}. Then we concatenate all these words from $\Theta$ to obtain one word of length
	\begin{equation}
	k_n \coloneqq r_n \cdot \lfloor a_{r_nl_{n}q_{n}}(u_n) \rfloor.
	\end{equation}
	This accomplishes the inductive step.
	
	Finally, we compute with the aid of Lemma \ref{lem:lower3} that
	\begin{equation*}
	\limsup_{N\to \infty} \frac{S_{\eta_m}\left(T,N,\varepsilon\right)}{a_N(t)} \geq \limsup_{n\to \infty} \frac{S_{\eta_m}\left(T,r_nl_nq_n,\varepsilon\right)}{a_{r_nl_nq_n}(t)}  \geq \lim_{n\to \infty} \frac{0.5 \lfloor a_{r_nl_{n}q_{n}}(u_n)\rfloor}{a_{r_nl_nq_n}(t)},
	\end{equation*}
	which is positive for all $t \leq u_n$. Since $u_n \to \infty$ as $n\to \infty$, this yields $\uent^{\mu}_{a_n(t)}(T) = \infty$.
\end{proof}

\subsection{Proof of Theorem \ref{theo:cyclicLower}}
We fix $0<u\leq 1$ and choose a sequence $(\epsilon_n)_{n\in \N}$ of quickly decreasing positive numbers satisfying (\ref{eq:eps3}). Moreover, we choose a strictly increasing sequence $(s_n)_{n\in\N}$ of positive integers, where $s_n$ is a multiple of $s_{n-1}$. The parameter sequences $(k_m)_{m\in\N}$ and $(l_m)_{m\in\N}$ are defined inductively. Assume that they have been defined up to $m=n-1$. At this juncture, we assume that $k_{n-1}$ is of the form $k_{n-1}= r^{n-1}_{n-2}r_{n-1}^{n}$ with some integers $r_{n-2}<r_{n-1}$ satisfying $r_{n-1} \geq K_0\left(\{(kr_{n-2})^{n-1}\}_{k \in \N},\frac{k_{n-2}}{r_{n-2}},\frac{1}{1000},\left(\frac{r_{n-2}}{k_{n-2}}\right)^2\right)$, where we recall its definition from Remark \ref{rem:kForCLT}.

In the inductive step, we set
\begin{equation} \label{eq:l4}
l_n = \lfloor (r_{n-1})^{n\cdot \frac{1-u}{u}} \cdot \frac{6}{\epsilon_n} \rfloor.
\end{equation}
We apply Proposition \ref{prop:prob} with subexponential sequences $\{b_k\}_{k\in \N} = \{(kr_{n-1})^{n}\}_{k \in \N}$,
\[
\varepsilon=\frac{1}{1000}, \ \gamma = \left(\frac{r_{n-1}}{k_{n-1}}\right)^2, \ \text{ and alphabet } \Sigma=\left\{0,1,\dots, \frac{k_{n-1}}{r_{n-1}} -1\right\}.
\]
Let $L_0$ be the resulting threshold. Then we choose
\begin{equation}\label{eq:r4}
r_n \geq \max\left( L_0 ,\frac{4}{\epsilon_{n}}\right)
\end{equation}
as a multiple of $k_{n-1}$ and $s_{n}$. Hereby, we obtain $(r_nr_{n-1})^n$ many \emph{good words} of length $r_n$ from Proposition \ref{prop:prob} and we set
\begin{equation} \label{eq:k4}
k_n \coloneqq  r^n_{n-1} \cdot (r_n)^{n+1},
\end{equation}
which is the length of the concatenation of all these words. We also denote the corresponding values of $\tau_{{\mathbf b}}$ and $\beta(r_n)$ from Proposition \ref{prop:prob} by $\tau_n$ and $\beta_n$, respectively. Moreover, the collection of words $\Theta$ constructed in Proposition \ref{prop:prob} determines the assignment
\begin{equation}\label{eq:assignment2}
\psi_n:\{0,1,\dots ,k_n -1\} \to \left\{0,1,\dots ,\frac{k_{n-1}}{r_{n-1}} -1\right\}
\end{equation}
and, hence, the combinatorics of the conjugation map $h_n$ of the form
\begin{equation*}
h_n\left(\Gamma^{i,j}_{k_n,q_n}\right) = R^{i+\psi_n(j)r_{n-1} l_{n-1}q_{n-1}, \ \lfloor \frac{j}{s_n} \rfloor ,\ j \mod s_n}_{q_n,k_n,s_n}.
\end{equation*}
This finishes the construction in the inductive step.

Taking this combinatorics into account we make the following observation with regard to the generating sequence $\{\eta_m=H_m\left(\xi_{k_mq_m,s_m}\right)\}_{m\in \N}$.

\begin{lemma}\label{lem:lower4}
	Let $\varepsilon =\frac{1}{1000}$ and $m\in \N$. For $n>m$ we have
$$	S_{\eta_{m}}\left(T,N,\frac{\varepsilon}{2}\right)\geq\begin{cases}
	0.5 \frac{k_{n-1}}{r_{n-1}} & \text{for }q_{n}\leq N<5l_{n}q_{n},\\
	0.5 \lfloor (jr_{n-1})^{n} \rfloor & \text{for }jl_{n}q_{n}\leq N<(j+1)l_nq_n,\text{ where } 5\leq j <\beta_{n},\\
	0.5 \frac{k_n}{r_n}  & \text{for }\beta_{n}l_{n}q_{n}\leq N<k_{n}l_{n}q_{n}^{2}=q_{n+1}.
	\end{cases}
$$
\end{lemma}

\begin{proof}
	As in the proof of Lemma \ref{lem:lower3} we define the sets
	\begin{equation*}
	\overline{\Gamma}^{i,j}_{k_n,q_n} \coloneqq \Gamma^{i,j}_{k_n,q_n} \cap \Xi_{n+1}
	\end{equation*}
	and recall that for every $t\leq q_{n+1}$ and $P\in H_n\left(\overline{\Gamma}^{i,j}_{k_n,q_n}\right)$ the images $T^t\left(P\right)$ and $T^t_n\left(P\right)$ lie in the same element $R^{i',j'}_{k_nq_n,s_n}$ by Remark \ref{rem:safety}. By the same methods as in the proof of Lemma \ref{lem:lower3} we can show that Claim \ref{claim:1} and Claim \ref{claim:2} also hold true in our new setting. Then Proposition \ref{prop:prob} also determines the Hamming-distance of $\{h_n \circ R^t_{\alpha_{n+1}}\}_{t\leq q_{n+1}}$-trajectories of points $P_1 \in H_n\left(\overline{\Gamma}^{i_1,j_1r_n+v_1}_{k_n,q_n}\right)$ and $P_2 \in H_n\left(\overline{\Gamma}^{i_2,j_2r_n+v_2}_{k_n,q_n}\right)$  with $0\leq j_1,j_2 <\frac{k_n}{r_n}$, $j_1 \neq j_2$. An analysis like in the proof of Lemma \ref{lem:lower2} yields the statement.	
\end{proof}

\begin{proof}[Proof of Theorem \ref{theo:cyclicLower}]
	From Lemma \ref{lem:goodcyclic} and the accompanying Remark \ref{rem:cyclicCrit} we conclude that our transformation $T=\lim_{n\to \infty} T_n$ admits a good cyclic approximation.
	
	To get a lower bound on $\lent^{\mu}_{n^t}(T)$ we investigate
$\liminf_{N\to \infty} \frac{S_{\eta_m}\left(T,N,\varepsilon\right)}{N^t}$
	by examining the three different cases in the statement of Lemma \ref{lem:lower4}. In case of $q_n \leq N<5l_nq_n$ we see
	\begin{align*}
	\frac{S_{\eta_m}\left(T,N,\varepsilon\right)}{N^t} &\geq \frac{0.5\frac{k_{n-1}}{r_{n-1}}}{\left(5l_nq_n\right)^t} \\
	& \geq  \frac{0.5 (r_{n-2}r_{n-1})^{n-1}}{\left(\frac{30}{\epsilon_n} \cdot  (r_{n-1})^{n\cdot \frac{1-u}{u}}\cdot (r_{n-1})^{n} \cdot r^{n-1}_{n-2} \cdot l_{n-1}\cdot q^2_{n-1}\right)^t} \\
	& =  \frac{0.5\epsilon^t_n \cdot (r_{n-2}r_{n-1})^{n-1}}{\left(30 \cdot (r_{n-1})^{n\cdot \frac{1}{u}} \cdot r^{n-1}_{n-2} \cdot l_{n-1}\cdot q^2_{n-1}\right)^t},
	\end{align*}
	whose limit is positive for all $t<u$. Similarly, for $jl_{n}q_{n}\leq N<(j+1)l_nq_n$, where $5\leq j <\beta_{n}$ and $u\in[0,1]$, we have
	\[
	\frac{S_{\eta_m}\left(T,N,\varepsilon\right)}{N^t}\geq \frac{0.25 \cdot (jr_{n-1})^n}{\left((j+1)l_nq_n\right)^t} \geq \frac{0.25 \cdot j^{u} \cdot r^n_{n-1}} {\left((j+1)\cdot \frac{6}{\epsilon_n} \cdot (r_{n-1})^{n\cdot \frac{1}{u}}\cdot r^{n-1}_{n-2} \cdot l_{n-1} \cdot q^2_{n-1}\right)^t},
	\]
	whose limit is also positive for all $t<u$. In the remaining case $\beta_n l_nq_n \leq N <q_{n+1}$, we have
	\[	\frac{S_{\eta_m}\left(T,N,\varepsilon\right)}{N^t}\geq \frac{0.25 (r_nr_{n-1})^{n}}{\left(k_n l_nq^2_n\right)^t} \geq \frac{0.25(r_nr_{n-1})^{n}} {\left(r^n_{n-1} \cdot (r_n)^{n+1}\cdot l_n \cdot q^2_n\right)^t},
	\]
	which is positive for all $t< 1$. Altogether, this yields $\lent^{\mu}_{n^t}(T)\geq u$.
	
	On the other hand, Lemma \ref{lem:upper1} still holds true since its proof was independent of the combinatorics of $h_n$ as observed in Remark \ref{rem:IndepCombi}. This implies that
	\begin{align*}
	\liminf_{L\to \infty} \frac{S_{\eta_m}\left(T,L,\varepsilon\right)}{L^t} & \leq \lim_{n\to \infty} \frac{\frac{2}{\varepsilon^2}k_n q_n s_n}{\left(\frac{\varepsilon}{2} l_{n+1}q_{n+1}\right)^t}  \\
	& \leq \lim_{n\to \infty} \frac{2 \cdot \epsilon^t_{n+1} \cdot r^n_{n-1} \cdot (r_n)^{n+1} \cdot q_n \cdot s_n }{\varepsilon^{t+2} \cdot\left((r_n)^{(n+1)\cdot \frac{1-u}{u}} \cdot (r_n)^{n+1} \cdot r^n_{n-1} \cdot  l_{n} \cdot q^2_n\right)^t} \\
	& \leq \lim_{n\to \infty} \frac{2 \cdot \epsilon^t_{n+1}\cdot r^n_{n-1}  \cdot (r_n)^{n+1} \cdot q_n \cdot s_n }{\varepsilon^{t+2} \cdot\left( (r_n)^{(n+1)\cdot \frac{1}{u}} \cdot r^n_{n-1} \cdot l_{n} \cdot q^2_n\right)^t},
	\end{align*}
	which is zero for all $t> u$.
	
	Overall, we obtain $\lent^{\mu}_{n^t}(T)=u$.
\end{proof}

\section{Slow entropy of finite rank systems}\label{sec:finiterank}
In this section, we discuss the slow entropy of finite rank systems. We first provide an upper bound for the lower slow entropy of rank-two systems without any spacers. Then for any given subexponential scale we construct a rank-two system such that its lower slow entropy is infinity with respect to the given scale.
\subsection{Slow entropy of rank-two systems without spacers}

\begin{proposition}\label{prop:Nospacers}
	If $(X,T,\mathcal{B},\mu)$ is a rank-two system without spacers, then we have the following estimate of its lower slow entropy with respect to scale $a_n(t)=n^t$:
	$$\lent_{n^t}^{\mu}(T)\leq1.$$
\end{proposition}
\begin{proof}
	
	Let $h^{(1)}_n$, $h^{(2)}_n$, $B^{(1)}_n$, $B^{(2)}_n$ and $\mathcal{P}_n=\{T^jB^{(i)}_n:1\leq i\leq2, 0\leq j\leq h_n^{(i)}-1\}$ be defined as in Definition \ref{def:finiteRank} for a fixed partition $\mathcal{P}$ and $\epsilon>0$. We have the following two situations (by passing to subsequences if necessary):
	\begin{enumerate}[label=(\roman*)]
		\item\label{case1} $\lim_{n\to\infty}\frac{h^{(2)}_n}{h^{(1)}_n}=+\infty\text{ or }0$;
		\item\label{case2} $0<\lim_{n\to\infty}\frac{h^{(2)}_n}{h^{(1)}_n}<+\infty$.
	\end{enumerate}
	\paragraph{Case \ref{case1}:}
	We start with the situation that $\lim_{n\to\infty}\frac{h^{(2)}_n}{h^{(1)}_n}=+\infty$ and the other situation follows by switching the two towers. We assume that $\frac{h^{(1)}_n}{h^{(2)}_n}<\epsilon^3$ by enlarging $n$ if necessary. Denote $S^{(1)}_n=\cup_{i=0}^{h^{(1)}_n-1}T^iB^{(1)}_n$ and $S^{(2)}_n=\cup_{i=0}^{h^{(2)}_n-1}T^iB^{(2)}_n$. The following considerations will help us to estimate the number of Hamming balls with length $\epsilon^2h^{(2)}_n$ with respect to $\mathcal{P}_n$.
	
	Denote $A=S^{(1)}_n\cap T^{-1}(S^{(2)}_n)$, i.e. the points belonging to the shorter tower which will enter the tall tower at the next iteration. Then let $B=\cup_{i=1}^{2\epsilon^2h^{(2)}_n-1}T^{-i}A$. We obtain from the definition of $A$ that
$\mu(A)\leq\mu(B^{(2)}_n)\leq\frac{1}{h^{(2)}_n}$
and thus we know that $\mu(B)<2\epsilon^2$. Since $h^{(1)}_n<\epsilon^3h^{(2)}_n$, there exists $m\in\mathbb{Z}$ such that $mh^{(1)}_n\leq2\epsilon^2h^{(2)}_n<(m+1)h^{(1)}_n$ and $mh^{(1)}_n>\epsilon^2h^{(2)}_n$. Since there are no spacers, we get for $x,y\in S^{(1)}_n\setminus B$ that if they are in the same level set,  then
\begin{equation}\label{eq:keyObser}
\begin{aligned}
&\text{$T^ix$ and $T^iy$ will lie in the same atom with respect to partition $\mathcal{P}_n$}\\
&\text{for the next $\epsilon^2h^{(2)}_n$ iterates.}
\end{aligned}
\end{equation}
Let $C$ be the top level of $S^{(2)}_n$ and $D=\cup_{i=0}^{2\epsilon^2h^{(2)}_n-1}T^{-i}C$. For any $x,y\in S^{(2)}_n\setminus D$ we have, that if they are in the same level set, then they lie in the atom with respect to partition $\mathcal{P}_n$ for the next $\epsilon^2h^{(2)}_n$ iterates. Since the total measure of these points is larger than $\mu(S^{(1)}_n\setminus B)+\mu(S^{(2)}_n\setminus D)$, i.e. $1-4\epsilon^2$, we obtain that we need at most $h^{(1)}_n+h^{(2)}_n$ different $\epsilon^2-$Hamming balls with length $\epsilon^2h^{(2)}_n$ to cover an $1-4\epsilon^2$ portion of the space with respect to partition $\mathcal{P}_n$. Since $|\mathcal{P}_n-\mathcal{P}|<\epsilon$, we complete the proof in this case.
	\paragraph{Case \ref{case2}:}\label{para:case2}
	In this situation, by passing to a subsequence, we can assume that $h^{(1)}_n$ and $h^{(2)}_n$'s ratio is approximating to a constant. Then either by considering $\epsilon$-Hamming balls with length $\epsilon^2h^{(1)}_n$ or $\epsilon^2h^{(2)}_n$ with respect to $\mathcal{P}_n$, we need at most $h^{(1)}_n+h^{(2)}_n$ different such Hamming balls with length $\epsilon^2h^{(1)}_n$ or $\epsilon^2h^{(2)}_n$ to cover an $1-4\epsilon^2$ portion of the space. This is a similar argument as dealing with the taller tower in case \ref{case1}. As a result, we obtain linear bounds for the number of Hamming balls with lengths $\epsilon^2h^{(2)}_n$ or $\epsilon^2h^{(1)}_n$. Since $|\mathcal{P}_n-\mathcal{P}|<\epsilon$, we also complete the proof in this case.
\end{proof}
\begin{remark}\label{rem:multitower}
The proof in \eqref{para:case2} also works for any finite rank system provided that the height ratio of the towers is bounded away from $0$ and $+\infty$ (even with spacers), which indeed shows that any finite rank systems with bounded height ratio has linear complexity.
\end{remark}

\subsection{Rank-two system with arbitrarily large lower slow entropy}\label{sec:CrazyRankTwo}
The construction of our desired finite rank system is a rank-two system with spacers. More precisely, we proceed the construction by cutting-and-stacking method with two towers and several spacer levels, where the ratio of heights of the two towers is increasing to infinity and grows faster than a given subexponential scaling function $\{a_n(t)\}_{n\in\mathbb{N},t>0}$ (see \eqref{eq:scalingFunction} and \eqref{eq:initialEquation}). Then the inductive processes (especially step \ref{item:step6}) guarantee that the growth rate of the number of Hamming balls we need to cover a significant part of the higher tower is faster than the speed given by $\{a_n(t)\}_{n\in\mathbb{N},t>0}$ (see Proposition \ref{prop:towerSequence}).
\subsubsection{Construction process and some basic ergodic properties}
The finite rank system $T$ will act on $(X,\mathcal{B},\mu_X)$\footnote{For any $A\in\mathcal{B}$, $\mu_X(A)=\frac{\mu(A\cap X)}{\mu(X)}$.}, where $X$ is a non-empty interval of $[0,+\infty)$ of finite length, $\mathcal{B}$ is Borel $\sigma$-algebra and $\mu$ is Lebesgue measure. Let $\{a_n(t)\}_{n\in\mathbb{N},t>0}$ be a family of positive sequences increasing to infinity and monotone in $t$ that has subexponential growth, i.e. satisfying the following conditions:
\begin{equation}\label{eq:scalingFunction}
	\begin{aligned}
		&\lim_{n\to+\infty}a_n(t)=+\infty, \forall t>0;\lim_{n\to+\infty}\frac{\log a_n(t)}{n}=0, \forall t>0;\\
		&a_n(t)\leq a_{n+1}(t),\forall t>0;\ \ a_n(t_1)\leq a_n(t_2), \forall t_2> t_1>0.\\
	\end{aligned}
\end{equation}
We also fix a positive and increasing sequence $\{t_n\}_{n \in \N}$ such that $\lim_{n\to+\infty}t_n=+\infty$, which will be used in our construction of a rank-two system with infinite lower slow entropy. In the following context, we fix $\epsilon=\frac{1}{1000}$ and introduce $Q_{n,k}$ for any $n,k\geq1$ as
\begin{equation}\label{eq:Nk}
	Q_{n,k}=\exp(k(24\epsilon\log(n+1)-(1-2\epsilon)\log(1-4\epsilon)-6\epsilon\log(2\epsilon)))
\end{equation}
whose definition is motivated by Lemma \ref{lem:HammingCloseEstimate} for an alphabet of cardinality $|\Sigma|=(n+1)^6$ and strings of length $k$. Thus, $Q_{n,k}$ gives the number of words in $\Sigma^k$ that are $2\epsilon$-Hamming close to a given $k$-string.

We start our construction from two towers $S_1^{(1)}$, $S_1^{(2)}$, where
\begin{equation}
	\begin{aligned}
		&S_1^{(1)}=\{[\frac{i}{2^{21}},\frac{i+1}{2^{21}}):i=0,1,\ldots,2^{21}-1\},\\ &S_1^{(2)}=\{[1+\frac{i}{2^{70}},1+\frac{i+1}{2^{70}}):i=0,1,\ldots,2^{70}-1\},
	\end{aligned}
\end{equation}
where our map $T$ on $S^{(1)}_1$ and $S^{(2)}_1$ is a translation between intervals in each tower respectively. Moreover, we denote $H_0=1$.

In the inductive process, we denote our tower at step $n$ as $S_n^{(1)}$ and $S_n^{(2)}$ with heights $h_n^{(1)}$, $h_n^{(2)}$ and bases $B_n^{(1)}$, $B_n^{(2)}$, respectively, as in the following picture:
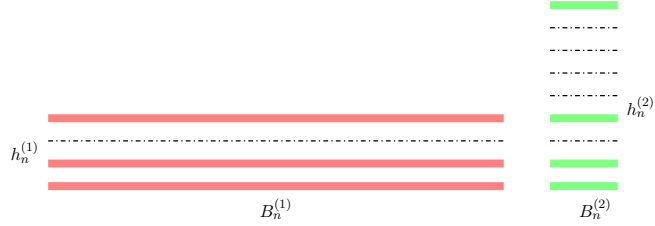
\begin{figure}[H]
	\centering
	\scalebox{0.6}
	{
		\begin{tikzpicture}[scale=5]
			\tikzstyle{vertex}=[circle,minimum size=0pt,inner sep=0pt]
			\tikzstyle{selected vertex} = [vertex, fill=red!24]
			\tikzstyle{edge1} = [draw,line width=5pt,-,red!50]
			\tikzstyle{edge2} = [draw,line width=5pt,-,green!50]
			\tikzstyle{edge3} = [draw,line width=5pt,-,blue!50]
			\tikzstyle{edge4} = [draw,line width=5pt,-,brown!50]
			
			\tikzstyle{edge} = [draw,thick,-,black]
			\node[vertex] (t000) at (0.0,0) {};
			\node[vertex] (t001) at (2,0) {};
			\node[vertex] (t010) at (0.0,0.1) {};
			\node[vertex] (t011) at (2,0.1) {};
			\node[vertex] (t020) at (0.0,0.2) {};
			\node[vertex] (t021) at (2,0.2) {};
			\node[vertex] (t030) at (0.0,0.3) {};
			\node[vertex] (t031) at (2,0.3) {};
			\node[vertex] (w00) at (1,-0.1) {$B_n^{(1)}$};
			\node[vertex] (w01) at (-0.1,0.15) {$h_n^{(1)}$};
			
			\node[vertex] (t100) at (2.2,0) {};
			\node[vertex] (t101) at (2.5,0) {};
			\node[vertex] (t110) at (2.2,0.1) {};
			\node[vertex] (t111) at (2.5,0.1) {};
			\node[vertex] (t120) at (2.2,0.2) {};
			\node[vertex] (t121) at (2.5,0.2) {};
			\node[vertex] (t130) at (2.2,0.3) {};
			\node[vertex] (t131) at (2.5,0.3) {};
			\node[vertex] (t140) at (2.2,0.4) {};
			\node[vertex] (t141) at (2.5,0.4) {};
			\node[vertex] (t150) at (2.2,0.5) {};
			\node[vertex] (t151) at (2.5,0.5) {};
			\node[vertex] (t160) at (2.2,0.6) {};
			\node[vertex] (t161) at (2.5,0.6) {};
			\node[vertex] (t170) at (2.2,0.7) {};
			\node[vertex] (t171) at (2.5,0.7) {};
			\node[vertex] (t180) at (2.2,0.8) {};
			\node[vertex] (t181) at (2.5,0.8) {};
			\node[vertex] (b10) at (2.4,-0.1) {$B_n^{(2)}$};
			\node[vertex] (b11) at (2.6,0.35) {$h_n^{(2)}$};
			
			\draw[edge1] (t000)--(t001);
			\draw[edge1] (t010)--(t011);
			\draw[thick,dash dot] (t020)--(t021);
			\draw[edge1] (t030)--(t031);

			\draw[edge2] (t100)--(t101);
			\draw[edge2] (t110)--(t111);
			\draw[thick,dash dot] (t120)--(t121);
			\draw[edge2] (t130)--(t131);
			\draw[thick,dash dot] (t140)--(t141);
			\draw[thick,dash dot] (t150)--(t151);
			\draw[thick,dash dot] (t160)--(t161);
			\draw[thick,dash dot] (t170)--(t171);
			\draw[edge2] (t180)--(t181);

		\end{tikzpicture}.
	}
\end{figure}

Under this setting, the inductive construction proceeds as follows:
\begin{enumerate}[label=(\roman*)]
	\item\label{item:step1} Let $K_{n,0}=K_0(\{k\}_{k\in\mathbb{N}},(n+1)^6,\frac{1}{1000},\frac{1}{(n+1)^{12}})$ be defined as in Proposition \ref{prop:prob} and Remark \ref{rem:kForCLT}. Suppose that we have already constructed towers $S_n^{(1)}$, $S_n^{(2)}$ and that positive integers $H_n$, $H_{n-1}$, and $\tau_{n}$ are given such that the following conditions are satisfied:
	\begin{equation}\label{eq:initialEquation}
		\begin{aligned}
			&H_n=\prod_{i=0}^{2(n+1)^6-1}(h_n^{(1)}+iH_{n-1}),
			\ \ h_n^{(1)}\geq(n+1)^{21}H_{n-1},\\
			&h_n^{(2)}\geq2^{27n^2} (K_{n,0}h_n^{(1)})^2 H_n,\ \
			\frac{1}{3}\mu(S_n^{(2)})<\mu(S_n^{(1)})<\frac{31}{10}\mu(S_n^{(2)}),\\
&H_{n-1}|h_n^{(1)}, \ \ H_{n}|h_n^{(2)},
		\end{aligned}
	\end{equation}
	and $\tau_{n}\geq n$ is large enough such that for every $k\geq\tau_{n}$, we have $\frac{Q_{n,k}b_{k,n}}{(n+1)^{6k}}< 1$, where $b_{k,n} = a_{(k+1)H_{n}}(t_{n})$ and $Q_{n,k}$ is defined in \eqref{eq:Nk}. We observe that $H_{n-1}|h_n^{(1)}$ and the definition of $H_n$ imply $H_{n-1}|H_n$.

\item\label{item:step2} Let
\begin{align*}
	\alpha_k &=\frac{\frac{1}{h_n^{(1)}+kH_{n-1}}}{\sum_{i=0}^{(n+1)^6-1}\frac{1}{h_n^{(1)}+iH_{n-1}}} & \text{ for } 0\leq k\leq(n+1)^6-1, \\
	\beta_j &=\frac{\frac{1}{h_n^{(1)}+jH_{n-1}}}{\sum_{i=(n+1)^6}^{2(n+1)^6-1}\frac{1}{h_n^{(1)}+iH_{n-1}}} & \text{ for } (n+1)^6\leq j\leq2(n+1)^6-1.
\end{align*}
Then we cut $B_n^{(1)}$ into two disjoint subintervals $B_n^{(1,1)}$ and $B_n^{(1,2)}$ such that for some $g_n\geq(n+2)^{10}K_{n,0}$ we have
    \begin{equation}\label{eq:choiceOfB23}
    \begin{aligned}
\mu(B_n^{(1,2)})(h_n^{(1)}+\sum_{j=(n+1)^6}^{2(n+1)^6-1}jH_{n-1}\beta_j)&=g_n^2H_n(n+1)^{12}\mu(B_n^{(2)}),\\
\mu(B_n^{(1,2)})&\in(\frac{1}{100}\mu(B_n^{(1)}),\frac{1}{75}\mu(B_n^{(1)})).
    \end{aligned}
    \end{equation}
    The existence of such $g_n$ follows from \eqref{eq:initialEquation} since
    \begin{equation*}
    \begin{aligned}
    \frac{K_{n,0}^2(n+2)^{20}(n+1)^{12}H_n\mu(B_n^{(2)})}{(h_n^{(1)}+\sum_{j=(n+1)^6}^{2(n+1)^6-1}jH_{n-1}\beta_j)\mu(B_n^{(1)})}&\leq\frac{K_{n,0}^2(n+2)^{32}H_n\mu(B_n^{(2)})}{\mu(S_n^{(1)})}\\
    &\leq\frac{3(n+2)^{32}K_{n,0}^2H_n}{h_n^{(2)}}\leq\frac{3}{2^{27n^2}}.
    \end{aligned}
    \end{equation*}

\item\label{item:step3} Cut $B_n^{(1,1)}$ into $(n+1)^6$ disjoint intervals $W_n^{(k)}$ such that for any $0\leq k\leq (n+1)^6-1$, we have
	\begin{equation} \label{eq:WidthsBases}
	\begin{aligned}
	&\mu(W_n^{(k)})=\alpha_k\mu(B_n^{(1,1)}).\\
	\end{aligned}
	\end{equation}
	Then we add $kH_{n-1}$ spacers on top of $T^{h_n^{(1)}-1}W_n^{(k)}$ as first part of newly added mass.
Next, we cut $B_n^{(1,2)}$ into $(n+1)^6$ disjoint intervals $W_n^{(j)}$ such that for any $(n+1)^6\leq j\leq 2(n+1)^6-1$, we have
	\begin{equation} \label{eq:WidthsBases1}
    \begin{aligned}
&\mu(W_n^{(j)})=\beta_j\mu(B_n^{(1,2)}).\\
	\end{aligned}
	\end{equation}
Then we add $jH_{n-1}$ spacers on top of $T^{h_n^{(1)}-1}W_n^{(j)}$ as the second part of newly added mass.

Indeed this construction can be intuitively expressed by the following picture, where the red levels belong to $S_n^{(1)}$, green levels belong to $S_n^{(2)}$ and blue levels are newly added spacers:
	
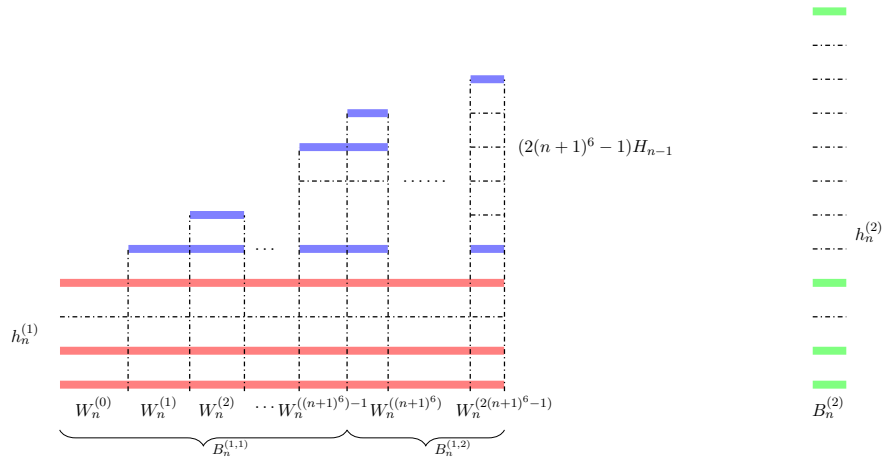
\begin{figure}[H]
		\centering
		\scalebox{0.6}
		{
			\begin{tikzpicture}[scale=5]
				\tikzstyle{vertex}=[circle,minimum size=0pt,inner sep=0pt]
				\tikzstyle{selected vertex} = [vertex, fill=red!24]
				\tikzstyle{edge1} = [draw,line width=5pt,-,red!50]
				\tikzstyle{edge2} = [draw,line width=5pt,-,green!50]
				\tikzstyle{edge3} = [draw,line width=5pt,-,blue!50]
				\tikzstyle{edge4} = [draw,line width=5pt,-,brown!50]
				
				\tikzstyle{edge} = [draw,thick,-,black]
				\node[vertex] (t000) at (0.0,0) {};
				\node[vertex] (t001) at (1.95,0) {};
				\node[vertex] (t010) at (0.0,0.15) {};
				\node[vertex] (t011) at (1.95,0.15) {};
				\node[vertex] (t020) at (0.0,0.3) {};
				\node[vertex] (t021) at (1.95,0.3) {};
				\node[vertex] (t030) at (0.0,0.45) {};
				\node[vertex] (t031) at (1.95,0.45) {};
				\node[vertex] (b01) at (-0.15,0.225) {$h_n^{(1)}$};
				
				\node[vertex] (e00) at (1.95,-0.03) {};
				\node[vertex] (e01) at (1.95,1.35) {};
				
				\node[vertex] (w00) at (0.3,-0.03) {};
				\node[vertex] (w01) at (0.3,0.6) {};
				\node[vertex] (w02) at (0.15,-0.1) {$W_n^{(0)}$};
				\node[vertex] (w10) at (0.57,-0.03) {};
				\node[vertex] (w11) at (0.57,0.75) {};
				\node[vertex] (w12) at (0.435,-0.1) {$W_n^{(1)}$};
				\node[vertex] (w20) at (0.81,-0.03) {};
				\node[vertex] (w21) at (1.05,0.9) {};
				\node[vertex] (w22) at (0.69,-0.1) {$W_n^{(2)}$};
				\node[vertex] (w30) at (1.05,-0.03) {};
				\node[vertex] (w31) at (1.05,1.05) {};
				\node[vertex] (w32) at (0.90,-0.1) {$\ldots$};
                \node[vertex] (w33) at (0.90,0.6) {$\ldots$};
				\node[vertex] (w40) at (1.26,-0.03) {};
				\node[vertex] (w41) at (1.26,1.2) {};
				\node[vertex] (w42) at (1.155,-0.1) {$W_n^{((n+1)^6)-1}$};
				\node[vertex] (w50) at (1.8,-0.03) {};
				\node[vertex] (w51) at (1.8,1.35) {};
				\node[vertex] (w52) at (1.95,-0.1) {$W_n^{(2(n+1)^6-1)}$};
                \node[vertex] (w62) at (1.52,-0.1) {$W_n^{((n+1)^6)}$};

				\node[vertex] (i00) at (1.8,0.6) {};
				\node[vertex] (i01) at (1.95,0.6) {};
				\node[vertex] (i10) at (1.8,0.75) {};
				\node[vertex] (i11) at (1.95,0.75) {};
				\node[vertex] (i20) at (1.8,0.9) {};
				\node[vertex] (i21) at (1.95,0.9) {};
				\node[vertex] (i30) at (1.8,1.05) {};
				\node[vertex] (i31) at (1.95,1.05) {};
				\node[vertex] (i40) at (1.8,1.2) {};
				\node[vertex] (i41) at (1.95,1.2) {};

				\node[vertex] (sp10) at (0.81,0.6) {};
				\node[vertex] (sp20) at (0.81,0.75) {};
				\node[vertex] (sp30) at (1.44,0.9) {};
				\node[vertex] (ad50) at (1.05,0.6) {};
                \node[vertex] (ad51) at (1.44,0.6)
{};
                \node[vertex] (ad52) at (1.44,1.2)
{};
                \node[vertex] (ad53) at (1.44,1.05)
{};
                \node[vertex] (ad54) at (1.44,-0.03) {};

				\node[vertex] (h00) at (2.35,1.05) {$(2(n+1)^6-1)H_{n-1}$};
				
				\node[vertex] (h01) at (1.6,0.9) {$\ldots\ldots$};
				
				\node[vertex] (t100) at (3.3,0) {};
				\node[vertex] (t101) at (3.45,0) {};
				\node[vertex] (t110) at (3.3,0.15) {};
				\node[vertex] (t111) at (3.45,0.15) {};
				\node[vertex] (t120) at (3.3,0.3) {};
				\node[vertex] (t121) at (3.45,0.3) {};
				\node[vertex] (t130) at (3.3,0.45) {};
				\node[vertex] (t131) at (3.45,0.45) {};
				\node[vertex] (t140) at (3.3,0.6) {};
				\node[vertex] (t141) at (3.45,0.6) {};
				\node[vertex] (t150) at (3.3,0.75) {};
				\node[vertex] (t151) at (3.45,0.75) {};
				\node[vertex] (t160) at (3.3,0.9) {};
				\node[vertex] (t161) at (3.45,0.9) {};
				\node[vertex] (t170) at (3.3,1.05) {};
				\node[vertex] (t171) at (3.45,1.05) {};
				\node[vertex] (t180) at (3.3,1.2) {};
				\node[vertex] (t181) at (3.45,1.2) {};
				\node[vertex] (t190) at (3.3,1.35) {};
				\node[vertex] (t191) at (3.45,1.35) {};
				\node[vertex] (t1100) at (3.3,1.5) {};
				\node[vertex] (t1101) at (3.45,1.5) {};
				\node[vertex] (t1110) at (3.3,1.65) {};
				\node[vertex] (t1111) at (3.45,1.65) {};
				\node[vertex] (b10) at (3.37,-0.1) {$B_n^{(2)}$};
				\node[vertex] (b11) at (3.55,0.675) {$h_n^{(2)}$};
				

				\draw[edge1] (t000)--(t001);
				\draw[edge1] (t010)--(t011);
				\draw[thick,dash dot] (t020)--(t021);
				\draw[edge1] (t030)--(t031);

				\draw[edge2] (t100)--(t101);
				\draw[edge2] (t110)--(t111);
				\draw[thick,dash dot] (t120)--(t121);
				\draw[edge2] (t130)--(t131);
				\draw[thick,dash dot] (t140)--(t141);
				\draw[thick,dash dot] (t150)--(t151);
				\draw[thick,dash dot] (t160)--(t161);
				\draw[thick,dash dot] (t170)--(t171);
				\draw[thick,dash dot] (t180)--(t181);
				\draw[thick,dash dot] (t190)--(t191);
				\draw[thick,dash dot] (t1100)--(t1101);
				\draw[edge2] (t1110)--(t1111);

				\draw[edge3] (w01)--(sp10);
				\draw[edge3] (w11)--(sp20);
				\draw[thick,dash dot] (w21)--(sp30);
                \draw[thick,dash dot] (ad52)--(ad54);
				\draw[edge3] (w31)--(ad53);
                \draw[edge3] (w41)--(ad52);
                \draw[edge3] (ad50)--(ad51);
				\draw[edge3] (e01)--(w51);
				\draw[edge3] (i00)--(i01);
				\draw[thick,dash dot] (i10)--(i11);
				\draw[thick,dash dot] (i20)--(i21);
				\draw[thick,dash dot] (i30)--(i31);
				\draw[thick,dash dot] (i40)--(i41);
				
				\draw[thick,dash dot] (e00)--(e01);
				\draw[thick,dash dot] (w00)--(w01);
				\draw[thick,dash dot] (w10)--(w11);
				\draw[thick,dash dot] (w20)--(sp20);
				\draw[thick,dash dot] (w30)--(w31);
				\draw[thick,dash dot] (w40)--(w41);
				\draw[thick,dash dot] (w50)--(w51);
				
				\draw [decorate,decoration={brace,amplitude=10pt},xshift=-0pt,yshift=0pt]
(1.26,-0.2) -- (0.0,-0.2) node [black,midway,xshift=0.6cm,yshift=-0.4cm]
{\footnotesize $B_n^{(1,1)}$};
				\draw [decorate,decoration={brace,amplitude=10pt},xshift=-0pt,yshift=0pt]
(1.95,-0.2) -- (1.26,-0.2) node [black,midway,xshift=0.6cm,yshift=-0.4cm]
{\footnotesize $B_n^{(1,2)}$};
			\end{tikzpicture}
		}.
	\end{figure}
Now we estimate the total measure of spacers we added in step \ref{item:step3} of stage $n+1$, i.e. the blue levels in the above pictures: 
	\begin{equation}\label{eq:addMass1}
	\begin{aligned} \sum_{k=1}^{2(n+1)^6-1}kH_{n-1}\mu(W_n^{(k)})&=\sum_{k=1}^{(n+1)^6-1}kH_{n-1}\alpha_k\mu(B_n^{(1,1)})+\sum_{j=(n+1)^6}^{2(n+1)^6-1}jH_{n-1}\beta_j\mu(B_n^{(1,2)})\\
&< 2(n+1)^6H_{n-1}\mu(B_n^{(1)})=\frac{2(n+1)^6H_{n-1}}{h_n^{(1)}}\mu(S_n^{(1)}),\\
&\leq\frac{1}{(n+1)^{14}}\mu(S_n^{(1)}).
	\end{aligned}
	\end{equation}

	\item\label{item:step4} For $0\leq k\leq 2(n+1)^6-1$, we define: \begin{equation}
		\overline{D}_n^{(k)}=\bigcup_{i=0}^{h_n^{(1)}+kH_{n-1}-1}T^iW_n^{(k)}.
	\end{equation}
	Then we divide $\overline{D}_n^{(k)}$ into $\frac{H_n}{h_n^{(1)}+kH_{n-1}}$ many sub-towers of the same width and then stack them together. We denote this new tower by $D_n^{(k)}$ and observe that the height of $D_n^{(k)}$ is $H_n$. By our choices of $\mu(W^{(k)}_n)$ in (\ref{eq:WidthsBases}), we also have
\begin{enumerate}
\item For $0\leq k \leq (n+1)^6-1$, towers $D_n^{(k)}$ have the same width;
\item For $(n+1)^6\leq j \leq 2(n+1)^6-1$, towers $D_n^{(j)}$ have the same width. Moreover, \eqref{eq:choiceOfB23} gives that  $$\mu(D_n^{(j)})=g_n^2(n+1)^6H_n\mu(B_n^{(2)}).$$
\end{enumerate}
The following picture provides an intuitive explanation of this step:
	
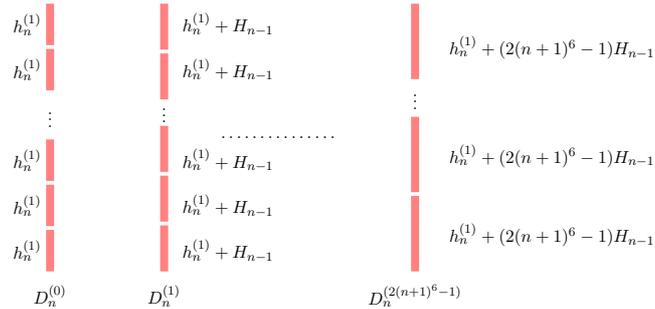
\begin{figure}[H]
		\centering
		\scalebox{0.6}
		{
			\begin{tikzpicture}[scale=5]
				\tikzstyle{vertex}=[circle,minimum size=2pt,inner sep=0pt]
				\tikzstyle{selected vertex} = [vertex, fill=red!24]
				\tikzstyle{edge1} = [draw,line width=5pt,-,red!50]
				\tikzstyle{edge2} = [draw,line width=5pt,-,green!50]
				\tikzstyle{edge3} = [draw,line width=5pt,-,blue!50]
				\tikzstyle{edge4} = [draw,line width=5pt,-,brown!50]
				
				\tikzstyle{edge} = [draw,thick,-,black]
				\node[vertex] (d00) at (0,-0.1) {$D_n^{(0)}$};
				\node[vertex] (t00) at (0,0) {};
				\node[vertex] (t01) at (0,0.2) {};
				\node[vertex] (t02) at (0,0.4) {};
				\node[vertex] (t03) at (0,0.6) {};
				\node[vertex] (t04) at (0,0.7) {$\vdots$};
				\node[vertex] (t05) at (0,0.8) {};
				\node[vertex] (t06) at (0,1) {};
				\node[vertex] (t07) at (0,1.2) {};
				\node[vertex] (e01) at (-0.1,0.1) {$h_n^{(1)}$};
				\node[vertex] (e02) at (-0.1,0.3) {$h_n^{(1)}$};
				\node[vertex] (e03) at (-0.1,0.5) {$h_n^{(1)}$};
				\node[vertex] (e04) at (-0.1,0.9) {$h_n^{(1)}$};
				\node[vertex] (e05) at (-0.1,1.1) {$h_n^{(1)}$};
				
				\draw[edge1] (t00)--(t01);
				\draw[edge1] (t01)--(t02);
				\draw[edge1] (t02)--(t03);
				\draw[edge1] (t05)--(t06);
				\draw[edge1] (t06)--(t07);
				
				\node[vertex] (d10) at (0.5,-0.1) {$D_n^{(1)}$};
				\node[vertex] (t10) at (0.5,0) {};
				\node[vertex] (t11) at (0.5,0.22) {};
				\node[vertex] (t12) at (0.5,0.44) {};
				\node[vertex] (t13) at (0.5,0.66) {};
				\node[vertex] (t14) at (0.5,0.72) {$\vdots$};
				\node[vertex] (t15) at (0.5,0.76) {};
				\node[vertex] (t16) at (0.5,0.98) {};
				\node[vertex] (t17) at (0.5,1.2) {};
				\node[vertex] (e11) at (0.78,0.1) {$h_n^{(1)}+H_{n-1}$};
				\node[vertex] (e12) at (0.78,0.3) {$h_n^{(1)}+H_{n-1}$};
				\node[vertex] (e13) at (0.78,0.5) {$h_n^{(1)}+H_{n-1}$};
				\node[vertex] (e14) at (0.78,0.9) {$h_n^{(1)}+H_{n-1}$};
				\node[vertex] (e15) at (0.78,1.1) {$h_n^{(1)}+H_{n-1}$};
				
				\draw[edge1] (t10)--(t11);
				\draw[edge1] (t11)--(t12);
				\draw[edge1] (t12)--(t13);
				\draw[edge1] (t15)--(t16);
				\draw[edge1] (t16)--(t17);

				\node[vertex] (d20) at (1.6,-0.1) {$D_n^{(2(n+1)^6-1)}$};
				\node[vertex] (t20) at (1.6,0) {};
				\node[vertex] (t21) at (1.6,0.35) {};
				\node[vertex] (t22) at (1.6,0.7) {};
				\node[vertex] (t23) at (1.6,0.78) {$\vdots$};
				\node[vertex] (t24) at (1.6,0.85) {};
				\node[vertex] (t25) at (1.6,1.2) {};
				
				\node[vertex] (e21) at (2.2,0.17) {$h_n^{(1)}+(2(n+1)^6-1)H_{n-1}$};
				\node[vertex] (e22) at (2.2,0.52) {$h_n^{(1)}+(2(n+1)^6-1)H_{n-1}$};
				\node[vertex] (e23) at (2.2,1) {$h_n^{(1)}+(2(n+1)^6-1)H_{n-1}$};

				\draw[edge1] (t20)--(t21);
				\draw[edge1] (t21)--(t22);
				\draw[edge1] (t24)--(t25);
				
				\node[vertex] (t25) at (1,0.6) {$\ldots\ldots\ldots\ldots\ldots$};

			\end{tikzpicture}
		}.
	\end{figure}

	\item\label{item:step5} Since $g_n\geq K_{n,0}=K_0(\{k\}_{k\in\mathbb{N}},(n+1)^6,\frac{1}{1000},\frac{1}{(n+1)^{12}})$, we can apply Proposition \ref{prop:prob} with $\Sigma=\{(n+1)^6,\ldots,2(n+1)^6-1\}$, $\gamma=\frac{1}{(n+1)^{12}}$, $\epsilon=\frac{1}{1000}$ with subexponential sequence $b_k=k$ to find $g_n(n+1)^6$ different words $a_n^{(0)},\ldots,a_{n}^{(g_n(n+1)^6-1)}\in\Sigma^{g_n(n+1)^6}$ which are uniform in the symbols.
Then for $(n+1)^6\leq j\leq 2(n+1)^6-1$, the exact uniformity of Proposition \ref{prop:prob} and the choice of $g_n$ indicate that by cutting each $D_n^{(j)}$ into $g_n^2(n+1)^{6}$ subtowers of equal width we can obtain $g_n(n+1)^6$ towers $A_n^{(0)},\ldots,A_{n}^{(g_n(n+1)^6-1)}$ that are built as concatenations of $g_n(n+1)^6$ of these subtowers such that $A_n^{(k)}$ corresponding to $a_n^{(k)}$ by viewing $D_n^{(j)}$ as symbol $j$. Since $\mu(D_n^{(j)})=g_n^2H_n(n+1)^6\mu(B_n^{(2)})$ and the height of $D_n^{(j)}$ is $H_n$, we have that the width of $A_n^{(k)}$ is $\mu(B_n^{(2)})$ for $0\leq k\leq g_n(n+1)^6-1$.

We stack $A_n^{(0)},\ldots,A_{n}^{(g_n(n+1)^6-1)}$ over $S_n^{(2)}$ and denote this new tower as $S_{n+1}^{(1)}$, whose height is denoted by $h_{n+1}^{(1)}$. By construction of $S_{n+1}^{(1)}$, we have that $$h_{n+1}^{(1)}=(g_n(n+1)^6)^2H_n+h_n^{(2)},$$ which together with $H_n |h_n^{(2)}$  guarantees that $H_n |h_{n+1}^{(1)}$. Since $g_n\geq K_{n,0}(n+2)^{10}$, we also obtain that
$$h_{n+1}^{(1)}\geq(n+2)^{21}H_n.$$

	\item\label{item:step6}
	Let $H_{n+1}=\prod_{i=0}^{2(n+2)^6-1}(h_{n+1}^{(1)}+iH_{n})$, $\{b_{k,n+1}\}_{k\in \N} = \{a_{(k+1)H_{n+1}}(t_{n+1})\}_{k\in \N}$ and $\tau_{n+1}\geq n+1$ large enough such that for every $k\geq\tau_{n+1}$, we have $\frac{Q_{n+1,k}b_{k,n+1}}{(n+2)^{6k}}< 1$, where $Q_{n+1,k}$ is defined according to \eqref{eq:Nk}. Then we pick $R_{n} \in \Z^{+}$ as a multiple of $(n+1)^6H_{n+1}$ large enough such that
	\begin{equation}\label{eq:heightHn23}
		\begin{aligned}
			&b_{R_n,n}>(n+1)^8,\\
			&R_{n}H_n\geq\tau_{n+1}H_{n+1},\\
			&R_{n}\geq\max\{\tau_{n}+1,K_{n,1},2^{27(n+1)^2}(K_{n+1,0}h_{n+1}^{(1)})^2H_{n+1}\},\\
		\end{aligned}
	\end{equation}
	where $\tau_{n}$ is defined in step \ref{item:step1}, $K_{n,1}\coloneqq K_{0}\left(\{b_{k,n}\}_{k\in \N},(n+1)^6,\frac{1}{1000},\frac{1}{(n+1)^{12}}\right)$ and $K_{n+1,0}=K_0(\{k\}_{k\in\mathbb{N}},(n+2)^6,\frac{1}{1000},\frac{1}{(n+2)^{12}})$ from Proposition \ref{prop:prob} and Remark \ref{rem:kForCLT}.
	
	We apply Proposition \ref{prop:prob} with alphabet $\Sigma=\{0,1,\dots , (n+1)^6-1\}$, $\gamma=\frac{1}{(n+1)^{12}}$, $\epsilon=\frac{1}{1000}$, $k=R_n$, and subexponential sequence $\{b_{m,n}\}_{m\in \N}$ to find $N_n \coloneqq b_{R_n,n}$ many different words $e_n^{(0)},\ldots,e_{n}^{(N_n-1)}\in\Sigma^{R_n}$ which are uniform in the symbols. Then for $0\leq i\leq (n+1)^6-1$, the exact uniformity of Proposition \ref{prop:prob} indicates that by cutting $D_n^{(i)}$ into $\frac{N_nR_n}{(n+1)^6}$ subtowers of equal width, we can obtain $N_n$ subtowers $E_n^{(0)},\ldots,E_{n}^{(N_n-1)}$ that are built by concatenating $R_n$ of these substowers in such a way that $E_n^{(k)}$ corresponds to $e_n^{(k)}$ by viewing $D_n^{(i)}$ as symbol $i$.

\item Stacking $E_n^{(0)},\ldots,E_{n}^{(N_n-1)}$ above each other, we get a tower of height $N_nR_nH_n$ and we denote this tower as $S_{n+1}^{(2)}$. The choice of $R_n$ implies that its height $h_{n+1}^{(2)}$ satisfies $H_{n+1}|h_{n+1}^{(2)}$. Now we check that our construction satisfies our inductive assumptions in \eqref{eq:initialEquation} for $h_{n+1}^{(2)}$, $\mu(S_{n+1}^{(1)})$ and $\mu(S_{n+1}^{(2)})$:

Recall that $h_{n+1}^{(2)}=N_nR_nH_n$. Then \eqref{eq:heightHn23} guarantees that
	$$h_{n+1}^{(2)}\geq R_n\geq2^{27(n+1)^2}(K_{n+1,0}h_{n+1}^{(1)})^2H_{n+1}.$$
	
	Finally, step \ref{item:step2}, estimate \eqref{eq:addMass1}, and step \ref{item:step5} imply
	\begin{equation}\label{eq:meaureEstimate1}
		\begin{aligned}
			\frac{98}{100}\mu(S_n^{(1)})\leq&\mu(S_{n+1}^{(2)})\leq\frac{99}{100}\mu(S_n^{(1)})+\frac{1}{10000}\mu(S_n^{(1)}),\\
			\mu(S_n^{(2)})+\frac{1}{100}\mu(S_n^{(1)})\leq&\mu(S_{n+1}^{(1)})\leq\mu(S_n^{(2)})+\frac{2}{100}\mu(S_n^{(1)}).
		\end{aligned}
	\end{equation}
	which together with \eqref{eq:initialEquation} and \eqref{eq:choiceOfB23} guarantee the following estimates:
	\begin{equation}
		\begin{aligned}
			&\frac{\mu(S_{n+1}^{(1)})}{\mu(S_{n+1}^{(2)})}\leq\frac{\mu(S_n^{(2)})+\frac{2}{100}\mu(S_n^{(1)})}{\frac{98}{100}\mu(S_n^{(1)})}<\frac{3\mu(S_n^{(1)})+\frac{2}{100}\mu(S_n^{(1)})}{\frac{98}{100}\mu(S_n^{(1)})}=\frac{302}{98}<\frac{31}{10},\\
			&\frac{\mu(S_{n+1}^{(1)})}{\mu(S_{n+1}^{(2)})}\geq\frac{\mu(S_n^{(2)})+\frac{1}{100}\mu(S_n^{(1)})}{\frac{99}{100}\mu(S_n^{(1)})+\frac{1}{10000}\mu(S_n^{(1)})}>\frac{\frac{10}{31}\mu(S_n^{(1)})+\frac{1}{100}\mu(S_n^{(1)})}{\frac{99}{100}\mu(S_n^{(1)})+\frac{1}{10000}\mu(S_n^{(1)})}>\frac{1}{3}.
		\end{aligned}
	\end{equation}
	Combining these two estimates, we obtain that
	\begin{equation}\label{eq:massEstimate}
		\begin{aligned}
			\frac{1}{3}\mu(S_{n+1}^{(2)})<\mu(S_{n+1}^{(1)})<\frac{31}{10}\mu(S_{n+1}^{(2)}),
		\end{aligned}
	\end{equation}
	and, thus, we finish our construction at step $n+1$.
\end{enumerate}

After completing the construction, it is worth to point out that the right-hand side of \eqref{eq:addMass1} gives a summable sequence. Thus,
\begin{equation}\label{eq:addMassEst}
	\mu(X)\leq(\mu(S_1^{(1)})+\mu(S_1^{(2)}))\prod_{n=1}^{\infty}(1+\frac{1}{(n+1)^{14}})\leq 2\exp(\sum_{n=1}^{\infty}\frac{1}{(n+1)^{14}})<+\infty,
\end{equation}
which implies that the measurable system $(T,X,\mathcal{B},\mu_X)$ that we construct above is a measure-preserving transformation on a finite interval of $[0,+\infty)$. Moreover, the construction implies that this system's rank is at most $2$.

We obtain the ergodicity of the system by exchanging mass between the two towers on each stage of the construction (see step \ref{item:step5}). The detailed proof is postponed to Appendix \ref{sec:ergodic}.
\begin{proposition}\label{prop:ergodic}
	$(T,X,\mathcal{B})$ is ergodic with respect to the measure $\mu_X$ induced from Lebesgue measure $\mu$.
\end{proposition}


\subsubsection{Estimate of the lower slow entropy}
\paragraph{Outline of the proof:} The estimate of lower slow entropy can be divided into three parts. The first part is Lemma \ref{lem:useTool}, which estimates under some assumptions of codings for the shorter tower the Hamming distance between ``basic blocks'' that are constructed in step \ref{item:step4}. The second part is Lemma \ref{lem:inductionHamming}, which shows by induction that the assumptions in Lemma \ref{lem:useTool} always hold true. The third part is Proposition \ref{prop:towerSequence}, which gives a lower bound for the cardinality of minimal covers along certain subsequences by using Lemma \ref{lem:inductionHamming} and combinatorial estimates.

Let $n_0,n\in\mathbb{Z}^+$, $n_0\geq1000$ sufficiently large such that for any $n\geq n_0$, we have $2^{\frac{n^2}{2}}\geq n^8$. We denote $v_n^{(1)}=\phi_{\mathcal{P}_{n_0},h_n^{(1)}}(x)$ for $x\in B_n^{(1)}$ and $v_n^{(2)}=\phi_{\mathcal{P}_{n_0},h_n^{(2)}}(y)$ for $y\in B_n^{(2)}$, where we recall the notation $\phi_{\mathcal{P}_{n_0},h_n}$ from Section \ref{sec:slowentropy} and that $\mathcal{P}_{n_0}$ is defined as
\begin{equation}\label{eq:finiteRankPartition}
	\mathcal{P}_{n_0}=\{B_{n_0}^{(1)},\ldots,T^{h_{n_0}^{(1)}-1}B_{n_0}^{(1)},B_{n_0}^{(2)},\ldots,T^{h_{n_0}^{(2)}-1}B_{n_0}^{(2)},X\setminus(S_{n_0}^{(1)}\cup S_{n_0}^{(2)})\}.
\end{equation}
Moreover, for any $0\leq k\leq2(n+1)^6-1$, we denote $z_n^{(k)}=v_n^{(1)}0^{kH_{n-1}}$ and   $w_n^{(k)}=(z_n^{(k)})^{\frac{H_n}{h_n^{(1)}+kH_{n-1}}}$. Descriptively,
\begin{equation}\label{eq:w_n}
\begin{aligned}
&\text{the string $z^{(k)}_n$ codes the subtower $\overline{D}^{(k)}_n$ from step \ref{item:step3} of our construction;}\\
&\text{the string $w^{(k)}_n$ corresponds to $D^{(k)}_n$ from step \ref{item:step3} of our construction,}
\end{aligned}
\end{equation}
where the symbol $0$ stands for a newly introduced spacer level.
\begin{definition}
	For any $0\leq i< j\leq2(n+1)^6-1$, we say $w_n^{(i)}$ and $w_n^{(j)}$ are \emph{realigned} at $k$ if $\sh^{kH_{n-1}}(w_n^{(i)})$ starts with $z_n^{(i)}$ and $\sh^{kH_{n-1}}(w_n^{(j)})$ starts with $z_n^{(j)}$.
\end{definition}

The following lemma is a very useful tool in our proof, where $\frac{h_n^{(1)}}{H_{n-1}}$ is a positive integer due to  \eqref{eq:initialEquation}.
\begin{lemma}\label{lem:useTool}
	Let $n\geq n_0$ and $w^{(0)}_n,\ldots,w^{(2(n+1)^6-1)}_n$ be defined as above. For each $1\leq t_n^{(1)}\leq(1-\frac{1}{(n-1)^6})\frac{h_n^{(1)}}{H_{n-1}}$ we define $I_n^{(1)}$ as the interval of indices in the overlap of $v_n^{(1)}$ and $\sh^{t^{(1)}_nH_{n-1}}(v_n^{(1)})$, denote the length of $I_n^{(1)}$ by $L_n^{(1)}$, and suppose that
	\begin{equation}\label{eq:lemUesAssume}
		d^H_{L_n^{(1)}}(v_n^{(1)}\upharpoonright I_n^{(1)},\sh^{t_n^{(1)}H_{n-1}}(v_n^{(1)})\upharpoonright I_n^{(1)})\geq \alpha.
	\end{equation}
	Then we have for any $0\leq i\neq j\leq(n+1)^2-1$, that
	$$d_{H_n}^H(w^{(i)}_n,w^{(j)}_n)\geq(1-\frac{6}{n^2})^2\alpha.$$
\end{lemma}
\begin{proof}
Without loss of generality, we may assume that $i<j$. Denote $$\mathcal{M}=\{m\in[0,\frac{H_n}{H_{n-1}}]\cap\mathbb{Z}:\text{$\sh^{mH_{n-1}}(w_n^{(j)})$ starts with $z_n^{(j)}$}\},$$
$$\mathcal{NRM}=\{m\in[0,\frac{H_n}{H_{n-1}}]\cap\mathbb{Z}:\text{$w_n^{(i)}$ and $w_n^{(j)}$ are not realigned at $m$}\},$$
then we have
\begin{equation}\label{eq:uselem1}
\operatorname{Card}(\mathcal{M})=\frac{H_n}{h_n^{(1)}+jH_{n-1}}.
\end{equation}
Moreover, if $w_n^{(i)}$ and $w_n^{(j)}$ are realigned at both $r_1$ and $r_2$ with $r_1<r_2$, then we have
\begin{equation}\label{eq:uselem2}
\operatorname{Card}(\mathcal{M}\cap[r_1,r_2])\geq \frac{h_n^{(1)}}{2(n+1)^6H_{n-1}}.
\end{equation}
We obtain that
\begin{equation}\label{eq:NRM}
\operatorname{Card}(\mathcal{NRM})\geq(1-\frac{4(n+1)^6H_{n-1}}{h_n^{(1)}})\operatorname{Card}(\mathcal{M}).
\end{equation}

Notice that if $\sh^{mH_{n-1}}(w_n^{(j)})$ starts with $z_n^{(j)}$, then there is a unique $k_m\in\mathbb{Z}$ corresponding to $m$ such that $z_n^{(j)}$ is $k_m$-th $z_n^{(j)}$ in  $w_n^{(j)}$. Then denote by $\mathcal{K}$ the set
\begin{equation*}
\begin{aligned}
\{k_m:\, &\text{$w_n^{(i)}$ and $w_n^{(j)}$ are not realigned at $m$},\\
&a_{k_m}=k_m(j-i)H_{n-1}\mod(h_n^{(1)}+iH_{n-1}) \text{ satisfies } a_{k_m}\in[1,(1-\frac{1}{(n-1)^6})h_n^{(1)}]\}.
\end{aligned}
\end{equation*}
Assumption \eqref{eq:lemUesAssume} implies for any $m\in\mathcal{M}$, that if $k_m\in\mathcal{K}$, then we have
\begin{equation}\label{eq:uselem3}		d_{h_n^{(1)}+iH_{n-1}}^H(\sh^{mH_{n-1}}(w_n^{(i)}),\sh^{mH_{n-1}}(w_n^{(j)}))\geq\frac{h_n^{(1)}-2(n+1)^6H_{n-1}}{h_n^{(1)}+2(n+1)^6H_{n-1}}\alpha.
\end{equation}
where $\frac{h_n^{(1)}-2(n+1)^6H_{n-1}}{h_n^{(1)}+2(n+1)^6H_{n-1}}$ is a lower bound of the proportion of $v_n^{(1)}$ in $w_n^{(i)}$ aligned with a shifted $v_n^{(1)}$ in $w_n^{(j)}$.

Combining \eqref{eq:NRM} and the definition of $\mathcal{K}$, we obtain
\begin{equation}\label{eq:uselem4}
			\frac{\operatorname{Card}(\mathcal{K})}{\operatorname{Card}(\mathcal{M})}\geq1-\frac{1}{(n-1)^6}-\frac{4(n+1)^6H_{n-1}}{h_n^{(1)}}\geq1-\frac{6}{n^2},
\end{equation}
where the summand $-\frac{1}{(n-1)^6}$ comes from the condition $a_{k_m}\in[1,(1-\frac{1}{(n-1)^6})h_n^{(1)}]$.

	Since each $k$ corresponds to a unique $m$, the estimates in \eqref{eq:uselem3}, of the cardinality of realigned $m$, and in \eqref{eq:uselem4} imply
	\begin{equation}
		\begin{aligned}
			d_{H_n}^H(w_n^{(i)},w_n^{(j)})&\geq\frac{\operatorname{Card}(\mathcal{K})}{\operatorname{Card}(\mathcal{M})}\frac{h_n^{(1)}-2(n+1)^6H_{n-1}}{h_n^{(1)}+2(n+1)^6H_{n-1}}\alpha\\
			&\geq(1-\frac{6}{n^2})(1-\frac{2}{(n+1)^{14}})\alpha\\
			&\geq(1-\frac{6}{n^2})^2\alpha.
		\end{aligned}
	\end{equation}

\end{proof}

The following lemma is the first step to establish the estimate of lower slow entropy. Recall that $\frac{h_n^{(1)}}{H_{n-1}}$ and $\frac{h_n^{(2)}}{H_{n-1}}$ are positive integers due to \eqref{eq:initialEquation}.
\begin{lemma}\label{lem:inductionHamming}
	For any $n\geq n_0$, $1\leq t_n^{(1)}\leq(1-\frac{1}{(n-1)^2})\frac{h_n^{(1)}}{H_{n-1}}$, $1\leq t_n^{(2)}\leq(1-\frac{1}{n^2})\frac{h_n^{(2)}}{H_{n-1}}$ and $1\leq t_n^{(3)}\leq\frac{h_n^{(2)}-h_n^{(1)}}{H_{n-1}}$, let
\begin{enumerate}
\item $I_n^{(1)}$ be the interval of indices in the overlap of $v_n^{(1)}$ and $\sh^{t_n^{(1)}H_{n-1}}(v_n^{(1)})$,
\item $I_n^{(2)}$ be the interval of indices in the overlap of $v_n^{(2)}$ and $\sh^{t_n^{(2)}H_{n-1}}(v_n^{(2)})$,
\item $I_n^{(3)}$ be the interval of indices in the overlap of $v_n^{(1)}$ and $\sh^{t_n^{(3)}H_{n-1}}(v_n^{(2)})$.
\end{enumerate}
Denote the length of $I_n^{(s)}$ by $L_n^{(s)}$ for $s=1,2,3$. Then we have for $0\leq i\neq j\leq 2(n+1)^6-1$:
\begin{equation}\label{eq:ineq1}
\begin{aligned}
&d^H_{L_n^{(1)}}(v_n^{(1)}\upharpoonright I_n^{(1)},\sh^{t_n^{(1)}H_{n-1}}(v_n^{(1)})\upharpoonright I_n^{(1)})\geq \prod_{m=n_0}^{2n}(1-\frac{1600}{m^2})^5,\\
\end{aligned}
\end{equation}	
\begin{equation}\label{eq:ineq2}
\begin{aligned}
&d^H_{L_n^{(2)}}(v_n^{(2)}\upharpoonright I_n^{(2)},\sh^{t_n^{(2)}H_{n-1}}(v_n^{(2)})\upharpoonright I_n^{(2)})\geq \prod_{m=n_0}^{2n}(1-\frac{1600}{m^2})^5,\\
\end{aligned}
\end{equation}
\begin{equation}\label{eq:ineq3}
\begin{aligned}
&d_{h_n^{(1)}}^H(v_n^{(1)}\upharpoonright I_n^{(3)},\sh^{t_n^{(3)}H_{n-1}}(v_n^{(2)})\upharpoonright I_n^{(3)})\geq \prod_{m=n_0}^{2n}(1-\frac{1600}{m^2})^5,\\
\end{aligned}
\end{equation}
\begin{equation}\label{eq:ineq4}
\begin{aligned}
&d_{H_n}^H(w^{(i)}_n,w^{(j)}_n)\geq\prod_{m=n_0}^{2n+1}(1-\frac{1600}{m^2})^5.
\end{aligned}
\end{equation}
\end{lemma}
\begin{proof}
	We formulate our proof by induction.
	
	\paragraph{$n=n_0$:} By the definition of $\mathcal{P}_{n_0}$, $n_0\geq3$ and $v^{(s)}_{n_0}$ for $s=1,2$, we have for $1\leq t_{n_0}^{(1)}\leq(1-\frac{1}{({n_0}-1)^2})\frac{h_{n_0}^{(1)}}{H_{n_0-1}}$ and $1\leq t_{n_0}^{(2)}\leq(1-\frac{1}{{n_0}^2})\frac{h_{n_0}^{(2)}}{H_{n_0-1}}$:	$$d^H_{L_{n_0}^{(s)}}(v_{n_0}^{(s)}\upharpoonright I_{n_0}^{(s)},\sh^{t_{n_0}^{(s)}H_{n_0-1}}(v_{n_0}^{(s)})\upharpoonright I_{n_0}^{(s)})=1.$$
Moreover, the definition of $\mathcal{P}_{n_0}$ implies that for $1\leq t_{n_0}^{(3)}\leq\frac{h_{n_0}^{(2)}-h_{n_0}^{(1)}}{H_{n_0-1}}$, we have
$$d^H_{h_{n_0}^{(1)}}(v_{n_0}^{(1)}\upharpoonright I_{n_0}^{(3)},\sh^{t_{n_0}^{(3)}H_{n_0-1}}(v_{n_0}^{(2)})\upharpoonright I_{n_0}^{(3)})=1.$$

	By Lemma \ref{lem:useTool} with $n=n_0$, we obtain for $0\leq i\neq j\leq2(n_0+1)^6-1$ $$d_{H_{n_0}}^{H}(w^{(i)}_{n_0},w^{(j)}_{n_0})\geq(1-\frac{6}{n_0})^2.$$

	\paragraph{\eqref{eq:ineq2} for $n=\ell+1$:} Suppose our lemma holds for any $n_0\leq n\leq\ell$, then we obtain
	\begin{equation}\label{eq:wInduction1} d_{H_\ell}^H(w_\ell^{(i)},w_\ell^{(j)})\geq\prod_{m=n_0}^{2\ell+1}(1-\frac{1600}{m^2})^5
	\end{equation}
	for any $0\leq i<j\leq2(\ell+1)^6-1$ by induction  assumption. Recall that $v_{\ell+1}^{(2)}$ only depends on $w_\ell^{(0)},\ldots,w_\ell^{((\ell+1)^6-1)}$ by step \ref{item:step6}. Denote the $N_{\ell}$ different words obtained at step \ref{item:step6} by $u_{\ell}^{(1)},\ldots,u_{\ell}^{(N_{\ell})}$, then Proposition \ref{prop:prob} (2) and \eqref{eq:wInduction1} imply that
	\begin{equation}\label{eq:differentWord}
		d_{R_{\ell}H_{\ell}}^H(u_\ell^{(i)},u_\ell^{(j)})\geq \left(1-\frac{2}{(\ell+1)^6}\right) \cdot \prod_{m=n_0}^{2\ell+1}(1-\frac{1600}{m^2})^5
	\end{equation}
	for any $0\leq i\neq j\leq N_{\ell}-1$ and $R_{\ell}$ defined in step \ref{item:step6}.
	Fix $0\leq i,j\leq N_{\ell}-1$ and recall the conditions on $R_\ell$ in \eqref{eq:heightHn23}. Moreover, let $t_\ell^{(4)}\in[1,(1-\frac{1}{(\ell+1)^6})R_{\ell}]\cap\mathbb{Z}$ and $I_{\ell}^{(4)}$ be the interval of indices in the overlap of $u_\ell^{(i)}$ and $\sh^{t_{\ell}^{(4)}H_{\ell}}(u_\ell^{(j)})$, whose length is denoted by $L_\ell^{(4)}$, then $L_\ell^{(4)}\geq\frac{1}{(\ell+1)^6}R_\ell H_{\ell}$, \eqref{eq:wInduction1} and Proposition \ref{prop:prob} (2) imply
	\begin{equation}\label{eq:wInduction2}
		\begin{aligned}
			&d^H_{L_\ell^{(4)}}(u_\ell^{(i)}\upharpoonright I_\ell^{(4)},\sh^{t_\ell^{(4)}H_{\ell}}(u_\ell^{(j)})\upharpoonright I_\ell^{(4)})\geq(1-\frac{2}{(\ell+1)^6}) \prod_{m=n_0}^{2\ell+1}(1-\frac{1600}{m^2})^5.
		\end{aligned}
	\end{equation}
	Recall that $S_{\ell+1}^{(2)}$ is constructed by stacking words $u_\ell^{(1)},\ldots,u_\ell^{(N_\ell)}$ above each other and $I_{\ell+1}^{(2)},L_{\ell+1}^{(2)}$ are defined by induction as in the statement of the lemma. Then $1\leq t_{\ell+1}^{(2)}\leq(1-\frac{1}{(\ell+1)^2})\frac{h_{\ell+1}^{(2)}}{H_{\ell}}$ implies $L_{\ell+1}^{(2)}\geq\frac{1}{(\ell+1)^2}h_{\ell+1}^{(2)}$, which together with \eqref{eq:differentWord} and \eqref{eq:wInduction2} guarantee
	\begin{equation}
		\begin{aligned}			
&d_{L_{\ell+1}^{(2)}}^H(v_{\ell+1}^{(2)}\upharpoonright I_{\ell+1}^{(2)},\sh^{t_{\ell+1}^{(2)}H_{\ell}}(v_{\ell+1}^{(2)})\upharpoonright I_{\ell+1}^{(2)})\\
&\geq (1-\frac{2}{(\ell+1)^6})(1-\frac{2}{(\ell+1)^6}) \prod_{m=n_0}^{2\ell+1}(1-\frac{1600}{m^2})^5\\ &\geq\prod_{m=n_0}^{2\ell+2}(1-\frac{1600}{m^2})^5,
		\end{aligned}
	\end{equation}
where the first factor comes from the case that $t_{\ell+1}^{(2)}H_{\ell} \mod R_\ell H_{\ell}$ is too small or too large. As a result, we complete the induction step for \eqref{eq:ineq2} of our lemma.

\paragraph{\eqref{eq:ineq1} for $n=\ell+1$:} Denote the $g_{\ell}(\ell+1)^6$ different words obtained at step \ref{item:step5} with $n=\ell$ as $\bar{u}_\ell^{(0)},\ldots,\bar{u}_{\ell}^{(g_{\ell}(\ell+1)^6-1)}$. Recall that step \ref{item:step5} implies
\begin{equation}\label{eq:v1}
v_{\ell+1}^{(1)}=v_{\ell}^{(2)}\bar{v}_{\ell},
\end{equation}
	where $\bar{v}_{\ell}=\bar{u}_\ell^{(0)}\ldots\bar{u}_{\ell}^{(g_{\ell}(\ell+1)^6-1)}$. Then for any $1\leq t_{\ell+1}^{(1)}\leq(1-\frac{1}{l^2})\frac{h_{\ell+1}^{(1)}}{H_{\ell}}$ and for $I_{\ell+1}^{(1)},L_{\ell+1}^{(1)}$ defined as in the statement of the lemma for $v_{\ell+1}^{(1)}$ and $\sh^{t_{\ell+1}^{(1)}H_{\ell}}(v_{\ell+1}^{(1)})$, we decompose $I_{\ell+1}^{(1)}$ as the union three intervals $I_{\ell+1}^{(1,1)},I_{\ell+1}^{(1,2)},I_{\ell+1}^{(1,3)}$, where
\begin{enumerate}[label=(\alph*)]
		\item $I_{\ell+1}^{(1,1)}$ is the interval of indices in the overlap of $v_\ell^{(2)}$ and $\sh^{t_{\ell+1}^{(1)}H_{\ell}}(v_\ell^{(2)})$, whose length is denoted by $L_{\ell+1}^{(1,1)}$;
		\item $I_{\ell+1}^{(1,2)}$ is the interval of indices in the overlap of $v_\ell^{(2)}$ and $\sh^{t_{\ell+1}^{(1)}H_{\ell}-h_{\ell}^{(2)}}(\bar{v}_{\ell})$, whose length is denoted by $L_{\ell+1}^{(1,2)}$;
		\item $I_{\ell+1}^{(1,3)}$ is the interval of indices in the overlap of $\sh^{-h_{\ell}^{(2)}}(\bar{v}_{\ell})$ and $\sh^{t_{\ell+1}^{(1)}H_{\ell}-h_{\ell}^{(2)}}(\bar{v}_{\ell})$, whose length is denoted by $L_{\ell+1}^{(1,3)}$.
\end{enumerate}
We examine the situation on each of these three intervals:
\begin{enumerate}
\item If  $L_{\ell+1}^{(1,1)}\geq\frac{1}{\ell^2}h_{\ell}^{(2)}$, then by induction assumption \eqref{eq:ineq2} and $H_{\ell-1}|H_{\ell}$, we have
	\begin{equation}\label{eq:induction1} d_{L_{\ell+1}^{(1,1)}}^H(v_\ell^{(2)}\upharpoonright I_{\ell+1}^{(1,1)},\sh^{t_{\ell+1}^{(1)}H_{\ell}}(v_\ell^{(2)})\upharpoonright I_{\ell+1}^{(1,1)})\geq\prod_{m=n_0}^{2\ell}(1-\frac{1600}{m^2})^5.
	\end{equation}
\item If  $L_{\ell+1}^{(1,2)}\geq\frac{1}{\ell^2}h_{\ell}^{(2)}$, then \eqref{eq:heightHn23} together with $H_{\ell}|h_{\ell}^{(2)}$ by \eqref{eq:initialEquation} yields
\begin{equation}\label{eq:hamEst1}
\begin{aligned}
&\sh^{t_{\ell+1}^{(1)}H_{\ell}-h_{\ell}^{(2)}}(\bar{v}_{\ell})\upharpoonright I_{\ell+1}^{(1,2)}=\tilde{w}_{\ell}^{(0)}\tilde{w}_{\ell}^{(1)}\ldots \tilde{w}_{\ell}^{(p)},\\
&\text{each } \tilde{w}_{\ell}^{(k)}\text{ is equal to one of }w_{\ell}^{((\ell+1)^6)},\ldots,w_{\ell}^{(2(\ell+1)^6-1)},\\
&p\geq(\ell+1)^6.\\
\end{aligned}
\end{equation}
Since the heights of spacers that we add in step \ref{item:step2} are multiples of $H_{\ell-1}$, for any $1\leq k\leq p-1$ every $v_{\ell}^{(1)}$ substring of $\tilde{w}_{\ell}^{(k)}$ is Hamming-matching to  $\sh^{t_{\ell}^{(5)}H_{\ell-1}}(v_\ell^{(2)})\upharpoonright[0,h_{\ell}^{(1)}-1]$ for some $1\leq t_{\ell}^{(5)}\leq\frac{h_{\ell}^{(2)}-h_{\ell}^{(1)}}{H_{\ell-1}}$.
Recall that by induction assumption \eqref{eq:ineq3} we have for any $1\leq t_{\ell}^{(5)}\leq\frac{h_{\ell}^{(2)}-h_{\ell}^{(1)}}{H_{\ell-1}}$ that
\begin{equation}
d_{h_{\ell}^{(1)}}^H(v_\ell^{(1)}\upharpoonright I_\ell^{(5)},\sh^{t_\ell^{(5)}H_{\ell-1}}(v_\ell^{(2)})\upharpoonright I_\ell^{(5)})\geq \prod_{m=n_0}^{2\ell}(1-\frac{1600}{m^2})^5,
\end{equation}
where $I_{\ell}^{(5)}$ is the interval of indices in the overlap of $v_{\ell}^{(1)}$ and $\sh^{t_{\ell}^{(5)}H_{\ell-1}}(v_{\ell}^{(2)})$. Moreover, if we remove all $v_{\ell}^{(1)}$ substrings from $\tilde{w}_{\ell}^{(k)}$, there is a proportion of at most $\frac{2(\ell+1)^6H_{\ell-1}}{h_\ell^{(1)}+2(\ell+1)^6H_{\ell-1}}$ left, which is less than $\frac{1}{(\ell+1)^{14}}$ by \eqref{eq:initialEquation}. Combining all these observations, we obtain that
\begin{equation}\label{eq:induction22}
	\begin{aligned}
&d^H_{L_{\ell+1}^{(1,2)}}(v_{\ell}^{(2)}\upharpoonright I_{\ell+1}^{(1,2)},\sh^{t_{\ell+1}^{(1)}H_{\ell}-h_{\ell}^{(2)}}(\bar{v}_{\ell})\upharpoonright I_{\ell+1}^{(1,2)})\\
&\geq(1-\frac{2}{(\ell+1)^6})(1-\frac{1}{(\ell+1)^{14}})\prod_{m=n_0}^{2\ell}(1-\frac{1600}{m^2})^5.
\end{aligned}
\end{equation}

\item If $L_{\ell+1}^{(1,3)}\geq\frac{1}{\ell^2}h_{\ell}^{(2)}$, then by step \ref{item:step5} and Proposition \ref{prop:prob} we have
\begin{equation}\label{eq:induction31}
\begin{aligned} d_{L_{\ell+1}^{(1,3)}}^H(\sh^{-h_{\ell}^{(2)}}(\bar{v}_\ell)\upharpoonright I_{\ell+1}^{(1,3)},\sh^{t_{\ell+1}^{(1)}H_{\ell}-h_{\ell}^{(2)}}(\bar{v}_{\ell})\upharpoonright I_{\ell+1}^{(1,3)}) \geq(1-\frac{2}{(\ell+1)^6})\prod_{m=n_0}^{2\ell+1}(1-\frac{1600}{m^2})^5.
\end{aligned}
\end{equation}
\end{enumerate}

Now we can finish the inductive proof of inequality \eqref{eq:ineq1} of our lemma. In order to simplify the notation, we define
\begin{equation}
\begin{aligned}
&d_1=d_{L_{\ell+1}^{(1,1)}}^H(v_\ell^{(2)}\upharpoonright I_{\ell+1}^{(1,1)},\sh^{t_{\ell+1}^{(1)}H_{\ell}}(v_\ell^{(2)})\upharpoonright I_{\ell+1}^{(1,1)}),\\ &d_2=d_{L_{\ell+1}^{(1,2)}}^H(v_\ell^{(2)}\upharpoonright I_{\ell+1}^{(1,2)},\sh^{t_{\ell+1}^{(1)}H_{\ell}-h_{\ell}^{(2)}}(\bar{v}_{\ell})\upharpoonright I_{\ell+1}^{(1,2)}),\\ &d_3=d_{L_{\ell+1}^{(1,3)}}^H(\sh^{-h_{\ell}^{(2)}}(\bar{v}_\ell)\upharpoonright I_{\ell+1}^{(1,3)},\sh^{t_{\ell+1}^{(1)}H_{\ell}-h_{\ell}^{(2)}}(\bar{v}_{\ell})\upharpoonright I_{\ell+1}^{(1,3)}).
\end{aligned}
\end{equation}
Then we have
\begin{equation}\label{eq:dSummary}
d_{L_{\ell+1}^{(1)}}^H(v_{\ell+1}^{(1)}\upharpoonright I_{\ell+1}^{(1)},\sh^{t_{\ell+1}^{(1)}H_{\ell}}(v_{\ell+1}^{(1)})\upharpoonright I_{\ell+1}^{(1)})=\frac{L_{\ell+1}^{(1,1)}}{L_{\ell+1}^{(1)}}d_1+\frac{L_{\ell+1}^{(1,2)}}{L_{\ell+1}^{(1)}}d_2+\frac{L_{\ell+1}^{(1,3)}}{L_{\ell+1}^{(1)}}d_3.
\end{equation}
We have the following cases:
\begin{enumerate}
  \item If $L_{\ell+1}^{(1,1)}\in[0,\frac{1}{\ell^2}h_{\ell}^{(2)})$, then we have $L_{\ell+1}^{(1,3)}=0$, which is due to $\ell\geq1000$ and $\frac{l(\bar{v}_\ell)}{h_{\ell}^{(2)}}\in(\frac{1}{400},\frac{1}{20})$ by \eqref{eq:initialEquation} and step \ref{item:step2}. Together with our assumption $L_{\ell+1}^{(1)}=L_{\ell+1}^{(1,1)}+L_{\ell+1}^{(1,2)}+L_{\ell+1}^{(1,3)}\geq\frac{1}{\ell^2}h_{\ell+1}^{(1)}$, this implies that we are in one of the following two cases:
      \begin{enumerate}
        \item $L_{\ell+1}^{(1,1)}=0$ and $L_{\ell+1}^{(1,2)}=L_{\ell+1}^{(1)}\geq\frac{1}{\ell^2}h_{\ell+1}^{(1)}>\frac{1}{\ell^2}h_{\ell}^{(2)}$;
        \item $L_{\ell+1}^{(1,1)}>0$ and $L_{\ell+1}^{(1,2)}=l(\bar{v}_\ell)>\frac{1}{400}h_{\ell}^{(2)}>\frac{1}{\ell^2}h_{\ell}^{(2)}$.
      \end{enumerate}
       As a result, \eqref{eq:dSummary} and \eqref{eq:induction22} guarantee that
      \begin{equation}\label{eq:case1}
      \begin{aligned}
&d_{L_{\ell+1}^{(1)}}^H(v_{\ell+1}^{(1)}\upharpoonright I_{\ell+1}^{(1)},\sh^{t_{\ell+1}^{(1)}H_{\ell}}(v_{\ell+1}^{(1)})\upharpoonright I_{\ell+1}^{(1)})\\
\geq&\frac{L_{\ell+1}^{(1,2)}}{L_{\ell+1}^{(1)}}(1-\frac{2}{(\ell+1)^6})(1-\frac{1}{(\ell+1)^{14}})\prod_{m=n_0}^{2\ell}(1-\frac{1600}{m^2})^5 \\
      \geq&\frac{\frac{1}{400}h_{\ell}^{(2)}}{(\frac{1}{400}+\frac{1}{\ell^2})h_{\ell}^{(2)}}(1-\frac{2}{(\ell+1)^6})(1-\frac{1}{(\ell+1)^{14}})\prod_{m=n_0}^{2\ell}(1-\frac{1600}{m^2})^5\\
      \geq&(1-\frac{400}{\ell^2})(1-\frac{2}{(\ell+1)^6})(1-\frac{1}{(\ell+1)^{14}})\prod_{m=n_0}^{2\ell}(1-\frac{1600}{m^2})^5\\
      \geq&\prod_{m=n_0}^{2\ell+2}(1-\frac{1600}{m^2})^5
      \end{aligned}
      \end{equation}
  \item If $L_{\ell+1}^{(1,1)}\geq\frac{1}{\ell^2}h_{\ell}^{(2)}$, then we have the following cases:
      \begin{enumerate}
      \item Suppose $L_{\ell+1}^{(1,2)}<\frac{1}{\ell^2}h_{\ell}^{(2)}$. Since $I_{\ell+1}^{(1,1)}, I_{\ell+1}^{(1,2)},I_{\ell+1}^{(1,3)}$ are consecutive intervals and $l(\bar{v}_{\ell})>\frac{1}{400}h_{\ell}^{(2)}$, we get $t_{\ell+1}^{(1)}H_{\ell}<\frac{1}{\ell^2}h_{\ell}^{(2)}$, which yields
          \begin{equation}\label{eq:Lesti}
          L_{\ell+1}^{(1,1)}\geq h_{\ell}^{(2)}-\frac{1}{\ell^2}h_{\ell}^{(2)},\ \ \ L_{\ell+1}^{(1,3)}=l(\bar{v}_{\ell})-L_{\ell+1}^{(1,2)}.
          \end{equation}
          Combining \eqref{eq:induction1} and \eqref{eq:Lesti}, we obtain
          \begin{equation*}
          \begin{aligned}
          \frac{L_{\ell+1}^{(1,1)}}{L_{\ell+1}^{(1,1)}+L_{\ell+1}^{(1,2)}}d_1\geq\frac{h_{\ell}^{(2)}-\frac{1}{\ell^2}h_{\ell}^{(2)}}{h_{\ell}^{(2)}}\prod_{m=n_0}^{2\ell}(1-\frac{1600}{m^2})^5=(1-\frac{1}{\ell^2})\prod_{m=n_0}^{2\ell}(1-\frac{1600}{m^2})^5,
          \end{aligned}
          \end{equation*}
          which together with \eqref{eq:dSummary} and \eqref{eq:induction31} gives
          \begin{equation}\label{eq:case2}
          d_{L_{\ell+1}^{(1)}}^H(v_{\ell+1}^{(1)}\upharpoonright I_{\ell+1}^{(1)},\sh^{t_{\ell+1}^{(1)}H_{\ell}}(v_{\ell+1}^{(1)})\upharpoonright I_{\ell+1}^{(1)})\geq\prod_{m=n_0}^{2\ell+2}(1-\frac{1600}{m^2})^5.
          \end{equation}
      \item Suppose $L_{\ell+1}^{(1,2)}\geq\frac{1}{\ell^2}h_{\ell}^{(2)}$. Then we have the following two subcases:
      \begin{enumerate}
        \item If $L^{(1,3)}_{\ell+1}\geq\frac{1}{\ell^2}h_{\ell}^{(2)}$, then \eqref{eq:dSummary}, \eqref{eq:induction1},  \eqref{eq:induction22} and \eqref{eq:induction31} imply
            \begin{equation}\label{eq:case3}
            d_{L_{\ell+1}^{(1)}}^H(v_{\ell+1}^{(1)}\upharpoonright I_{\ell+1}^{(1)},\sh^{t_{\ell+1}^{(1)}H_{\ell}}(v_{\ell+1}^{(1)})\upharpoonright I_{\ell+1}^{(1)})\geq\prod_{m=n_0}^{2\ell+2}(1-\frac{1600}{m^2})^5.
            \end{equation}
        \item If $L^{(1,3)}_{\ell+1}<\frac{1}{\ell^2}h_{\ell}^{(2)}$, then we have $L_{\ell+1}^{(1,2)}\geq l(\bar{v}_\ell)-\frac{1}{\ell^2}h_{\ell}^{(2)}$, which together with \eqref{eq:induction22} and $l(\bar{v}_\ell)>\frac{1}{400}h_{\ell}^{(2)}$ imply
            \begin{equation}\label{eq:inductionKey}
            \begin{aligned}
            &\frac{L_{\ell+1}^{(1,2)}}{L_{\ell+1}^{(1,2)}+L_{\ell+1}^{(1,3)}}d_2 \\
            \geq&\frac{l(\bar{v}_\ell)-\frac{1}{\ell^2}h_{\ell}^{(2)}}{l(\bar{v}_\ell)}(1-\frac{2}{(\ell+1)^6})(1-\frac{1}{(\ell+1)^{14}})\prod_{m=n_0}^{2\ell}(1-\frac{1600}{m^2})^5\\
            \geq&\prod_{m=n_0}^{2\ell+2}(1-\frac{1600}{m^2})^5.
            \end{aligned}
            \end{equation}
             Combining \eqref{eq:dSummary}, \eqref{eq:induction1} and \eqref{eq:inductionKey}, we obtain
            \begin{equation}\label{eq:case4}
            \begin{aligned}
            d_{L_{\ell+1}^{(1)}}^H(v_{\ell+1}^{(1)}\upharpoonright I_{\ell+1}^{(1)},\sh^{t_{\ell+1}^{(1)}H_{\ell}}(v_{\ell+1}^{(1)})\upharpoonright I_{\ell+1}^{(1)})\geq \prod_{m=n_0}^{2\ell+2}(1-\frac{1600}{m^2})^5.
            \end{aligned}
            \end{equation}
      \end{enumerate}
      \end{enumerate}

\end{enumerate}
Combining \eqref{eq:case1}, \eqref{eq:case2}, \eqref{eq:case3} and \eqref{eq:case4}, we complete the inductive proof of inequality \eqref{eq:ineq1} of our lemma.
	
\paragraph{\eqref{eq:ineq3} for $n=\ell+1$:}Suppose $1\leq t_{\ell+1}^{(3)}\leq\frac{h_{\ell+1}^{(2)}-h_{\ell+1}^{(1)}}{H_\ell}$, we define $I_{\ell+1}^{(3)}, L_{\ell+1}^{(3)}$ as in the statement of the lemma for $v_{\ell+1}^{(1)}$ and $\sh^{t_{\ell+1}^{(3)}H_{\ell}}(v_{\ell+1}^{(2)})$. Recall $v_{\ell+1}^{(1)}=v_{\ell}^{(2)}\bar{v}_{\ell}$ from \eqref{eq:v1}. We decompose $I_{\ell+1}^{(3)}$ as the union of two intervals $I_{\ell+1}^{(3,1)},I_{\ell+1}^{(3,2)}$, where
\begin{enumerate}
  \item $I_{\ell+1}^{(3,1)}$ is the interval of indices in the overlap of $v_{\ell}^{(2)}$ and $\sh^{t_{\ell+1}^{(3)}H_{\ell}}(v_{\ell+1}^{(2)})$, whose length is denoted by $L_{\ell+1}^{(3,1)}$;
  \item $I_{\ell+1}^{(3,2)}$ is the interval of indices in the overlap of $\sh^{-h_{\ell}^{(2)}}(\bar{v}_{\ell})$ and $\sh^{t_{\ell+1}^{(3)}H_{\ell}}(v_{\ell+1}^{(2)})$, whose length is denoted by $L_{\ell+1}^{(3,2)}$.
\end{enumerate}
Indeed, we have $L_{\ell+1}^{(3,1)}=h_{\ell}^{(2)}$ and $L_{\ell+1}^{(3,2)}$ is the length of $\bar{v}_{\ell}$.

Since we shift by multiples of $H_{\ell}$, step \ref{item:step6} and \eqref{eq:heightHn23} imply
\begin{equation}\label{eq:decompose1}
\begin{aligned}
&\sh^{t_{\ell+1}^{(3)}H_{\ell}}(v_{\ell+1}^{(2)})\upharpoonright I_{\ell+1}^{(3,1)}=\bar{w}_{\ell}^{(0)}\bar{w}_{\ell}^{(1)}\ldots \bar{w}_{\ell}^{(q)},\\
&\text{each } \bar{w}_{\ell}^{(k)}\text{ is equal to one of the }w_{\ell}^{(0)},\ldots,w_{\ell}^{((\ell+1)^6-1)},\\
&q\geq(\ell+1)^8.\\
\end{aligned}
\end{equation}
Since the heights of spacers that we add in step \ref{item:step2} are multiples of $H_{\ell-1}$, then for any $1\leq k\leq q-1$ every $v_{\ell}^{(1)}$ substring of $\bar{w}_{\ell}^{(k)}$ is Hamming-matching $\sh^{t_{\ell}^{(3)}H_{\ell-1}}(v_{\ell}^{(2)})\upharpoonright[0,h_{\ell}^{(1)}-1]$ for some $1\leq t_{\ell}^{(3)}\leq\frac{h_{\ell}^{(2)}-h_{\ell}^{(1)}}{H_{\ell-1}}$. By our induction assumption of \eqref{eq:ineq3} for $n=\ell$, we have for any $1\leq t_{\ell}^{(3)}\leq\frac{h_{\ell}^{(2)}-h_{\ell}^{(1)}}{H_{\ell-1}}$ that
\begin{equation}\label{eq:decompose2}
d_{h_{\ell}^{(1)}}^H(v_\ell^{(1)}\upharpoonright I_\ell^{(3)},\sh^{t_\ell^{(3)}H_{\ell-1}}(v_\ell^{(2)})\upharpoonright I_\ell^{(3)})\geq \prod_{m=n_0}^{2\ell}(1-\frac{1600}{m^2})^5,
\end{equation}
where $I_{\ell}^{(3)}$ is the interval of indices in the overlap of $v_{\ell}^{(1)}$ and $\sh^{t_{\ell}^{(3)}H_{\ell-1}}(v_{\ell}^{(2)})$. If we remove all $v_{\ell}^{(1)}$ substrings from $\bar{w}_{\ell}^{(k)}$, there is a proportion of at most $\frac{2(\ell+1)^6H_{\ell-1}}{h_\ell^{(1)}+2(\ell+1)^6H_{\ell-1}}$ left, which is less than $\frac{1}{(\ell+1)^{14}}$ by \eqref{eq:initialEquation}. Combining all these observations, we have
\begin{equation}\label{eq:shiftHam1}
d_{L_{\ell+1}^{(3,1)}}^H(v_{\ell+1}^{(1)}\upharpoonright I_{\ell+1}^{(3,1)},\sh^{t_{\ell+1}^{(3)}H_{\ell}}(v_{\ell+1}^{(2)})\upharpoonright I_{\ell+1}^{(3,1)})\geq(1-\frac{2}{(\ell+1)^8})(1-\frac{1}{(\ell+1)^{14}})\prod_{m=n_0}^{2\ell}(1-\frac{1600}{m^2})^5.
\end{equation}
Furthermore, we recall that $v_{\ell+1}^{(2)}$ is built with $w_{\ell}^{(0)},\ldots,w_{\ell}^{((\ell+1)^6-1)}$ while $\bar{v}_{\ell}$ is built with   $w_{\ell}^{((\ell+1)^6)},\ldots,w_{\ell}^{(2(\ell+1)^6-1)}$. Then \eqref{eq:wInduction1} and $H_{\ell}|h_{\ell}^{(2)}$ guarantee that
\begin{equation*}
d_{L_{\ell+1}^{(3,2)}}^H(\sh^{-h_{\ell}^{(2)}}(\bar{v}_{\ell})\upharpoonright I_{\ell+1}^{(3,2)},\sh^{t_{\ell+1}^{(3)}H_{\ell}}(v_{\ell+1}^{(2)})\upharpoonright I_{\ell+1}^{(3,2)})\geq \prod_{m=n_0}^{2\ell+1}(1-\frac{1600}{m^2})^5,
\end{equation*}
which is equivalent to
\begin{equation}\label{eq:shiftHam2}
d_{L_{\ell+1}^{(3,2)}}^H(v_{\ell+1}^{(1)}\upharpoonright I_{\ell+1}^{(3,2)},\sh^{t_{\ell+1}^{(3)}H_{\ell}}(v_{\ell+1}^{(2)})\upharpoonright I_{\ell+1}^{(3,2)})\geq \prod_{m=n_0}^{2\ell+1}(1-\frac{1600}{m^2})^5.
\end{equation}
Combining \eqref{eq:shiftHam1} and \eqref{eq:shiftHam2}, we complete the induction step for \eqref{eq:ineq3}.

\paragraph{\eqref{eq:ineq4} for $n=\ell+1$:} For the remaining inequality of our lemma we applying Lemma \ref{lem:useTool} with $n=\ell+1$ and notice that $(1-\frac{6}{(\ell+1)^2})^2\geq(1-\frac{600}{(2\ell+3)^2})^2$. Hereby, we complete our proof.
\end{proof}

With the help of Lemma \ref{lem:inductionHamming}, we are able to estimate the number of Hamming balls we need to cover a large portion of the space along a given subsequence. Recall that in step \ref{item:step6} at stage $n$ of our construction we apply Proposition \ref{prop:prob} with  $\gamma=\frac{1}{(n+1)^{12}}$ and $\epsilon=\frac{1}{1000}$. Then we have the following estimate:
\begin{proposition}\label{prop:towerSequence}
	Let $H_n$, $\tau_n$ and $R_n$ be defined as in the construction. For any $\epsilon\in(0,\frac{1}{1000})$, there exists $N_1\in\mathbb{N}$ such that for any  $n_0\geq N_1$ we have for every $n\geq n_0$ and $k\in[\frac{7}{(n+1)^{12}}R_n,R_n]\cap\mathbb{Z}$:
	\begin{equation}\label{eq:finiteLower1}
		S_{\mathcal{P}_{n_0}}(T,kH_n,\frac{\epsilon}{2})\geq \frac{1}{5}a_{(R_n+1)H_{n}}(t_{n}),
	\end{equation}
	and for any $k\in[\tau_n,\frac{7}{(n+1)^{12}}R_{n})\cap\mathbb{Z}$:
	\begin{equation}\label{eq:finiteLower2}
		S_{\mathcal{P}_{n_0}}(T,kH_{n},\frac{3}{4}\epsilon)\geq \frac{1}{6}a_{(k+1)H_{n}}(t_{n}).
	\end{equation}
\end{proposition}
\begin{proof}
	Let $N_1$ be sufficiently large such that $\prod_{m=N_1}^{\infty}(1-\frac{600}{m^2})^5>\frac{3}{4}$. For any $n\geq n_0$ and $k\in[K_n,R_n]$, we denote the $N_n$ different words obtained at step \ref{item:step6} as $\{u_i\}_{i=1}^{N_n}$. Then there exists a subset $G_{n+1}^{(2)}\subset S_{n+1}^{(2)}$ with $\mu(G_{n+1}^{(2)})\geq(1-\frac{1}{N_n})\mu(S_{n+1}^{(2)})$ such that for any $x\in G_{n+1}^{(2)}$ its symbolic name $\phi_{\mathcal{P}_{n_0},kH_n}(x)$ of length $kH_n$ with respect to $\mathcal{P}_{n_0}$ is a substring of $u_iu_{i+1}$ with length $kH_n$, where $i=1,\ldots,N_n-1$. 
	We decompose the set $G_{n+1}^{(2)}$ into $H_n$ different subsets $G_{n+1}^{(2,s)}$ for $s=0,\ldots,H_n-1$ defined as follows:
\begin{equation}
\begin{aligned}
G_{n+1}^{(2,s)}=\{x\in G_{n+1}^{(2)}:\ & \phi_{\mathcal{P}_{n_0},kH_n}(x)\upharpoonright[s,(k-1)H_n+s]\text{ is a concatenation }\\
&\text{of } w_n^{(0)},\ldots,w_n^{((n+1)^6-1)}\},
\end{aligned}
\end{equation}
where $w_n^{(k)}$ is defined in \eqref{eq:w_n}. Suppose $x,y\in G_{n+1}^{(2,s)}$ for some $s=0,1,\ldots,H_n-1$. Let $v_x,v_y$ be defined as
\begin{equation}\label{eq:v1v2}
v_x=\phi_{\mathcal{P}_{n_0},kH_n}(x)\upharpoonright[s,(k-1)H_n+s], \ \ \ \ v_y=\phi_{\mathcal{P}_{n_0},kH_n}(y)\upharpoonright[s,(k-1)H_n+s].
\end{equation}
Then we have
\begin{equation}\label{eq:seperateWordsCon}
\begin{aligned}
&\text{$v_x$ is a substring of $u_iu_{i+1}$ with length $(k-1)H_n$,}\\
&\text{$v_y$ is a substring of $u_ju_{j+1}$ with length $(k-1)H_n$,}\\
&\text{$v_x$, $v_y$ are concatenations of $w_n^{(0)},\ldots,w_n^{((n+1)^6-1)}$.}
\end{aligned}
\end{equation}

	\paragraph{Proof of \eqref{eq:finiteLower1}:} Let $x\in G_{n+1}^{(2)}$ and $\epsilon\in(0,\frac{1}{200})$. If we have for $i=0,\ldots,H_n-1$ and  $k\in[\frac{7}{(n+1)^{12}}R_n,R_n]\cap\mathbb{Z}$ that
\begin{equation}\label{eq:seperateMass}
\mu(B_{\mathcal{P}_{n_0},kH_n}(x,T,\frac{\epsilon}{2})\cap G_{n+1}^{(2,i)})\leq \frac{4}{N_n}\mu(G_{n+1}^{(2,i)}),
\end{equation}
then the first part of our lemmas follows from $N_n=a_{(R_n+1)H_{n}}(t_{n})$ and \begin{equation}\label{eq:Sn+1Mea}
		\mu(S_{n+1}^{(2)})>\frac{1}{5}\mu(X).
	\end{equation}
which is due to \eqref{eq:initialEquation}, \eqref{eq:addMass1} and \eqref{eq:addMassEst}. Thus, we only need to prove \eqref{eq:seperateMass} to complete the proof of \eqref{eq:finiteLower1}.

Suppose $x,y\in G_{n+1}^{(2,s)}$ for some $s=0,\ldots,H_n-1$ and let $v_x,v_y$ be defined as is \eqref{eq:v1v2}. We write $v_x=A_1A_2$ and $v_y=B_1B_2$, where $A_1$, $A_2$, $B_1$ and $B_2$ are substrings of $u_i$, $u_{i+1}$, $u_{j}$, and $u_{j+1}$, respectively. Suppose $l(A_1)\geq l(B_1)$ (if $l(A_1)\leq l(B_1)$, we just switch the role of $v_x$ and $v_y$). Since we have $l(A_1)+l(A_2)=l(B_1)+l(B_2)$, we have $l(A_2)\leq l(B_2)$. We decompose $A_1$ as $C_1C_2$ such that $C_1$, $C_2$ are substrings of $u_i$ and $l(C_1)=l(B_1)$. We also decompose $B_2$ into $D_2D_3$ such that $D_2$, $D_3$ are substrings of $u_{j+1}$ and $l(D_3)=l(A_2)$. Finally let $C_3=A_2$ and $D_1=B_1$. Hence, we have $v_x=C_1C_2C_3$ and $v_y=D_1D_2D_3$ such that $C_1$, $C_2$ are substrings of $u_i$; $C_3$ is a substring of $u_{i+1}$; $D_1$ is a substring of $u_j$; $D_2$, $D_3$ are substrings of $u_{j+1}$ and $l(C_p)=l(D_p)$ for $p=1,2,3$. Moreover, \eqref{eq:seperateWordsCon} guarantees that
\begin{equation}\label{eq:seperateWordsCon1}
\text{$C_1$, $C_2$, $C_3$, $D_1$, $D_2$, and $D_3$ are concatenations of $w_n^{(0)},\ldots,w_n^{((n+1)^6-1)}$.}
\end{equation}
	
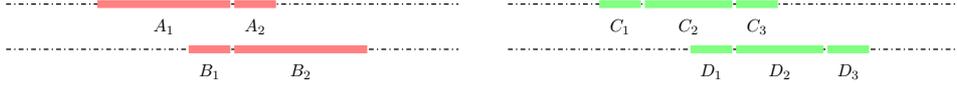
\begin{figure}[H]
		\centering
		\scalebox{0.6}
		{
			\begin{tikzpicture}[scale=5]
				\tikzstyle{vertex}=[circle,minimum size=2pt,inner sep=0pt]
				\tikzstyle{selected vertex} = [vertex, fill=red!24]
				\tikzstyle{edge1} = [draw,line width=5pt,-,red!50]
				\tikzstyle{edge2} = [draw,line width=5pt,-,green!50]
				\tikzstyle{edge3} = [draw,line width=5pt,-,blue!50]
				\tikzstyle{edge4} = [draw,line width=5pt,-,brown!50]
				
				\tikzstyle{edge} = [draw,thick,-,black]
				\node[vertex] (l01) at (-1.0,0) {};
				\node[vertex] (l02) at (-0.2,0) {};
				\node[vertex] (l03) at (0,0) {};
				\node[vertex] (l04) at (0.6,0) {};
				\node[vertex] (l05) at (1.0,0) {};
				
				\node[vertex] (l11) at (-1.0,0.2) {};
				\node[vertex] (l12) at (-0.6,0.2) {};
				\node[vertex] (l13) at (0,0.2) {};
				\node[vertex] (l14) at (0.2,0.2) {};
				\node[vertex] (l15) at (1.0,0.2) {};
				
				\node[vertex] (l21) at (1.2,0) {};
				\node[vertex] (l22) at (2,0) {};
				\node[vertex] (l23) at (2.2,0) {};
				\node[vertex] (l231) at (2.6,0) {};
				\node[vertex] (l24) at (2.8,0) {};
				\node[vertex] (l25) at (3.2,0) {};
				
				\node[vertex] (l31) at (1.2,0.2) {};
				\node[vertex] (l32) at (1.6,0.2){};
				\node[vertex] (l321) at (1.8,0.2) {};
				\node[vertex] (l33) at (2.2,0.2) {};
				\node[vertex] (l34) at (2.4,0.2) {};
				\node[vertex] (l35) at (3.2,0.2) {};
				
				\node[vertex] (w01) at (-0.1,-0.1) {$B_1$};
				\node[vertex] (w01) at (0.3,-0.1) {$B_2$};
				\node[vertex] (w01) at (-0.3,0.1) {$A_1$};
				\node[vertex] (w01) at (0.1,0.1) {$A_2$};
				\node[vertex] (w01) at (1.7,0.1) {$C_1$};
				\node[vertex] (w01) at (2.0,0.1) {$C_2$};
				\node[vertex] (w01) at (2.3,0.1) {$C_3$};
				\node[vertex] (w01) at (2.1,-0.1) {$D_1$};
				\node[vertex] (w01) at (2.4,-0.1) {$D_2$};
				\node[vertex] (w01) at (2.7,-0.1) {$D_3$};
				
				\draw[thick,dash dot] (l01)--(l05);
				\draw[edge1] (l02)--(l03);
				\draw[edge1] (l03)--(l04);
				\draw[thick,dash dot] (l11)--(l15);
				\draw[edge1] (l12)--(l13);
				\draw[edge1] (l13)--(l14);
				
				\draw[thick,dash dot] (l21)--(l25);
				\draw[edge2] (l22)--(l23);
				\draw[edge2] (l23)--(l231);
				\draw[edge2] (l231)--(l24);
				\draw[thick,dash dot] (l31)--(l35);
				\draw[edge2] (l32)--(l321);
				\draw[edge2] (l321)--(l33);
				\draw[edge2] (l33)--(l34);
				
			\end{tikzpicture}
		}.
	\end{figure}
	
	Recall $\gamma=\frac{1}{(n+1)^{12}}$ and $l(v_x)=l(v_y)\geq 6\gamma R_nH_n$. Then by Pigeonhole principle, there exists $p_0\in\{1,2,3\}$ such that $l(C_{p_0})\geq\frac{1}{3}l(v_1)$. Thus, there exists $l_0\in\mathbb{Z}^+$ such that $$l_0(\|\gamma R_n\|+1)H_n\leq l(C_{p_0})\leq (l_0+1)(\|\gamma R_n\|+1)H_n,$$
	where $\|a\|$ is the integer part of $a\in\mathbb{R}^+$. We denote the first $l_0(\|\gamma R_n\|+1)H_n$ symbols of $C_{p_0}$ and $D_{p_0}$ by $C'_{p_0}$ and $D'_{p_0}$, respectively. Suppose $|i-j|\geq2$, then by Proposition \ref{prop:prob} (2), step \ref{item:step6}, \eqref{eq:ineq4} of Lemma \ref{lem:inductionHamming} and \eqref{eq:seperateWordsCon1}, we have
	\begin{equation}\label{eq:HD1}
		\begin{aligned}
			d_{l(C_{p_0})}^H(C_{p_0},D_{p_0})&\geq \frac{l_0}{l_0+1}d_{l_0(\|\gamma R_n\|+1)H_n}^H(C'_{p_0},D'_{p_0})\\
			&\geq\frac{1}{2}\prod_{m=n_0}^{\infty}(1-\frac{600}{m^2})^5.\\
		\end{aligned}
	\end{equation}
Combining \eqref{eq:HD1} with the assumption that $n_0\geq N_1$ guarantees
	\begin{equation*}
		d_{l(C_{p_0})}^H(C_{p_0},D_{p_0})>\frac{3}{8},
	\end{equation*}
which together with $l(C_{p_0})\geq\frac{1}{3}l(v_x)$,  \eqref{eq:v1v2} and $\frac{1}{(n+1)^{12}}R_n\geq(n+1)^{20}$ implies that \begin{equation}\label{eq:secondHammingSepe}
		d_{kH_n}^H(\phi_{\mathcal{P}_{n_0},kH_n}(x),\phi_{\mathcal{P}_{n_0},kH_n}(y))\geq \frac{1}{16}.
	\end{equation}
	
Since each string of length $kH_n$ has the same measure, we obtain from \eqref{eq:secondHammingSepe} and the assumption $|i-j|\geq2$ that for every $z\in G_{n+1}^{(2,s)}$ and $\epsilon\in(0,\frac{1}{100})$, \begin{equation}\label{eq:middleSeperate}
\mu(B_{\mathcal{P}_{n_0},kH_n}(z,T,\epsilon)\cap G_{n+1}^{(2,s)})\leq\frac{4}{N_n}\mu(G_{n+1}^{(2,s)}).
\end{equation}
For any $x\in G_{n+1}^{(2)}$, we have either $B_{\mathcal{P}_{n_0},kH_n}(x,T,\frac{\epsilon}{2})\cap G_{n+1}^{(2,s)}=\emptyset$ or there exists $z\in B_{\mathcal{P}_{n_0},kH_n}(x,T,\frac{\epsilon}{2})\cap G_{n+1}^{(2,s)}$. The triangle inequality gives that $$\left(B_{\mathcal{P}_{n_0},kH_n}(x,T,\frac{\epsilon}{2})\cap G_{n+1}^{(2,s)}\right)\subset B_{\mathcal{P}_{n_0},kH_n}(z,T,\epsilon)\cap G_{n+1}^{(2,s)}.$$
Thus, \eqref{eq:middleSeperate} implies that  \eqref{eq:seperateMass}.

	\paragraph{Proof of \eqref{eq:finiteLower2}:}
	If $k\in[\tau_n,6\gamma R_n]$, denote
	\begin{equation*}
		\begin{aligned}
			\Omega&=\{v:\text{$v$ is a substring of $u_iu_{i+1}$ with length $kH_n$ for some $i\in\{1,\ldots,N_n-1\}$}\},\\
			\Omega_G&=\{v\in\Omega:\text{$v$ is a substring of $u_i$ with length $kH_n$ for some $i\in\{1,\ldots,N_n\}$}\},\\
\Omega'_G&=\{v\in\Omega_G:\text{$v=u_i\upharpoonright[t,t+kH_n]$ for $0\leq t\leq(1-2\gamma)R_nH_n$.}\}
		\end{aligned}
	\end{equation*}

	By direct computation, we have
	\begin{equation*}
		\begin{aligned}
			\operatorname{Card}(\Omega)=N_nR_nH_n-kH_n, \ \ \operatorname{Card}(\Omega_G)=N_n(R_nH_n-kH_n).
		\end{aligned}
	\end{equation*}
	Recall $k\leq6\gamma R_n$ and, thus, we have
	\begin{equation*}
		\frac{\operatorname{Card}(\Omega_G)}{\operatorname{Card}(\Omega)}=\frac{N_n(R_nH_n-kH_n)}{N_nR_nH_n-kH_n}\geq\frac{R_n-k}{R_n}\geq1-6\gamma.
	\end{equation*}
which together with $\operatorname{Card}(\Omega'_G)\geq(1-3\gamma)\operatorname{Card}(\Omega_G)$ implies that
\begin{equation}\label{eq:goodStrings}
\operatorname{Card}(\Omega'_G)\geq(1-3\gamma)(1-6\gamma)\operatorname{Card}(\Omega)\geq(1-9\gamma)\operatorname{Card}(\Omega).
\end{equation}
Recalling that each substring of length $kH_n$ has the same measure, the estimates in \eqref{eq:Sn+1Mea} and \eqref{eq:goodStrings} imply that
	\begin{equation}\label{eq:goodStrings1}
		\begin{aligned} \mu(\phi_{\mathcal{P}_{n_0},kH_n}^{-1}(\Omega'_G))&\geq(1-9\gamma)\mu(\phi_{\mathcal{P}_{n_0},kH_n}^{-1}(\Omega))\geq(1-9\gamma)\mu(S_{n+1}^{(2)})\\			&>\frac{(1-9\gamma)}{5}\mu(X).
		\end{aligned}
	\end{equation}

If for any $x\in G_{n+1}^{(2)}$ and $s=0,\ldots,H_n-1$, we have
\begin{equation}\label{eq:seperateMass1}
\mu\left(B_{\mathcal{P}_{n_0},kH_n}(x,T,\frac{\epsilon}{4})\cap \phi_{\mathcal{P}_{n_0},kH_n}^{-1}(\Omega'_G)\cap G_{n+1}^{(2,s)}\right)\leq \frac{5}{a_{(k+1)H_{n}}(t_{n})}\mu(\phi_{\mathcal{P}_{n_0},kH_n}^{-1}(\Omega'_G)\cap G_{n+1}^{(2,s)}),
\end{equation}
then the second part of our lemma follows from  \eqref{eq:goodStrings1}. Thus, we only need to prove \eqref{eq:seperateMass1} to complete the proof of \eqref{eq:finiteLower2}.
	
Suppose $x,y\in G_{n+1}^{(2,s)}$ for some $s=0,\ldots,H_n-1$ and let $v_x, v_y$ be defined as in \eqref{eq:v1v2}. Recall that at step \ref{item:step6} we apply Proposition \ref{prop:prob} with basic symbols $\{w_n^{(0)},\ldots,w_n^{((n+1)^6-1)}\}$. Hence we get, that if $v_x$ and $v_y$ are $\epsilon$ Hamming far away with respect to alphabet $\{w_n^{(0)},\ldots,w_n^{((n+1)^6-1)}\}$, then we have the following estimates by \eqref{eq:ineq4} of Lemma \ref{lem:inductionHamming}:
	\begin{equation*}
		d_{(k-1)H_n}^H(v_x,v_y)\geq \epsilon\prod_{m=n_0}^{\infty}(1-\frac{600}{m^2})^5,
	\end{equation*}
which together with the assumption $n_0\geq N_1$ and $\tau_n\geq n$ guarantees that
	\begin{equation}\label{eq:HammingFar}
		d_{kH_n}^H(\phi_{\mathcal{P}_{n_0},kH_n}(x),\phi_{\mathcal{P}_{n_0},kH_n}(y))\geq \frac{1}{2}\epsilon.
	\end{equation}

Combining Proposition \ref{prop:prob} (3), (4), the choice of $\tau_n$ and \eqref{eq:HammingFar}, we have for $z\in G_{n+1}^{(2,s)}$ that:
	\begin{equation}\label{eq:middleSeperate1}
		\begin{aligned}
		&\mu\left(B_{\mathcal{P}_{n_0},kH_n}(z,T,\frac{\epsilon}{2})\cap \phi_{\mathcal{P}_{n_0},kH_n}^{-1}(\Omega'_G)\cap G_{n+1}^{(2,s)}\right)\\
		\leq & \left(\frac{3}{a_{(k+1)H_{n}}(t_{n})}+\frac{2}{N_n}\right)\cdot \mu(\phi_{\mathcal{P}_{n_0},kH_n}^{-1}(\Omega'_G)\cap G_{n+1}^{(2,s)}),
		\end{aligned}
	\end{equation}
where $\frac{3}{a_{(k+1)H_{n}}(t_{n})}$ comes from Proposition \ref{prop:prob} and $\frac{2}{N_n}$ comes from the substrings, which lie in the same $u_i$ as $v_z$.

Recall for any $x\in G_{n+1}^{(2)}$ that we have either $B_{\mathcal{P}_{n_0},kH_n}(x,T,\frac{\epsilon}{4})\cap G_{n+1}^{(2,s)}=\emptyset$ or there exists $z\in B_{\mathcal{P}_{n_0},kH_n}(x,T,\frac{\epsilon}{4})\cap G_{n+1}^{(2,s)}$. The triangle inequality gives that $$\left(B_{\mathcal{P}_{n_0},kH_n}(x,T,\frac{\epsilon}{4})\cap \phi_{\mathcal{P}_{n_0},kH_n}^{-1}(\Omega'_G)\cap G_{n+1}^{(2,s)}\right)\subset B_{\mathcal{P}_{n_0},kH_n}(z,T,\frac{\epsilon}{2})\cap\phi_{\mathcal{P}_{n_0},kH_n}^{-1}(\Omega'_G)\cap  G_{n+1}^{(2,s)}.$$
Thus, \eqref{eq:middleSeperate1}, $N_n=a_{(R_n+1)H_{n}}(t_{n})$ and $k+1\leq R_n$ imply \eqref{eq:seperateMass1}.
	
\end{proof}

Now we will show that the lower slow entropy of $T$ with respect to the scaling function $a_n(t)$ with properties \eqref{eq:scalingFunction} is positive:
\begin{proposition}
	For $a_n(t)$ defined in \eqref{eq:scalingFunction}  and $n_0\in\mathbb{N}$ fixed, we have
	$$\lent^{\mu_X}_{a_n(t)}(T,\mathcal{P}_{n_0})=+\infty,$$
	where $\mathcal{P}_{n_0}$ is defined in \eqref{eq:finiteRankPartition}.
\end{proposition}
\begin{proof}
	Since $R_nH_n\geq\tau_{n+1}H_{n+1}$ for any $n\in\mathbb{Z}^+$ (see \eqref{eq:heightHn23} for more details), for any $m\in\mathbb{Z}^+$ there exist $k_m,n_m\in\mathbb{Z}^+$ such that
	\begin{equation}\label{eq:boundsOnM}
		\begin{aligned}
			&k_mH_{n_m}\leq m\leq(k_m+1)H_{n_m},\\
			&\tau_{n_m}\leq k_m\leq R_{n_m}-1.
		\end{aligned}
	\end{equation}
	As a result of \eqref{eq:boundsOnM} we see that if we have $$d_{m}^H(\phi_{\mathcal{P}_{n_0},m}(x),\phi_{\mathcal{P}_{n_0},m}(y))<\epsilon,$$
	then we get
	\begin{equation}\label{eq:stairControl}
		d_{k_mH_{n_m}}^H(\phi_{\mathcal{P}_{n_0},k_mH_{n_m}}(x),\phi_{\mathcal{P}_{n_0},k_mH_{n_m}}(y))<\frac{k_m+1}{k_m}\epsilon.
	\end{equation}
	
	Then by Proposition \ref{prop:towerSequence} and \eqref{eq:stairControl}, we obtain
	\begin{equation}\label{eq:secondInfinity1}
		\frac{S_{\mathcal{P}_{n_0}}(T,m,\epsilon)}{ a_m(t)}\geq \frac{ S_{\mathcal{P}_{n_0}}(T,k_mH_{n},\frac{k_m+1}{k_m}\epsilon)}{ a_{(k_m+1)H_{n}}(t)}\geq\frac{\frac{1}{6}a_{(k_m+1)H_{n}}(t_{n})}{a_{(k_m+1)H_{n}}(t)}.
	\end{equation}
	Recall that for any fixed $t>0$ there exists $N_0$ such that for every $n>N_0$ we have $t_{n}>t$. Moreover, recall that $a_n(t_1)\leq a_n(t_2)$ for $t_2>t_1>0$. Then for every $n>N_0$, we have
	\begin{equation}\label{eq:secondInfinity2}
		a_{(k_m+1)H_{n}}(t)\leq a_{(k_m+1)H_{n}}(t_{n}).
	\end{equation}
	
	Combining \eqref{eq:secondInfinity1}, \eqref{eq:secondInfinity2} and arbitrariness of $t$, we complete the proof of the proposition.
\end{proof}

Since $\mathcal{P}_{n_0}$ converges to the decomposition into points as $n_0\to+\infty$, we conclude $\lent_{a_n(t)}^{\mu_X}(T)=+\infty$ with the aid of Proposition \ref{prop:generatingSequence}.  Recalling that our system has rank at most $2$, our estimates and the bounds on the polynomial lower slow entropy in Ferenczi \cite[Proposition 5]{Fe} and Kanigowski \cite[Proposition 1.3]{Kanigowski} show that the rank of our system is exactly $2$. This completes the proof of Theorem \ref{thm:slowentropyFiniterank}.

\section{Appendix}
\subsection{Proof of Proposition \ref{prop:ergodic}}\label{sec:ergodic}
We first show that $T$ preserves Lebesgue measure $\mu$. By the construction of $T$, we know that there exist countable disjoint intervals $I_i\subset X$ for $i\in\mathcal{I}$ such that $\cup_{\mathcal{I}}I_i=X$ and $T$ is measure preserving from $I_i$ to $T(I_i)$. Let $A\subset X$ be a measurable subset, then there exists a countable decomposition of $A$ into $A_i=A\cap I_i$ for $i\in\mathcal{I}$ such that $A=\cup_{i\in\mathcal{I}}A_i$. Since $T$ is a measure preserving map on each interval $I_i$, we have $\mu(A_i)=\mu(T^{-1}(A_i))=\mu(T(A_i))$, which gives that $\mu(A)=\mu(T^{-1}(A))=\mu(T(A))$.

Now suppose that $A$ is an $T$-invariant set with positive measure, i.e. $A=T^{-1}(A)$ and $\mu(A)>0$. Based on our construction, there exists $N_0\in\mathbb{N}$ such that for any $n\geq N_0$, we have
\begin{equation}
\mu(A\cap(S_n^{(1)}\cup S_n^{(2)}))>\frac{\mu(A)}{2}>0.
\end{equation}
Moreover, it is worth to point out that our construction also gives us that for any $n\leq m$:
$$S_n^{(1)}\cup S_n^{(2)}\subset S_m^{(1)}\cup S_m^{(2)},$$
thus we have
\begin{equation}
A\cap(S_n^{(1)}\cup S_n^{(2)})\subset A\cap(S_m^{(1)}\cup S_m^{(2)}).
\end{equation}

Denote $A'=A\cap(S_{N_0}^{(1)}\cup S_{N_0}^{(2)})$ and let $x\in A'$ with Lebesgue density equal to $1$. This implies that for any $\epsilon\in(0,10^{-8})$ there exists $\delta>0$ such that if $I$ is an interval that contains $x$ and $\mu(I)<\delta$, we have $\mu(A'\cap I)>(1-\epsilon)\mu(I)$.

Let $N_1\in\mathbb{N}$ such that for any $n\geq N_1$, $\mu(B_{n}^{(1)}),\mu(B_n^{(2)})<\delta$ and $\mu(X\setminus(S_n^{(1)}\cup S_n^{(2)}))<\epsilon\mu(X)$. By Lebesgue density Theorem and the definition of $A'$, we obtain that for any $n\geq N_1$, there exists at least one level $I_0$ from $S_n^{(1)}$ or $S_n^{(2)}$ such that \begin{equation}\label{eq:firstInterval}
\mu(A'\cap I_0)\geq(1-\epsilon)\mu(I_0).
\end{equation}
By the stack's definition and $A=\cup_{i=-\infty}^{+\infty}T^i(A)\supset\cup_{i=-\infty}^{+\infty}T^i(A')$, we know that \eqref{eq:firstInterval} implies
\begin{equation}\label{eq:goodRegion2}
\mu(A\cap S_n^{(1)})\geq (1-\epsilon)\mu(S_n^{(1)}).
\end{equation}
or
\begin{equation}\label{eq:goodRegion3}
\mu(A\cap S_n^{(2)})\geq (1-\epsilon)\mu(S_n^{(2)}).
\end{equation}

If \eqref{eq:goodRegion2} holds, then denote the levels of $S_{n+1}^{(1)}$ from $S_n^{(1)}$ as $D_i$ and the levels of $S_{n+1}^{(2)}$ from $S_n^{(1)}$ as $E_j$. Since $\frac{\mu(\cup D_i)}{\mu(\cup E_j)}\in(\frac{1}{400},\frac{1}{20})$ (see \eqref{eq:initialEquation} and steps \ref{item:step2},\ref{item:step5} in our construction) and $(\cup D_i)\cup(\cup E_j)=S_n^{(1)}$, we obtain following inequalities from \eqref{eq:goodRegion2}:
\begin{equation}\label{eq:markovPreparation}
\begin{aligned}
\mu\left(A\cap(\cup D_i)\right)\geq(1-500\epsilon)\mu(\cup D_i),\ \ \mu\left(A\cap(\cup E_j)\right)\geq(1-500\epsilon)\mu(\cup E_j).
\end{aligned}
\end{equation}

Denote $k_{bad}^{(1)}=\operatorname{Card}\{D_i:\mu(A\cap D_i)\leq (1-10^5\epsilon)\mu(D_i)\}$ and $k_{bad}^{(2)}=\operatorname{Card}\{E_j:\mu(A\cap E_j)\leq (1-10^5\epsilon)\mu(E_j)\}$, we have the following inequality by \eqref{eq:markovPreparation}:
\begin{equation}
\begin{aligned}
k_{bad}^{(1)}\cdot10^5\epsilon\cdot\mu(D_i)\leq500\epsilon\mu(\cup D_i), \ \ k_{bad}^{(2)}\cdot10^5\epsilon\cdot\mu(E_j)\leq500\epsilon\mu(\cup E_j),
\end{aligned}
\end{equation}
which implies that
\begin{equation}\label{eq:markovResults}
\begin{aligned}
k_{bad}^{(1)}\leq \frac{1}{200}\operatorname{Card}(\{D_i\}),\ \
k_{bad}^{(2)}\leq \frac{1}{200}\operatorname{Card}(\{E_j\}).
\end{aligned}
\end{equation}
Then \eqref{eq:markovResults} implies that there exists $i_0$ and $j_0$ such that
\begin{equation}\label{eq:goodLevels}
\begin{aligned}
\mu(A\cap D_{i_0})\geq(1-10^5\epsilon)\mu(D_{i_0}),\ \
\mu(A\cap E_{j_0})\geq(1-10^5\epsilon)\mu(E_{j_0}).
\end{aligned}
\end{equation}
Recall that $D_{i_0}$ and $E_{j_0}$ are levels of $S_{n+1}^{(1)}$ and $S_{n+1}^{(2)}$, respectively. Together with $A=T^i(A)$ for $i\in\mathbb{Z}$ and the definition of towers, we obtain that
\begin{equation}\label{eq:ergodicInequality}
\begin{aligned}
\mu(A)&\geq\mu(A\cap(S_{n+1}^{(1)}\cup S_{n+1}^{(2)}))\geq(1-10^5\epsilon)(\mu(S_{n+1}^{(1)})+\mu(S_{n+1}^{(2)}))\\
&\geq (1-\epsilon)(1-10^5\epsilon)\mu(X).
\end{aligned}
\end{equation}

Since $\epsilon\in(0,10^{-8})$ is arbitrary, we obtain that $\mu(A)=\mu(X)$. If \eqref{eq:goodRegion3} holds, we will consider the towers $S_{n+2}^{(1)}$ and $S_{n+2}^{(2)}$ instead of $S_{n+1}^{(1)}$ and $S_{n+1}^{(2)}$, and then we will have \eqref{eq:ergodicInequality} again. Combining these two cases, we obtain that $T$ is ergodic with respect to $\mu_X$.

\end{document}